\documentclass{amsart}

\usepackage{enumerate}
\usepackage{amssymb}
\usepackage{graphicx}

\title{The nonabelian product modulo sum}
\author{Samuel M. Corson}

\theoremstyle{definition}\newtheorem{theorem}{Theorem}
\theoremstyle{definition}
\theoremstyle{definition}
\theoremstyle{definition}
\theoremstyle{definition}

\theoremstyle{definition}

\newtheorem*{hugetheorem}{Main Theorem}
\numberwithin{theorem}{section}
\theoremstyle{definition}
\theoremstyle{definition}\newtheorem{proposition}[theorem]{Proposition}
\theoremstyle{definition}\newtheorem{definition}[theorem]{Definition}
\theoremstyle{definition}
\theoremstyle{definition}\newtheorem{example}[theorem]{Example}
\theoremstyle{definition}
\theoremstyle{definition}
\theoremstyle{definition}\newtheorem{lemma}[theorem]{Lemma}
\theoremstyle{definition}
\theoremstyle{definition}
\theoremstyle{definition}
\theoremstyle{definition}\newtheorem{obs}[theorem]{Observation}
\theoremstyle{definition}\newtheorem{definitions}[theorem]{Definitions}

\newcommand{\Red}{\operatorname{Red}}

\newcommand{\Let}{\operatorname{Let}}

\newcommand{\im}{\operatorname{im}}
\newcommand{\Sub}{\operatorname{Sub}}
\newcommand{\Fine}{\operatorname{Fine}}
\newcommand{\set}{\operatorname{set}}

\newcommand{\W}{\mathcal{W}}
\newcommand{\Close}{\operatorname{Close}}
\newcommand{\coi}{\operatorname{coi}}
\newcommand{\dom}{\operatorname{dom}}

\def\pmc#1{\setbox0=\hbox{#1}
    \kern-.1em\copy0\kern-\wd0
    \kern.1em\copy0\kern-\wd0}

\DeclareMathOperator{\topprod}{\circledast}

\begin{document}

\address{Matematika Saila, UPV/EHU, Sarriena S/N, 48940, Leioa - Bizkaia, Spain.}
\email{sammyc973@gmail.com}

\keywords{algebraically compact group, infinite word, topologist's product}
\subjclass[2010]{Primary 03E75, 20A15, 55Q52; Secondary 20F10, 20F34}
\thanks{This work was supported by the Additional Funding Programme for Mathematical Sciences, delivered by EPSRC (EP/V521917/1) and the Heilbronn Institute for Mathematical Research}

\begin{abstract}  It is shown that if $\{H_n\}_{n \in \omega}$ is a sequence of groups without involutions, with $1 < |H_n| \leq 2^{\aleph_0}$, then the topologist's product modulo the finite words is (up to isomorphism) independent of the choice of sequence.  This contrasts with the abelian setting: if $\{A_n\}_{n \in \omega}$ is a sequence of countably infinite torsion-free abelian groups, then the isomorphism class of the product modulo sum $\prod_{n \in \omega} A_n/\bigoplus_{n \in \omega} A_n$ is dependent on the sequence.

\end{abstract}

\maketitle

\begin{section}{Introduction}

A classical theorem of Balcerzyk states that if $\{A_n\}_{n \in \omega}$ is a sequence of abelian groups, indexed by the set $\omega$ of natural numbers, then the quotient $$\prod_{n \in \omega} A_n/\bigoplus_{n\in \omega} A_n$$ is algebraically compact \cite[Corollary 6.1.12]{Fu}.  It is unsurprising to learn that the isomorphism class of this quotient can alter as one varies the collection $\{A_n\}_{n\in \omega}$.  For example, if each $A_n$ contains an involoution (i.e. an element of order $2$) then so does the quotient.  If all $A_n$ are torsion-free then so is the quotient.  But even when all $A_n$ are torsion-free (or even more particularly of rank $1$) one can obtain various non-isomorphic groups as quotient.  If $A_n \simeq \mathbb{Q}$ for all $n$, the quotient is abstractly isomorphic to $\mathbb{R}$.  In case $A_n \simeq \mathbb{Z}$ for all $n$, the quotient is abstractly isomorphic to $\mathbb{R} \oplus J$ where $J$ is reduced and uncountable \cite[Exercise 6.3.9]{Fu}.  We show that in the analogous nonabelian context there is less variation.

The natural nonabelian replacement for the product operation is the \emph{topologist's product} \cite{E}.  The topologist's product $\topprod_{n \in \omega} H_n$ is an infinitary extension of the free product, the essential difference being that infinite words (not just finite ones) are considered.  While $\bigoplus_{n \in \omega} A_n$ is the set of elements of finite support in $\prod_{n \in \omega}A_n$, the subgroup of finite words in $\topprod_{n \in \omega} H_n$ is precisely the free product $*_{n \in \omega} H_n$.   As the free product $*_{n \in \omega} H_n$ is not a normal subgroup we will be considering the quotient $\topprod_{n \in \omega} H_n/\langle\langle *_{n \in \omega}H_n\rangle\rangle$ where $\langle\langle \cdot \rangle\rangle$ denotes the normal closure.  For economy we denote this quotient by $\mathcal{A}(\{H_n\}_{n\in \omega})$, or simply $\mathcal{A}$ in case all $H_n$ are infinite cyclic.  The upshot is the following.

\begin{hugetheorem}\label{bigisomorphism}  If $\{H_n\}_{n \in \omega}$ is a sequence of groups without involutions such that $1 < |H_n| \leq 2^{\aleph_0}$ for all $n$ then $\mathcal{A}(\{H_n\}_{n \in \omega}) \simeq \mathcal{A}$.
\end{hugetheorem}

For example, some of the $H_n$ could be the group of cardinality $3$, or the additive group on the reals, or Thompson's group $F$.  Despite the generality of the theorem, it is reasonably sharp.  One can also allow finitely many of the $H_n$ to have cardinality greater than $2^{\aleph_0}$ and still obtain the same conclusion since the operation $\mathcal{A}(\cdot)$ does not change isomorphism type if finitely many groups in the sequence $\{H_n\}_{n \in \omega}$ are deleted (see Lemma \ref{nicefactsaboutarch}).  However if $|H_n| \leq 2^{\aleph_0}$ is violated at infinitely many $n$ then $|\mathcal{A}(\{H_n\}_{n \in \omega})| > 2^{\aleph_0}$  \cite[Theorem 9]{CHM} and so one obtains a group of cardinality strictly greater than $2^{\aleph_0} = |\mathcal{A}|$.

The isomorphism produced is nonconstructive; a transfinite induction of length $2^{\aleph_0}$ involving various arbitrary choices along the way is utilized (a back-and-forth argument over larger and larger subgroups).  Possibly one can allow elements of order $2$ in the $H_n$, but the attack would require serious modification (see Example \ref{nastyword}).  We note that such elements also create difficulties in classical word combinatorics (e.g. the Burnside problem, see  \cite[page 139]{Ol}).

The theorem has as corollaries the first claim of \cite[Theorem A]{CHM} (the proof was incorrect in that paper) and an affirmative answer to \cite[Question 2]{CHM}.  The reason that their proof was incorrect is that their homomorphism is not injective (see \cite[Corollary 16]{CHM}).  The referee has pointed out that their approach does in fact give an isomorphism between the subgroups of $\mathcal{A}$ and $\mathcal{A}(\{H_n\}_{n \in \omega})$ consisting of those words the Dedekind completion of whose domain is scattered.  Our approach is a careful transfinite induction over the continuum, extending isomorphisms between larger and larger subgroups via a back-and-forth argument.  The key is to analyze words of infinite length and show that if fewer than a continuum of certain choices have been made in producing the partial isomorphisms then it is always possible to make one more choice in a way which is consistent with the previous ones.  This will in particular ensure that no unforseen cancellations will arise.

The group $\mathcal{A}$ is isomorphic to the fundamental group of the harmonic archipelago mentioned in \cite{BS}, which is now known to be isomorphic to the fundamental group of the Griffiths space constructed in \cite{G} (see \cite{Cors}).  Another consequence of our main result is that, for $\{H_n\}_{n\in \omega}$ as in the theorem, $\mathcal{A}(\{H_n\}_{n \in \omega})$ is locally free and every countable locally free group embeds into it (see \cite[Theorem 2]{Ho}).

We give the layout of the paper.  In Section \ref{topprod} we establish some background results and definitions for the topologist's product.  In Sections \ref{Howtoisomorphize}, \ref{uncomplicated}, and \ref{Qconcatenation} we adapt machinery from \cite{Cors} to construct partial isomorphisms with larger and larger domain and codomain.  In Section \ref{Finally} the proof of the main theorem os finished.

\end{section}

\begin{section}{The topologist's product}\label{topprod}

In this section we concern ourselves with an exposition of the topologist's product (also known as the $\sigma$-product \cite{E}).  We'll begin by establishing some notation for total orders.  Given two totally ordered sets $\Lambda_0$, $\Lambda_1$ we will write $\Lambda_0 \equiv \Lambda_1$ if there exists an order isomorphism between them.  The \emph{concatenation} of totally ordered sets $\Lambda_0$ and $\Lambda_1$ will be written $\Lambda_0\Lambda_1$ and is the disjoint union $\Lambda_0 \sqcup \Lambda_1$ ordered by the unique order which extends that of $\Lambda_0$ and that of $\Lambda_1$ and places elements of $\Lambda_0$ below those of $\Lambda_1$.  More generally if $\{\Lambda_{\lambda}\}_{\lambda \in \Lambda}$ is a collection of totally ordered sets which is indexed by a totally ordered set we let $\prod_{\lambda \in \Lambda} \Lambda_{\lambda}$ denote the disjoint union $\bigsqcup_{\lambda \in \Lambda} \Lambda_{\lambda}$ ordered under the unique order which extends all of the orders $\Lambda_{\lambda}$ and places elements in $\Lambda_{\lambda}$ below those in $\Lambda_{\lambda'}$ whenever $\lambda < \lambda'$.  We say an interval $I \subseteq \Lambda$ is \emph{initial} provided every element in $I$ is below all elements in $\Lambda \setminus I$, and a \emph{terminal} interval is defined analogously.  For a totally ordered set $\Lambda$ we will let $\Lambda^{-1}$ denote the set $\Lambda$ under the reverse order.

Many of the word concepts which will be defined here will also apply to a collection of groups which is uncountable (see \cite{E}), but in our paper we are only concerned with a collection indexed by the set $\omega$ of natural numbers.  Let $\{G_n\}_{n\in \omega}$ be a collection of groups, and we will regard them as having the same identity element $1$ and not having any other elements in common: $G_{n_0} \cap G_{n_1} = \{1\}$ for $n_0 \neq n_1$.  A \emph{word} is a function $W: \overline{W} \rightarrow \bigcup_{n \in \omega} G_n$ such that the domain $\overline{W}$ is a totally ordered set and for each $n$ the set $\{i \in \overline{W} : W(i) \in G_n\}$ is finite (it follows immediately that $\overline{W}$ is countable).  We point out that $\overline{W}$ could be of any countable order type, including $\mathbb{Q}$.

We shall sometimes refer to the set $\bigcup_{n \in \omega} G_n$ as the set of \emph{letters}.  We will use some of the notation defined above for totally ordered sets in this setting as well.  We will consider two words $W_0$ and $W_1$ to be the same, writing $W_0 \equiv W_1$, provided there exists an order isomorphism $\iota: \overline{W_0} \rightarrow \overline{W_1}$ such that $W_0(i) = W_1(\iota(i))$ for all $i \in \overline{W_0}$.  Let $\W(\{G_n\}_{n \in \omega})$ denote the set of words, regarding two words as being the same if they are $\equiv$ equivalent.  

Given two words $W_0, W_1 \in \W(\{G_n\}_{n \in \omega})$ we form the \emph{concatenation}, denoted $W_0W_1$, by endowing $W_0W_1$ with domain which is the totally ordered set $\overline{W_0}$ $\overline{W_1}$ and letting $(W_0W_1)(i) = W_j(i)$ where $i \in \overline{W_j}$.  More generally, given a collection $\{W_{\lambda}\}_{\lambda \in \Lambda}$ of words indexed by a totally ordered set $\Lambda$ we form a function $\prod_{\lambda \in \Lambda}  W_{\lambda}$ whose domain is the concatenation $\prod_{\lambda \in \Lambda} \overline{W_{\lambda}}$ and letting $(\prod_{\lambda \in \Lambda} W_{\lambda})(i) = W_{\lambda'}(i)$ where $i \in \overline{W_{\lambda'}}$.  This function $\prod_{\lambda \in \Lambda} W_{\lambda}$ might or might not be a word.  Given a word $W$ we form the \emph{inverse}, $W^{-1}$, by letting $W^{-1}$ have domain equal to $\overline{W}$ under the reverse order, $\overline{W}^{-1}$, and letting $W^{-1}(i) = (W(i))^{-1}$.

For a word $W$ and $N \in \omega$ we define $p_N(W)$ to be the finite word given by the restriction $p_N(W) = W \upharpoonright \{i \in \overline{W} : W(i) \in \bigcup_{n = 0}^N G_n\}$.  Define an equivalence relation $\sim$ by letting $W_0 \sim W_1$ if for every $N \in \omega$ the words $p_N(W_0)$ and $p_N(W_1)$ are equal as elements in the free product $*_{n = 0}^N G_n$.  The group $\topprod_{n \in \omega} G_n$ is the set of words modulo $\sim$, the binary operation is given by concatenation $(W_0/\sim)(W_1/\sim) = (W_0W_1)/\sim$, inverses are given by $(W/\sim)^{-1} = W^{-1}/\sim$ and the identity element is $E/\sim$ where $E$ denotes the empty word.  The free product $*_{n \in \omega} G_n$ is a subgroup of $\topprod_{n\in \omega}G_n$ consisting of those equivalence classes which include a finite word.  Each word operation $p_N$ defines a homomorphic retraction from $\topprod_{n \in \omega} G_n$ to $*_{n = 0}^N G_n$.  When each of the groups $G_n$ is isomorphic to the integers, the group $\topprod_{n \in \omega} G_n$ is isomorphic to the fundamental group of the shrinking wedge of circles.

We'll say $W_1$ is a \emph{subword} of $W$ provided we can write $W \equiv W_0W_1W_2$ for some words $W_0$ and $W_2$.  A subword is \emph{initial} if in the writing we can have $W_0 \equiv E$ and is \emph{terminal} if we can have $W_2 \equiv E$.  A word $W$ is \emph{reduced} if $W \equiv W_0W_1W_2$ and $W_1 \sim E$ implies $W_1 \equiv E$, and for any neighboring $i_0, i_1 \in \overline{W}$ the letters $W(i_0)$ and $W(i_1)$ lie in distinct $G_n$.  Furthermore $W$ is \emph{quasi-reduced} if $W \equiv W_0W_1W_2$ and $W_1 \sim E$ implies that $\im(W_1) \subseteq G_n$ for some $n \in \omega$ and there exists $i\in \overline{W}$, with either $i = \max(W_0)$ or $i = \min(W_2)$, and $W(i) \in G_n \setminus \{1\}$.  Quasi-reduced means that one obtains a reduced word by multiplying all neighboring letters which belong to a common $G_n$.  For the following, see Theorem 1.4 and Corollaries 1.5 and 1.7 of \cite{E}.

\begin{lemma}\label{Eda}  Each $\sim$ class includes a reduced word, and this word is unique up to $\equiv$.  If $W_0$ and $W_1$ are reduced and $W_0W_1 \sim E$ then $W_1 \equiv W_0^{-1}$.  If $W$ and $W'$ are reduced then there exist reduced words $W_0$, $W_1$, $W_0'$, and $W_1'$ such that

\begin{enumerate}
\item $W \equiv W_0W_1$;

\item $W' \equiv W_0'W_1'$;

\item $W_0' \equiv (W_1)^{-1}$; and

\item $W_0W_1'$ is quasi-reduced.
\end{enumerate}

\noindent Given a word $W$ we will let $\Red(W)$ denote the reduced word such that $W \sim \Red(W)$.
\end{lemma}

\begin{obs}\label{howtomultuply}  In the notation of Lemma \ref{Eda}, since the words $W_0$ and $W_1'$ are reduced and $W_0W_1'$ is quasi-reduced, there exists an initial interval $I \subseteq \overline{W_0}$ and a terminal interval $I' \subseteq \overline{W_1'}$ such that either

\begin{enumerate}[(a)]

\item $|\overline{W_0} \setminus I| = 0 = |\overline{W_1'} \setminus I'|$ (in case $W_0W_1'$ is reduced); or

\item $|\overline{W_0} \setminus I| = 1 = |\overline{W_1'} \setminus I'|$ (in case $W_0W_1'$ is not reduced).
\end{enumerate}

\noindent In case (b) one obtains $\Red(W_0W_1')$ by performing the multiplication $$W_0(\max(\overline{W_0}))W_1'(\min(\overline{W_1'}))$$ in the common group where they lie.
\end{obs}

Let $\Red(\{G_n\}_{n \in \omega})$ denote the set of reduced words associated with the sequence $\{G_n\}_{n \in \omega}$.  We see by the above that $\topprod_{n \in \omega} G_n$ is isomorphic to $\Red(\{G_n\}_{n \in \omega})$ via the function $\Red(\cdot)$, the binary operation in $\Red(\{G_n\}_{n \in \omega})$ is given by $W_0 \circledcirc W_1 = \Red(W_0W_1)$, and the computation of $\Red(W_0W_1)$ is straightforward.  Since the combinatorics of  reduced words are so clean, we generally work with $\Red(\{G_n\}_{n \in \omega})$ instead of the isomorphic group $\topprod_{n \in \omega} G_n$.  Under this isomorphism, the free product $*_{n \in \omega}G_n$ is identified with the freely reduced finite words.

\begin{definition}  Given a letter $g \in \bigcup_{n \in \omega} G_n \setminus \{1\}$ we let $d(g) = m$ where $g\in G_m \setminus \{1\}$ and more generally given a word $W \in \W(\{G_n\}_{n \in \omega})$ which is not constantly $1$ we let $d(W)=
\min\{d(W(i)): i \in \overline{W}, W(i) \neq 1\}$.  Notice that if $W \equiv g$ is of length $1$ with $g \in G_m \setminus \{1\}$ then $d(W) = d(g)$.
\end{definition}

\begin{definitions}  We say $W \in \W(\{G_n\}_{n \in \omega})$ is \emph{proper} if $W(i) \neq 1$ for all $i \in \overline{W}$.  For a proper word $W$, a finite ordered list $C = (i_0; \ldots; i_k)$ is a \emph{reduction component on $W$} if $i_0, \ldots, i_k \in \overline{W}$ with $i_0 < i_1 < \cdots < i_k$ and there exists $M_C \in \omega$ with $W(i_0), \ldots, W(i_k) \in G_{M_C}$, and the list $C$ has at least two elements.  We further abuse notation and let $d(C) = M_C$.  If $C = (i_0; \ldots; i_k)$ is a reduction component on word $W$ we let $\pi(W, C)$ be the word of length $1$ obtained by taking the product $W(i_0)\cdots W(i_k)$ in $G_{d(C)}$ in case this product is not identity, and if the product is identity we let $\pi(W, C)$ be the empty word $E$.  We also let $\set(C) = \{i_0, \ldots, i_k\}$ be the set of elements appearing in the list $C$.
\end{definitions}

\begin{definition}\label{redschdef}  Given a word $W \in \W(\{G_n\}_{n\in \omega})$ a collection $\mathcal{S}$ of reduction components is a \emph{reduction scheme on $W$} if

\begin{itemize}

\item for distinct reduction components $C_0, C_1 \in \mathcal{S}$ we have $\set(C_0) \cap \set(C_1) = \emptyset$; and

\item for $C = (i_0; \ldots; i_k) \in \mathcal{S}$, $0 \leq j < k$, and $i$ in the open interval $(i_j, i_{j+1}) \subseteq \overline{W}$ there exists a $C_0 \in \mathcal{S}$ with $i \in \set(C_0) \subseteq (i_j, i_{j+1})$ and $\pi(W, C_0) \equiv E$.
\end{itemize}

\end{definition}

\begin{proposition}\label{reductionscheme}  Let $W \in \W(\{G_n\}_{n \in \omega})$ be proper.

\begin{enumerate}

\item $W \sim E$ if and only if there is a reduction scheme $\mathcal{S}$ on $W$ such that $\bigcup_{C \in \mathcal{S}} \set(C) = \overline{W}$ and $\pi(W, C) \equiv E$ for all $C \in \mathcal{S}$.

\item $W$ is reduced if and only if the only reduction scheme on $W$ is the empty reduction scheme $\mathcal{S} = \emptyset$.

\end{enumerate}

\end{proposition}

\begin{proof} (1) ($\Rightarrow$)  Suppose that $W \in \W(\{G_n\}_{n \in \omega})$ is such that $W(i) \neq 1$ for all $i \in \overline{W}$, and that $W \sim E$.  We will define a reduction scheme on $W$.  For each $n \in \omega$ let $X_n = \{i \in \overline{W}: d(W(i)) = n\}$.  As $W(i) \neq 1$ for all $i \in \overline{W}$ we have $\overline{W} = \bigsqcup_{n \in \omega} X_n$ and as $W$ is a word we know that each $X_n$ is finite.  We have $p_N(W) = W \upharpoonright \bigsqcup_{n = 0}^N X_n$ for each $N \in \omega$, and as $W \sim E$ we know $p_N(W)$ is equal in $*_{n = 0}^N G_n$ to $E$.  From combinatorics in free products, for every $N \in \omega$ we have a reduction scheme $\mathcal{S}_N$ such that $\pi(W, C) \equiv E$ for all $C \in \mathcal{S}_N$ and $\bigcup_{C \in \mathcal{S_N}}\set(C) = \overline{p_N(U)} = \bigsqcup_{n = 0}^N X_n$.  Let $Y_{-1} = \omega$.  If $X_0 = \emptyset$ then we let $Y_0 = Y_{-1}$.  Else we know, since $X_0$ is finite, that there is an infinite $Y_0 \subseteq Y_{-1}$ such that for all $m, m' \in Y_0$ we have $\{C \in \mathcal{S}_m: d(C) = 0\} = \{C \in \mathcal{S}_{m'}: d(C) = 0\}$.  Supposing that we have defined infinite $Y_N \subseteq \omega$, if $X_{N+1} = \emptyset$ then let $Y_{N+1} = Y_N$, else we select infinite $Y_{N+1} \subseteq Y_N$ such that $m, m' \in Y_{N+1}$ we have $\{C \in \mathcal{S}_m: d(C) = N+1\} = \{C \in \mathcal{S}_{m'}: d(C) = N+1\}$.  Let

\begin{center}

$\mathcal{S} = \{C: (\exists m, N \in \omega)m \in Y_N \wedge d(C) = N \wedge C \in \mathcal{S}_{m}\}$.

\end{center}

\noindent Letting $C_0, C_1 \in \mathcal{S}$ be distinct, we let $N = \max\{d(C_0), d(C_1)\}$.  Pick $m \in Y_N$ and we have $C_0, C_1 \in \mathcal{S}_m$, and so $\set(C_0) \cap \set(C_1) = \emptyset$.  Next, we let $C = (i_0; \ldots; i_k) \in \mathcal{S}$ be given, $0 \leq j < k$ and $i$ in the interval $(i_j, i_{j+1}) \subseteq \overline{W}$.  Let $N = \max\{d(U(i)), d(C)\}$ and select $m \in Y_N$.  There is a unique $C_0 \in \mathcal{S}_m$ such that $i \in \set(C_0) \subseteq (i_j, i_{j+1})$, and $C_0 \in \mathcal{S}$.  Also, $\pi(W, C_0) \equiv E$, as indeed $\pi(W, C_1) \equiv E$ for every $C_1 \in \mathcal{S}$.  Thus $\mathcal{S}$ is a reduction scheme on $W$ it is clear that $\bigcup_{C \in \mathcal{S}} \set(C) = \overline{W}$ and $\pi(W, C) = E$ for all $C \in \mathcal{S}$.

(1)  ($\Leftarrow$)  Suppose that $W \in \W(\{G_n\}_{n \in \omega})$ is such that $W(i) \neq 1$ for all $i \in \overline{W}$.  Suppose that there is a reduction scheme $\mathcal{S}$ on $W$ such that $\bigcup_{C \in \mathcal{S}} \set(C) = \overline{W}$ and $\pi(W, C) \equiv E$ for all $C \in \mathcal{S}$.  For a given $N \in \omega$ we let $\mathcal{S}_N = \{C \in \mathcal{S}: d(C) \leq N\}$, and notice that $\mathcal{S}_N$ witnesses that $p_N(W)$ is equal to identity in $*_{n = 0}^N G_n$, and therefore $W \sim E$.

(2)  ($\Rightarrow$)  Suppose that $W \in \W(\{G_n\}_{n \in \omega})$ has $W(i) \neq 1$ for all $i \in \overline{W}$.  Suppose that there exists a nonempty reduction scheme $\mathcal{S}$ on $W$.  Let $C = (i_0; \ldots; i_k) \in \mathcal{S}$.  If $i_0$ and $i_1$ are neighboring then we know that $W$ was not reduced, since $W(i_0)$ and $W(i_1)$ are in the group $G_{d(C)}$, and if $i_0$ and $i_1$ are not neighboring then the reduction scheme $\mathcal{S}' = \{C \in \mathcal{S}: \set(C) \cap (i_0, i_1) \neq \emptyset\}$ witnesses that the nonempty subword $W \upharpoonright (i_0, i_1)$ is $\sim E$, by part (1).  Thus in either case, $W$ is not reduced.

(2) ($\Leftarrow$)  Suppose that $W \in \W(\{G_n\}_{n \in \omega})$ has $W(i) \neq 1$ for all $i \in \overline{W}$.  Suppose that the only reduction scheme on $W$ is the empty scheme.  If there are neighboring $i_0, i_1 \in \overline{W}$ such that $W(i_0)$ and $W(i_1)$ are in the same $G_n$, then $\mathcal{S} = \{C\}$ where $C = (i_0; i_1)$ is a nonempty reduction scheme, contradiction.  If there is a nonempty interval $I \subseteq \overline{W}$ such that $W \upharpoonright I \sim E$ then by part (1) we have a reduction scheme $\mathcal{S}$ on $W \upharpoonright I$ such that $\bigcup_{C \in \mathcal{S}} \set(C) = I$, so in particular $\mathcal{S}$ is a nonempty reduction scheme on $W$ itself, contradiction.

\end{proof}

\begin{definition}  For $W \in \Red(\{G_n\}_{n \in \omega})$ we let $\Sub(W)$ be the set of subwords of $W$, and more generally for a collection $\{W_x\}_{x \in X} \subseteq \Red(\{G_n\}_{n \in \omega})$ we let $\Sub(\{W_x\}_{x \in X}) = \bigcup_{x \in X}\Sub(W_x)$.  For $W \in \Red(\{G_n\}_{n \in \omega})$ we let $\Let(W)$ be the set of letters used in $W$ (this is the image of the function $W$), and more generally for a collection $\{W_x\}_{x \in X} \subseteq \Red(\{G_n\}_{n \in \omega})$ we let $\Let(\{W_x\}_{x \in X}) = \bigcup_{x \in X} \Let(W_x)$.
\end{definition}

\begin{lemma}\label{nicefine}  Suppose $\{W_x\}_{x \in X} \subseteq \Red(\{G_n\}_{n \in \omega})$.  The following are equivalent for a word $W \in \Red(\{G_n\}_{n \in \omega})$:

\begin{enumerate}

\item $W \in \langle \Sub(\{W_x\}_{x \in X}) \rangle \leq \Red(\{G_n\}_{n \in \omega})$;

\item $W$ can be expressed (not necessarily uniquely) as a finite concatenation

\begin{center}
$W \equiv W_0W_1 \cdots W_k$
\end{center}

\noindent and for each $0 \leq j \leq k$ at least one of the following holds:

\begin{itemize}
\item $W_j \in \Sub(\{W_x\}_{x \in X} \cup \{W_x^{-1}\}_{x \in X})$ is nonempty;

\item $|\overline{W_j}| = 1$ with $\Let(W_j) = \{g\}$ such that $g \in G_{d(g)}\setminus \{1\}$ and $g$ is a product of elements in $G_{d(g)} \cap \Let(\{W_x\}_{x \in X} \cup \{W_x^{-1}\}_{x \in X})$.

\end{itemize}

\end{enumerate}

\end{lemma}

\begin{proof}  Certainly if $W$ may be expressed as a finitary concatenation as in (2) then $W \in \langle \Sub(\{W_x\}_{x \in X}) \rangle$ because each $W_j \in \Sub(\{W_x\}_{x \in X}) \rangle$, so that condition (2) implies condition (1).   But by Lemma \ref{Eda} and Observation \ref{howtomultuply} it is clear that the set of all words satisfying condition (2) forms a subgroup of $\Red(\{G_n\}_{n \in \omega})$ which includes the set $\Sub(\{W_x\}_{x \in X})$.  Thus condition (1) implies condition (2).
\end{proof}

\begin{definition}  Given a collection $\{W_x\}_{x \in X} \subseteq \Red(\{G_n\}_{n \in \omega})$ we let $$\Fine(\{W_x\}_{x \in X}) = \langle \Sub(\{W_x\}_{x \in X}) \rangle \leq \Red(\{G_n\}_{n \in \omega})$$ (compare \cite[page 600]{E2}).  It is clear by condition (2) of Lemma \ref{nicefine} that $$\Sub(\Fine(\{W_x\}_{x \in X})) = \Fine(\{W_x\}_{x \in X}).$$
\end{definition}

\begin{definitions}   Define $\mathcal{A}(\{G_n\}_{n \in \omega})$ to be the quotient group $$\Red(\{G_n\}_{n \in \omega})/\langle\langle *_{n \in \omega}G_n\rangle\rangle$$ Let $\beth: \topprod_{n \in \omega} G_n \rightarrow \mathcal{A}(\{G_n\}_{n \in \omega})$ denote the quotient homomorphism, and for a word $W \in \topprod_{n \in \omega} G_n$ we use $[[W]]$ to denote the equivalence class of $W$ in the quotient $\mathcal{A}(\{G_n\}_{n \in \omega})$, i.e. $\beth(W) = [[W]]$.
\end{definitions}

It turns out that one can make slight modifications to the sequence $\{G_n\}_{n \in \omega}$ without changing the isomorphism type of $\mathcal{A}(\{G_n)\}_{n\in \omega})$ (see \cite[Lemma 17]{CHM}):

\begin{lemma}\label{nicefactsaboutarch}  For a sequence $\{G_n\}_{n \in \omega}$ of groups the following assertions hold.

\begin{enumerate}

\item  If $f: \omega \rightarrow \omega$ is a bijection then $\mathcal{A}(\{G_n\}_{n\in \omega}) \simeq \mathcal{A}(\{G_{f(n)}\}_{n\in \omega})$.

\item  For each $k \in \omega$ we have an isomorphism $\mathcal{A}(\{G_n\}_{n \in \omega}) \simeq \mathcal{A}(\{G_{n + k}\}_{n\in \omega})$.

\item We have an isomorphism $\mathcal{A}(\{G_n\}_{n \in \omega}) \simeq \mathcal{A}(\{G_{2n} * G_{2n + 1}\}_{n \in \omega})$.

\end{enumerate}

\end{lemma}

\begin{obs}\label{finitedeletions}  For $W \in \Red(\{G_n\}_{n \in \omega})$ expressed as a finite concatenation $W \equiv W_0W_1 \cdots W_k$ and $J \subseteq \{0, \ldots, k\}$ such that $\overline{W_j}$ is finite for each $j \in J$, we may write $[[W]] = \prod_{0 \leq j \leq k, j \notin J}[[W_i]]$, since $\overline{W_j}$ finite implies that $W_j \in *_{n \in \omega}G_n$ and more particularly $[[W_j]]=[[E]]$.
\end{obs}

\end{section}

\begin{section}{Coi collections}\label{Howtoisomorphize}

We recall some notions introduced in \cite{Cors}.  The straightforward proofs of the lemmas are given in that paper.

\begin{definitions}  Recall that if $\Lambda$ is a totally ordered set we say $I \subseteq \Lambda$ is an \emph{interval} in $\Lambda$ if it is a convex subset (that is- if $\lambda_0, \lambda_1 \in \Lambda$ and $\lambda_0 < \lambda_2 < \lambda_1$ then $\lambda_2 \in I$).  In particular an interval in $\Lambda$ can be unbounded.  We'll use standard conventions for interval notation, so for example $[\lambda_0, \lambda_1] = \{\lambda_2 \in \Lambda \mid \lambda_0 \leq \lambda_2 \leq \lambda_1\}$, $(\lambda_0, \lambda_1] = \{\lambda_2 \in \Lambda \mid \lambda_0 < \lambda_2 \leq \lambda_1\}$, and $[\lambda_0, \infty) = \{\lambda_1 \in \Lambda \mid \lambda_0 \leq \lambda_1\}$.  If $\Lambda$ is a totally ordered set we'll say that $\Lambda_0 \subseteq \Lambda$ is \emph{close in $\Lambda$}, writing $\Close(\Lambda_0, \Lambda)$, provided every infinite interval in $\Lambda$ has nonempty intersection with $\Lambda_0$.  
\end{definitions}

\begin{lemma}\label{basiccloseproperties}  The following hold:

\begin{enumerate}[(i)]

\item  If $\Close(\Lambda_0, \Lambda)$ then for any infinite interval $I \subseteq \Lambda$ the set $I \cap \Lambda_0$ is infinite.

\item  If $\Lambda_2 \subseteq \Lambda_1 \subseteq \Lambda_0$ with $\Close(\Lambda_{j+1}, \Lambda_j)$  for $j = 0, 1$, then $\Close(\Lambda_2, \Lambda_0)$.

\item  If $\Lambda \equiv \prod_{\theta \in \Theta} \Lambda_{\theta}$, $\Close(\{\theta \in \Theta \mid \Lambda_{\theta, 0} \neq \emptyset\}, \Theta)$, and $\Close(\Lambda_{\theta, 0}, \Lambda_{\theta})$ for each $\theta \in \Theta$ then $\Close(\bigcup_{\theta \in \Theta} \Lambda_{\theta, 0}, \Lambda)$.

\item  If  $I_0$ is an interval in $\Lambda$ and $\Close(\Lambda_0, \Lambda)$ then $\Close(\Lambda_0 \cap I_0, I_0)$.

\end{enumerate}

\end{lemma}

\begin{definition}  If $\Close(\Lambda_0, \Lambda)$ then for each interval $I \subseteq \Lambda$ we let $\alpha(I, \Lambda_0)$ denote the smallest interval in $\Lambda$ which includes the set $I \cap \Lambda_0$.  In other words $\alpha(I, \Lambda_0) =  \bigcup_{\lambda_0, \lambda_1 \in I \cap \Lambda_0, \lambda_0 \leq \lambda_1} [\lambda_0, \lambda_1]$ where the intervals $[\lambda_0, \lambda_1]$ are being considered in $\Lambda$.
\end{definition}

\begin{lemma}\label{prettyclose} Let $\Close(\Lambda_0, \Lambda)$ and $I \subseteq \Lambda$ be an interval.

\begin{enumerate}[(i)]

\item The inclusion $I \supseteq \alpha(I, \Lambda_0)$ holds and $\alpha(I, \Lambda_0) = \alpha(\alpha(I, \Lambda_0), \Lambda_0)$.

\item  The set $I \setminus \alpha(I, \Lambda_0)$ is the disjoint union of an initial and terminal subinterval $I_0, I_1 \subseteq I$ (either subinterval could be empty) with $|I_0|, |I_1| < \infty$.
\end{enumerate}

\end{lemma}

\begin{definition}  Two totally ordered sets $\Lambda$ and $\Theta$ are \emph{close-isomorphic} if there exist $\Lambda_0 \subseteq \Lambda$ and $\Theta_0 \subseteq \Theta$ with $\Close(\Lambda_0, \Lambda)$, $\Close(\Theta_0, \Theta)$ and $\Lambda_0$ order isomorphic to $\Theta_0$.  If $\iota$ is an order isomorphism between such a $\Lambda_0$ and $\Theta_0$ then we will call $\iota$ a \emph{close order isomorphism from $\Lambda$ to $\Theta$}.
\end{definition}

Clearly the inverse of a close order isomorphism (abbreviated \emph{coi}) from $\Lambda$ to $\Theta$ is a close order isomorphism from $\Theta$ to $\Lambda$.  Also, a coi between $\Lambda$ and $\Theta$ induces a coi between the reversed orders $\Lambda^{-1}$ and $\Theta^{-1}$ in the obvious way.

\begin{definition}  Given coi $\iota: \Lambda_0 \rightarrow \Theta_0$ between $\Lambda$ and $\Theta$ and an interval $I \subseteq \Lambda$ we let $\alpha(I, \iota)$ denote the smallest interval in $\Theta$ which includes the set $\iota(I \cap \Lambda_0)$.  Thus  $\alpha(I, \iota) = \bigcup_{\theta_0, \theta_1 \in \iota(I \cap \Lambda_0), \theta_0 \leq \theta_1} [\theta_0, \theta_1]$, where each interval $[\theta_0, \theta_1]$ is being considered in $\Theta$.
\end{definition}

\begin{lemma}\label{almostidentified}  If $\iota: \Lambda_0 \rightarrow \Theta_0$ is a coi between $\Lambda$ and $\Theta$ and $I \subseteq \Lambda$ is an interval then $\alpha(\alpha(I, \iota), \iota^{-1}) = \alpha(I, \Lambda_0)$.
\end{lemma}

\begin{lemma}\label{coilemma}  Let $I \subseteq \Lambda$ be an interval, $I \equiv I_0 \cdots I_k$, and $\iota: \Lambda_0 \rightarrow \Theta_0$ a coi from $\Lambda$ to $\Theta$.  Then there exist (possibly empty) finite subintervals $I_0', \ldots, I_{k+1}'$ of $\alpha(I, \iota)$ such that 

\begin{center}

$\alpha(I, \iota) \equiv I_0'\alpha(I_0, \iota) I_1' \alpha(I_1, \iota) I_2' \cdots \alpha(I_k, \iota) I_{k+1}'$.

\end{center}
\end{lemma}

\begin{lemma}\label{coilemma2}  Let $\iota: \Lambda_0 \rightarrow \Theta_0$ be a coi from $\Lambda$ to $\Theta$.  If $I_0 \subseteq \Lambda$ is finite then $\alpha(I_0, \iota)$ is finite.
\end{lemma}

\begin{definition}  Let $\{G_n\}_{n\in \omega}$ and $\{K_n\}_{n \in \omega}$ be sequences of groups.  For words $W \in \Red(\{G_n\}_{n\in \omega})$ and $U \in \Red(\{K_n\}_{n \in \omega})$ we let $\coi(W, \iota, U)$ denote that $\iota$ is a coi from $\overline{W}$ to $\overline{U}$.  We will call such an ordered triple a \emph{coi triple} from $\Red(\{G_n\}_{n\in \omega})$ to $\Red(\{K_n\}_{n \in \omega})$.
\end{definition}

\begin{definition}  A collection $\{\coi(W_x, \iota_x, U_x)\}_{x\in X}$ of coi triples from $\Red(\{G_n\}_{n\in \omega})$ to $\Red(\{K_n\}_{n \in \omega})$ is \emph{coherent} if for any choice of $x_0, x_1 \in X$, intervals $I_0 \subseteq \overline{W_{x_0}}$ and $I_1 \subseteq \overline{W_{x_1}}$ and $\delta \in \{-1, 1\}$ such that $W_{x_0} \upharpoonright I_0 \equiv (W_{x_1} \upharpoonright I_1)^{\delta}$ we get

\begin{center}

$[[U_{x_0} \upharpoonright \alpha(I_0, \iota_{x_0})]] = [[(U_{x_1}\upharpoonright \alpha(I_1, \iota_{x_1}))^{\delta}]]$

\end{center}
 
\noindent and similarly for any choice of $x_2, x_3 \in X$, intervals $I_2 \subseteq \overline{U_{x_2}}$ and $I_3 \subseteq \overline{U_{x_3}}$ and $\epsilon \in \{-1, 1\}$ such that $U_{x_2} \upharpoonright I_2 \equiv (U_{x_3} \upharpoonright I_3)^{\epsilon}$ we get

\begin{center}

$[[W_{x_2} \upharpoonright \alpha(I_2, \iota_{x_2}^{-1})]] = [[(W_{x_3} \upharpoonright \alpha(I_3, \iota_{x_3}^{-1}))^{\epsilon}]]$.

\end{center}
\end{definition}

\noindent If collection of coi triples $\{\coi(W_x, \iota_x, U_x)\}_{x\in X}$ from $\Red(\{G_n\}_{n\in \omega})$ to $\Red(\{K_n\}_{n \in \omega})$ is coherent then the collection of coi triples $\{\coi(U_x, \iota_x^{-1}, W_x)\}_{x\in X}$ from $\Red(\{K_n\}_{n \in \omega})$ to $\Red(\{G_n\}_{n\in \omega})$ is also coherent.  

A coherent collection need not be a one-to-one pairing of a subset of elements in $\Red(\{G_n\}_{n\in \omega})$ with elements in $\Red(\{K_n\}_{n\in \omega})$.  If, for example, each element of $\{W_x\}_{x \in X}$ has $|\overline{W_x}| = 1$ then the collection $\{(W_x, \iota_x, E)\}_{x \in X}$ is coherent (each $\iota_x$ is the empty function).  The proof of the following is clear (for example, see the corresponding result in \cite{Cors}).

\begin{lemma}\label{ascendingchaincoi}  Suppose that $\Theta$ is a totally ordered set and that $\{\mathcal{T}_{\theta}\}_{\theta \in \Theta}$ is a collection of coherent collections of coi triples from $\Red(\{G_n\}_{n \in \omega})$ to $\Red(\{K_n\}_{n \in \omega})$ such that $\theta \leq \theta'$ implies $\mathcal{T}_{\theta} \subseteq \mathcal{T}_{\theta'}$.  Then $\bigcup_{\theta \in \Theta} \mathcal{T}_{\theta}$ is coherent.
\end{lemma}

The proof of the following follows closely that of a comparable claim in \cite{Cors}; we include the proof here since it is technical and also for the sake of completeness.

\begin{proposition}\label{welldefinedfuns}  From a coherent collection $\{\coi(W_x, \iota_x, U_x)\}_{x \in X}$ of coi triples from $\Red(\{G_n\}_{n \in \omega})$ to $\Red(\{K_n\}_{n \in \omega})$ we obtain a well-defined function $$\phi_0 : \Fine(\{W_x\}_{x \in X}) \rightarrow \mathcal{A}(\{K_n\}_{n \in \omega})$$ given by

\begin{center}

$W \mapsto [[(U_{x_{r_0}} \upharpoonright \alpha(I_0, \iota_{x_{r_0}}))^{\delta_0}]][[(U_{x_{r_1}} \upharpoonright \alpha(I_1, \iota_{x_{r_1}}))^{\delta_1}]]\cdots[[(U_{x_{r_s}}\upharpoonright \alpha(I_s, \iota_{x_{r_s}}))^{\delta_s}]]$

\end{center}

\noindent where

\begin{itemize}

\item $W  \equiv W_0W_1 \cdots W_k$ is a decomposition as in Lemma \ref{nicefine};

\item $0 \leq r_0 < r_1 < \cdots < r_s \leq k$ are such that $\{r_0, \ldots, r_s\} = \{r \in \{0, 1, \ldots, k\}\mid |\overline{W_r}| \geq 2\}$; and 

\item $W_{r_j} \equiv (W_{x_{r_j}} \upharpoonright I_j)^{\delta_j}$ with $I_j \subseteq \overline{W_{x_{r_j}}}$ an interval and $\delta_j \in \{-1, 1\}$ for each $0 \leq j \leq s$.

\end{itemize}

There is also a comparable well-defined function $$\phi_1: \Fine(\{U_x\}_{x \in X}) \rightarrow \mathcal{A}(\{G_n\}_{n \in \omega})$$

\noindent given by

\begin{center}

$U \mapsto [[(W_{x_{r_0}} \upharpoonright \alpha(I_0, \iota_{x_{r_0}}^{-1}))^{\delta_0}]][[(W_{x_{r_1}} \upharpoonright \alpha(I_1, \iota_{x_{r_1}}^{-1}))^{\delta_1}]]\cdots[[(W_{x_{r_s}}\upharpoonright \alpha(I_s, \iota_{x_{r_s}}^{-1}))^{\delta_s}]]$

\end{center}

\noindent where

\begin{itemize}

\item $U  \equiv U_0U_1 \cdots U_k$ is a decomposition as in Lemma \ref{nicefine};

\item $0 \leq r_0 < r_1 < \cdots < r_s \leq k$ are such that $\{r_0, \ldots, r_s\} = \{r \in \{0, 1, \ldots, k\}\mid |\overline{U_r}| \geq 2\}$; and 

\item $U_{r_j} \equiv (U_{x_{r_j}} \upharpoonright I_j)^{\delta_j}$ with $I_j \subseteq \overline{U_{x_{r_j}}}$ an interval and $\delta_j \in \{-1, 1\}$ for each $0 \leq j \leq s$.

\end{itemize}

\end{proposition}

\begin{proof}  We must show that the described function is well-defined; in other words the function is independent of all of the choices made in the description.  Therefore, suppose that we can write

\begin{itemize}

\item $\overline{W} \equiv J_0J_1 \cdots J_k$ where $W \equiv (W \upharpoonright J_0) \cdots (W \upharpoonright J_k)$ is a decomposition as in Lemma \ref{nicefine};

\item $0 \leq r_0 < r_1 < \cdots < r_s \leq k$ are such that $\{r_0, \ldots, r_s\} = \{r \in \{0, 1, \ldots, k\}\mid |J_r| \geq 2\}$; and 

\item $W \upharpoonright J_{r_j} \equiv (W_{x_{r_j}} \upharpoonright I_j)^{\delta_j}$ with $I_j \subseteq \overline{W_{x_{r_j}}}$ an interval and $\delta_j \in \{-1, 1\}$ for each $0 \leq j \leq s$;

\end{itemize}

\noindent and also we may write

\begin{itemize}

\item $\overline{W} \equiv L_0L_1 \cdots L_p$ where $W \equiv (W \upharpoonright L_0) \cdots (W \upharpoonright L_p)$ is a decomposition as in Lemma \ref{nicefine};

\item $0 \leq q_0 < q_1 < \cdots < q_t \leq p$ are such that $\{q_0, \ldots, q_t\} = \{q \in \{0, 1, \ldots, p\}\mid |L_q| \geq 2\}$; and 

\item $W \upharpoonright L_{q_l} \equiv (W_{x_{q_l}'} \upharpoonright E_l)^{\epsilon_l}$ with $E_l \subseteq \overline{W_{x_{q_l}'}}$ an interval and $\epsilon_l \in \{-1, 1\}$ for each $0 \leq l \leq t$.
\end{itemize}

\noindent We must show that the elements 

\begin{center}
$[[(U_{x_{r_0}} \upharpoonright \alpha(I_0, \iota_{x_{r_0}}))^{\delta_0}]]\cdots[[(U_{x_{r_s}}, \alpha(I_s, \iota_{x_{r_s}}))^{\delta_s}]]$
\end{center}

\noindent and

\begin{center}
$[[(U_{x_{q_0}'} \upharpoonright \alpha(E_0, \iota_{x_{q_0}'}))^{\epsilon_0}]]\cdots[[(U_{x_{q_t}'} \upharpoonright \alpha(E_t, \iota_{x_{q_t}'}))^{\epsilon_t}]]$
\end{center}

\noindent are equal in $\mathcal{A}(\{K_n\}_{n \in \omega})$.

First, we will let $f_{j}: J_{r_j} \rightarrow I_j^{\delta_j}$ be the order isomorphism witnessing $W \upharpoonright J_{r_j} \equiv  (W_{x_{r_j}} \upharpoonright I_j)^{\delta_j}$ for each $0 \leq j \leq s$.  Let $g_l: L_{q_l} \rightarrow E_l^{\epsilon_l}$ be the order isomorphism witnessing $W \upharpoonright L_{q_l} \equiv (W_{x_{q_l}} \upharpoonright E_l)^{\epsilon_l}$ for each $0 \leq l \leq t$.

We will take $\mathbb{J}$ to be the set of intervals consisting of nonempty intersections of one of the $J_r$ with one of the $L_q$, where $0 \leq r \leq k$ and $0 \leq q \leq p$.  The elements of $\mathbb{J}$ are pairwise disjoint intervals in $\overline{W}$, ordered in the natural way, such that $\overline{W} \equiv \prod_{M \in \mathbb{J}} M$.  For each $0 \leq r \leq k$ we let $\mathbb{J}_{r}$ be the set of those elements of $\mathbb{J}$ which are subsets of $J_{r}$, and for each $0 \leq q \leq p$ we let $\mathbb{J}_{(q)}$ be the set of those elements of $\mathbb{J}$ which are subsets of $L_{q}$.  Let $\overline{\mathbb{J}} \subseteq \mathbb{J}$ be the subset in $\mathbb{J}$ consisting of those intervals which are of cardinality greater than $1$, and $\overline{\mathbb{J}_{r}} = \mathbb{J}_{r} \cap \overline{\mathbb{J}}$ for each $0 \leq r \leq k$ and $\overline{\mathbb{J}_{(q)}} = \mathbb{J}_{(q)} \cap \overline{\mathbb{J}}$ and $0 \leq q \leq p$.

Define $f_{-}: \overline{\mathbb{J}} \rightarrow \{0, \ldots, s\}$ by $f_{-}(M) = j$ where $M \subseteq J_{r_j}$ and $f^{-}: \overline{\mathbb{J}} \rightarrow \{0, \ldots, t\}$ by $f^{-}(M) = l$ where $M \subseteq L_{q_l}$.  Notice that for each $M \in \overline{\mathbb{J}}$ we have $f_{f_{-}(M)}(M) \subseteq (\overline{W_{x_{r_{f_{-}(M)}}}})^{\delta_{f_{-}(M)}}$, and so $(f_{f_{-}(M)}(M))^{\delta_{f_{-}(M)}} \subseteq \overline{W_{x_{r_{f_{-}(M)}}}}$.  Similarly $(f_{f^{-}(M)}(M))^{\epsilon_{f^{-}(M)}} \subseteq \overline{W_{x_{q_{f^{-}(M)}}'}}$ and also

$$
\begin{array}{ll}
(W_{x_{r_{f_{-}(M)}}} \upharpoonright (f_{f_{-}(M)}(M))^{\delta_{f_{-}(M)}})^{\delta_{f_{-}(M)}} & \equiv  W \upharpoonright M\\
& \equiv (W_{x_{q_{f^{-}(M)}}'} \upharpoonright (f_{f^{-}(M)}(M))^{\epsilon_{f^{-}(M)}})^{\epsilon_{f^{-}(M)}}.
\end{array}
$$

\noindent Therefore by the coherence of the collection of coi triples we see for each $M \in \overline{\mathbb{J}}$ that

\begin{center}

$$[[(U_{x_{r_{f_{-}(M)}}} \upharpoonright \alpha((f_{f_{-}(M)}(M))^{\delta_{f_{-}(M)}} , \iota_{x_{r_{f_{-}(M)}}}))^{\delta_{f_{-}(M)}}]]$$

$$= [[(U_{x_{q_{f^{-}(M)}}'} \upharpoonright \alpha((g_{f^{-}(M)}(M))^{\epsilon_{f^{-}(M)}} , \iota_{x_{q_{f^{-}(M)}}'}))^{\epsilon_{f^{-}(M)}}]].\eqno{(*)}
$$.

\end{center}

Next, we claim that for each $0 \leq j \leq s$ we have

$$
\begin{array}{ll}
[[(U_{x_{r_j}}\upharpoonright \alpha(I_j, \iota_{x_{r_j}}))^{\delta_j}]] & =  \prod_{M \in \overline{\mathbb{J}_{r_j}}} [[(U_{x_{r_j}}\upharpoonright \alpha((f_j(M))^{\delta_j}, \iota_{x_{r_j}}))^{\delta_j}]].
\end{array}\eqno{(**)}
$$

\noindent To see why this is true, we recall that for each $M \in \mathbb{J}_{r_j}$ the set $f_j(M)$ is an interval in $I_j^{\delta_j}$, and so $(f_j(M))^{\delta_j}$ is an interval in $I_j$.  Then $I_j \equiv \prod_{M \in \mathbb{J}_{r_j}^{\delta_j}} (f_j(M))^{\delta_j}$.  Then we may write

\begin{center}
$\alpha(I_j, \iota_{x_{r_j}}) \equiv (\prod_{M \in \mathbb{J}_{r_j}^{\delta_j}} Q_M\alpha((f_j(M))^{\delta_j}, \iota_{x_{r_j}}))Q_f$
\end{center}

\noindent where each element of $\{Q_M\}_{M \in \mathbb{J}_{r_j}^{\delta_j}} \cup \{Q_f\}$ is a finite interval, by Lemma \ref{coilemma}.  Also, whenever $|M| = 1$ (that is, when $M \in \mathbb{J}_{r_j} \setminus \overline{\mathbb{J}_{r_j}}$) we have by Lemma \ref{coilemma2} that $\alpha(f_j(M), \iota_{x_{r_j}})$ is finite.  Now 

$$
\begin{array}{ll}
U_{x_{r_k}}\upharpoonright \alpha(I_k, \iota_{x_{r_j}}) & \equiv  (\prod_{M \in \mathbb{J}_{r_j}^{\delta_j}} (U_{x_{r_j}}\upharpoonright Q_M)(U_{x_{r_j}}\upharpoonright\alpha((f_j(M))^{\delta_j}, \iota_{x_{r_j}})))\\
& \cdot (U_{x_{r_j}}\upharpoonright Q_f)
\end{array}
$$

\noindent and by taking the $[[\cdot]]$ class of both sides and deleting $[[U_{x_{r_j}}\upharpoonright Q_f]]$, and all $[[U_{x_{r_j}}\upharpoonright Q_M]]$, and all $U_{x_{r_j}}\upharpoonright\alpha((f_j(M))^{\delta_j}, \iota_{x_{r_j}})$ for $M \in \mathbb{J}_{r_j} \setminus \overline{\mathbb{J}_{r_j}}$ we see that

$$
\begin{array}{ll}
[[U_{x_{r_j}}\upharpoonright \alpha(I_j, \iota_{x_{r_j}})]] & =  \prod_{M \in \overline{\mathbb{J}_{r_j}}^{\delta_j}} [[U_{x_{r_j}}\upharpoonright\alpha((f_j(M))^{\delta_j}, \iota_{x_{r_j}})]]
\end{array}
$$

\noindent and by taking the $\delta_j$ power of both sides we obtain the desired equality.  By the same reasoning we have for each $0 \leq l \leq t$ that  

$$
\begin{array}{ll}
[[(U_{x_{q_l}'} \upharpoonright \alpha(E_l, \iota_{x_{q_l}'}))^{\epsilon_l}]] & =  \prod_{M \in \overline{\mathbb{J}_{(q_l)}}}[[(U_{x_{q_l}'} \upharpoonright \alpha((g_l(M))^{\epsilon_l} , \iota_{x_{q_l}'}))^{\epsilon_l}]].
\end{array}
$$

\noindent Thus 

$$
\begin{array}{ll}
\prod_{j = 0}^s[[(U_{x_{r_j}}\upharpoonright \alpha(I_j, \iota_{x_{r_j}}))^{\delta_j}]]\\
= \prod_{j = 0}^s \prod_{M \in \overline{\mathbb{J}_{r_j}}} [[(U_{x_{r_j}}\upharpoonright \alpha((f_j(M))^{\delta_j}, \iota_{x_{r_j}}))^{\delta_j}]]\\
=  \prod_{j = 0}^s\prod_{M \in \overline{\mathbb{J}_{r_j}}}[[(U_{x_{r_{f_{-}(M)}}}\upharpoonright \alpha((f_{f_{-}(M)}(M))^{\delta_{f_{-}(M)}} , \iota_{x_{r_{f_{-}(M)}}}))^{\delta_{f_{-}(M)}}]]\\
=  \prod_{j = 0}^s\prod_{M \in \overline{\mathbb{J}_{r_j}}}[[(U_{x_{q_{f^{-}(M)}}'} \upharpoonright \alpha((g_{f^{-}(M)}(M))^{\epsilon_{f^{-}(M)}} , \iota_{x_{q_{f^{-}(M)}}'}))^{\epsilon_{f^{-}(M)}}]]\\
=  \prod_{M \in \overline{\mathbb{J}}}[[(U_{x_{q_{f^{-}(M)}}'} \upharpoonright \alpha((g_{f^{-}(M)}(M))^{\epsilon_{f^{-}(M)}} , \iota_{x_{q_{f^{-}(M)}}'}))^{\epsilon_{f^{-}(M)}}]]\\
=  \prod_{l = 0}^t \prod_{M \in \overline{\mathbb{J}_{(q_l)}}}[[(U_{x_{q_{f^{-}(M)}}'} \upharpoonright \alpha((g_{f^{-}(M)}(M))^{\epsilon_{f^{-}(M)}} , \iota_{x_{q_{f^{-}(M)}}'}))^{\epsilon_{f^{-}(M)}}]]\\
=  \prod_{l = 0}^t \prod_{M \in \overline{\mathbb{J}_{(q_l)}}}[[(U_{x_{q_l}'} \upharpoonright \alpha((g_l(M))^{\epsilon_l} , \iota_{x_{q_l}'}))^{\epsilon_l}]]\\
= \prod_{l = 0}^t[[(U_{x_{q_l}'} \upharpoonright \alpha(E_l, \iota_{x_{q_l}'}))^{\epsilon_l}]]
\end{array}
$$

\noindent where the first equality was established in (**), the second equality is an obvious relabeling, the third equality is a finite term-by-term replacement using (*), the fourth and fifth equalities are each a re-indexing, the sixth equality is an obvious relabeling, the seventh equality was established in the analogue of (**).  The proof for the well-definedness of the comparably defined $\phi_1$ is symmetric.
\end{proof}

Now we are approaching the main utility of coi triples.  We recall that $\beth: \topprod_{n \in \omega} H_n \rightarrow \mathcal{A}(\{H_n\}_{n \in \omega})$ is the quotient map.

\begin{theorem}\label{coicollectiongivesiso}  Suppose we have a coherent collection $\{\coi(W_x, \iota_x, U_x)\}_{x \in X}$ of coi triples from $\Red(\{G_n\}_{n \in \omega})$ to $\Red(\{K_n\}_{n \in \omega})$.  The functions $\phi_0 : \Fine(\{W_x\}_{x \in X}) \rightarrow \mathcal{A}(\{K_n\}_{n \in \omega})$ and $\phi_1: \Fine(\{U_x\}_{x \in X}) \rightarrow \mathcal{A}(\{G_n\}_{n \in \omega})$ defined in Lemma \ref{welldefinedfuns} are homomorphisms.  Moreover these descend to isomorphisms

\begin{center}
$\Phi_0: \beth(\Fine(\{W_x\}_{x \in X})) \rightarrow \beth(\Fine(\{U_x\}_{x \in X}))$
\end{center}

\noindent and

\begin{center}
$\Phi_1:  \beth(\Fine(\{U_x\}_{x \in X})) \rightarrow \beth(\Fine(\{W_x\}_{x \in X}))$
\end{center}

\noindent with $\Phi_1 = \Phi_0^{-1}$.
\end{theorem}

\begin{proof}  We first point out that for $W \in \Fine(\{W_x\}_{x \in X})$ we have $\phi_0(W^{-1}) = (\phi_0(W))^{-1}$.  To see this, we let

\begin{itemize}

\item $\overline{W} \equiv J_0J_1 \cdots J_k$ where $W \equiv (W \upharpoonright J_0) \cdots (W \upharpoonright J_k)$ is a decomposition as in Lemma \ref{nicefine};

\item $0 \leq r_0 < r_1 < \cdots < r_s \leq k$ are such that $\{r_0, \ldots, r_s\} = \{r \in \{0, 1, \ldots, k\}\mid |J_r| \geq 2\}$; and 

\item $W \upharpoonright J_{r_j} \equiv (W_{x_{r_j}} \upharpoonright I_j)^{\delta_j}$ with $I_j \subseteq \overline{W_{x_{r_j}}}$ an interval;
\end{itemize}

\noindent and as 

\begin{itemize}

\item $\overline{W^{-1}} \equiv J_k^{-1} \cdots J_1^{-1}J_0^{-1}$ where $W^{-1} \equiv (W^{-1} \upharpoonright J_k^{-1}) \cdots (W^{-1} \upharpoonright J_0^{-1})$ is a decomposition as in Lemma \ref{nicefine};

\item $0 \leq r_0 < r_1 < \cdots < r_s \leq k$ are such that $\{r_0, \ldots, r_s\} = \{r \in \{0, 1, \ldots, k\}\mid |J_r^{-1}| \geq 2\}$; and 

\item $W^{-1} \upharpoonright J_{r_j}^{-1} \equiv (W_{x_{r_j}} \upharpoonright I_j)^{-\delta_j}$ with $I_j \subseteq \overline{W_{x_{r_j}}}$ an interval;
\end{itemize}

\noindent we get

$$
\begin{array}{ll}
(\phi_0(W))^{-1} & = (\prod_{j = 0}^s[[(U_{x_{r_j}}\upharpoonright \alpha(I_j, \iota_{x_{r_j}}))^{\delta_j}]])^{-1}\\
& = \prod_{j = s}^0[[(U_{x_{r_j}}\upharpoonright \alpha(I_j, \iota_{x_{r_j}}))^{-\delta_j}]]\\
& = \phi_0(W^{-1}).
\end{array}
$$

Also, if $W \in \Fine(\{W_x\}_{x \in X})$ is finite (i.e. $|\overline{W}|< \infty$) then $\phi_0(W) = [[E]]$.  This is seen by writing $\overline{W} \equiv J_0J_1 \cdots J_k$ as in Lemma \ref{nicefine}, and letting $0 \leq r_0 < r_1 < \cdots < r_s \leq k$ be such that $\{r_0, \ldots, r_s\} = \{r \in \{0, 1, \ldots, j\}\mid |J_r| \geq 2\}$, and $W \upharpoonright J_{r_j} \equiv (W_{x_{r_j}} \upharpoonright I_j)^{\delta_j}$.  Since each $J_{r_j}$ is finite, we see that $(W_{x_{r_j}} \upharpoonright I_j)^{\delta_j}$ is a finite word, and so $(U_{x_{r_j}}\upharpoonright \alpha(I_j, \iota_{x_{r_j}}))^{\delta_j}$ will be finite by Lemma \ref{coilemma2}.  Then $\phi_0(W) = \prod_{j = 0}^s[[(U_{x_{r_j}}\upharpoonright \alpha(I_j, \iota_{x_{r_j}}))^{\delta_j}]] = [[E]]$.

Next we notice that when $W \in \Fine(\{W_x\}_{x \in X})$ and $\overline{W} \equiv E_aE_b$ we get $\phi_0(W) = \phi_0(W \upharpoonright E_a)\phi_0(W \upharpoonright E_b)$.  To see this, we let $J_0, J_1, \ldots, J_k \subseteq \overline{W}$, $0 \leq r_0 < \cdots < r_s \leq k$, etc. be as before.  The claim obviously holds if either $E_a$ or $E_b$ is empty, so without loss of generality we assume that $E_a \neq \emptyset \neq E_b$.  Let $0 \leq r \leq k$ be maximal such that $J_r \cap E_a \neq \emptyset$.  If $E_a \equiv J_0 \cdots J_r$ and $E_b \equiv J_{r + 1} \cdots J_k$ then these decompositions are as in Lemma \ref{nicefine} and it is easily seen that

$$
\begin{array}{ll}
\phi_0(W) & = \prod_{j = 0}^s[[(U_{x_{r_j}}\upharpoonright \alpha(I_j, \iota_{x_{r_j}}))^{\delta_j}]]\\
& = (\prod_{0\leq j \leq s, r_j \leq r}[[(U_{x_{r_j}}\upharpoonright \alpha(I_j, \iota_{x_{r_j}}))^{\delta_j}]])\\
& \cdot (\prod_{0\leq j \leq s, r_j > r}[[(U_{x_{r_j}}\upharpoonright \alpha(I_j, \iota_{x_{r_j}}))^{i\delta_j}]]))\\
& = \phi_0(W \upharpoonright E_a)\phi_0(W \upharpoonright E_b).
\end{array}
$$

\noindent Otherwise we have $J_r \cap E_a \neq \emptyset \neq J_r \cap E_b$, which implies in particular that $|J_r| \geq 2$ and therefore $r = r_{k'}$ for some $0 \leq k' \leq s$.  Let $J_r' = J_r \cap E_a$ and $J_r'' = J_r \cap E_b$.  If, for example, $|J_r'| = 1$ and $|J_r''| \geq 2$ then we let $f: J_{r_{k'}} \rightarrow I_{k'}^{\delta_{k'}}$ be an order isomorphism witnessing $W \upharpoonright J_{r_{k'}} \equiv (W_{x_{r_{k'}}} \upharpoonright I_{k'})^{\delta_{k'}}$.  We get 

$$[[(U_{x_{r_{k'}}}\upharpoonright \alpha(I_{k'}, \iota_{x_{r_{k'}}}))^{\delta_{k'}}]] = [[(U_{x_{r_{k'}}}\upharpoonright \alpha((f(J_r''))^{\delta_{k'}}, \iota_{x_{r_{k'}}}))^{\delta_{k'}}]]$$ 

\noindent because $|I_{k'} \setminus f(J_r'')^{\delta_{k'}}| = 1$ (by applying Lemmas \ref{coilemma} and \ref{coilemma2}).  Thus

$$
\begin{array}{ll}
\phi_0(W) & = \prod_{j = 0}^s[[(U_{x_{r_j}}\upharpoonright \alpha(I_j, \iota_{x_{r_j}}))^{\delta_j}]]\\
& = (\prod_{0\leq j \leq k'}[[(U_{x_{r_j}}\upharpoonright \alpha(I_k, \iota_{x_{r_j}}))^{\delta_j}]])\\
& \cdot (\prod_{k' < j \leq s}[[(U_{x_{r_j}}\upharpoonright \alpha(I_j, \iota_{x_{r_j}}))^{\delta_j}]]))\\
& = (\prod_{0\leq j <  k'}[[(U_{x_{r_j}}\upharpoonright \alpha(I_j, \iota_{x_{r_j}}))^{\delta_j}]])\\
& \cdot [[(U_{x_{r_{k'}}}\upharpoonright \alpha(I_{k'}, \iota_{x_{r_{k'}}}))^{\delta_{k'}}]])\\
& \cdot (\prod_{k' < j \leq s}[[(U_{x_{r_j}}\upharpoonright \alpha(I_j, \iota_{x_{r_j}}))^{\delta_j}]]))\\
& = \phi_0(W \upharpoonright E_a)\phi_0(W \upharpoonright E_b).
\end{array}
$$

\noindent If $|J_r'| \geq 2$ and $|J_r''| = 1$, or $|J_r'| = 1 = |J_r''|$, the proof is similar.  If $|J_r'|, |J_r''| \geq 2$ then

$$
\begin{array}{ll}
[[(U_{x_{r_{k'}}}\upharpoonright \alpha(I_{k'}, \iota_{x_{r_{k'}}}))^{\delta_{k'}}]] & = [[(U_{x_{r_{k'}}}\upharpoonright \alpha((f(J_r'))^{\delta_{k'}}, \iota_{x_{r_{k'}}}))^{\delta_{k'}}]]\\
& \cdot  [[(U_{x_{r_{k'}}}\upharpoonright \alpha((f(J_r''))^{\delta_{k'}}, \iota_{x_{r_{k'}}}))^{\delta_{k'}}]]
\end{array}
$$

\noindent by using Lemmas \ref{coilemma} and \ref{coilemma2} as was done in establishing (**) in Lemma \ref{welldefinedfuns} (if, say, $\delta_{k'} = -1$ then one has $I_{k'} \equiv (f(J_r''))^{-1}(f(J_r'))^{-1}$, writes $$\alpha(I_{k'}, \iota_{x_{r_{k'}}}) \equiv Q_0\alpha(f(J_r'), \iota_{x_{r_{k'}}})Q_1\alpha(f(J_r''), \iota_{x_{r_{k'}}})Q_2$$ where each of $Q_0$, $Q_1$, $Q_2$ is a finite interval, etc.)  Thus $\phi_0(W) = \phi_0(W \upharpoonright E_a)\phi_0(W \upharpoonright E_b)$, as required.

Now we let $W, W' \in \Fine(\{W_x\}_{x \in X})$ be given.  Suppose we can write, as in Lemma \ref{Eda}, $W \equiv W_0W_aW_1$ and $W' \equiv W_0'W_bW_1'$ so that $|\overline{W_a}| = 1 = |\overline{W_b}|$, $\Let(W_a)= \{g\}$,  $\Let(W_b) = \{g'\}$, $\{g, g'\} \subseteq G_n$ and $gg' = g'' \neq 1$ in $G_n$ and further that

\begin{enumerate}

\item $W \equiv W_0W_aW_1$;

\item $W' \equiv W_0'W_bW_1'$;

\item $W_0' \equiv (W_1)^{-1}$; and

\item $W_0g''W_1'$ is reduced.
\end{enumerate}

\noindent Then we have

$$
\begin{array}{ll}
\phi_0(W \circledcirc W') & = \phi_0(W_0g''W_1')\\
& = \phi_0(W_0)\phi_0(g'')\phi_0(W_1')\\
& = \phi_0(W_0)\phi_0(W_1')\\
& = \phi_0(W_0)\phi_0(W_1)(\phi_0(W_1))^{-1}\phi_0(W_1')\\
& = \phi_0(W_0)\phi_0(W_a)\phi_0(W_1)\phi_0(W_0')\phi_0(W_b)\phi_0(W_1')\\
& = \phi_0(W)\phi_0(W').
\end{array}
$$

\noindent In the case where $W \circledcirc W' \equiv W_0W_1'$, the $W_a$ and $W_b$ are not necessary and the proof is even easier.

Thus we have established that $\phi_0$ is a homomorphism, and the symmetric argument shows that $\phi_1$ is also a homomorphism.  Moreover all finite words are in the kernel of $\phi_0$, and also that of $\phi_1$, and so the homomorphisms descend to homomorphisms 

\begin{center}
$\Phi_0: \beth(\Fine(\{W_x\}_{x \in X})) \rightarrow \beth(\Fine(\{U_x\}_{x \in X}))$
\end{center}

\noindent and

\begin{center}
$\Phi_1:  \beth(\Fine(\{U_x\}_{x \in X})) \rightarrow \beth(\Fine(\{W_x\}_{x \in X}))$.
\end{center}

It remains to see that $\Phi_0$ and $\Phi_1$ are isomorphisms such that $\Phi_0^{-1} = \Phi_1$.  By applying Lemma \ref{nicefine} we can write any element of $\beth(\Fine(\{W_x\}_{x \in X}))$ as a product $[[W_0]]\cdots[[W_s]]$ where $W_k \in \Sub(\{W_x\}_{x\in X} \cup \{W_x^{-1}\}_{x \in X})$ for each $0 \leq j \leq s$.  For each $0 \leq j \leq s$ we therefore select $x_j \in X$, $\delta_j \in \{1, -1\}$, and interval $I_j \subseteq \overline{W_{x_j}}$ so that $W_j \equiv (W_{x_j} \upharpoonright I_j)^{\delta_j}$.  Then

$$
\begin{array}{ll}
\Phi_1 \circ \Phi_0(\prod_{j = 0}^s[[W_j]]) & = \prod_{j = 0}^s\Phi_1([[(U_{x_j}\upharpoonright \alpha (I_j, \iota_{x_j}))^{\delta_j}]])\\
& = \prod_{j = 0}^s (\Phi_1([[U_{x_j}\upharpoonright \alpha (I_j, \iota_{x_j}))]]))^{\delta_j}\\
& = \prod_{j = 0}^s[[W_{x_j}\upharpoonright \alpha(\alpha (I_j, \iota_{x_j})), \iota_{x_j}^{-1})]]^{\delta_j}\\
& = \prod_{j = 0}^s[[W_{x_j}\upharpoonright I_j]]^{i_j}\\
& = \prod_{j = 0}^s[[W_j]]
\end{array}
$$

\noindent where the fourth equality uses the fact that $I_j \equiv I\alpha(\alpha (I_j, \iota_{x_j})), \iota_{x_j}^{-1})I'$ where $I$ and $I'$ are finite intervals (Lemma \ref{prettyclose}).  Thus $\Phi_1 \circ \Phi_0$ is the identity map, and by the same argument we get $\Phi_0 \circ \Phi_1$ as the identity map.  The theorem is proved.
\end{proof}

Armed with this theorem, it is clear that for producing an isomorphism between $\mathcal{A}(\{G_n\}_{n \in \omega})$ and $\mathcal{A}(\{K_n\}_{n \in \omega})$ it would be sufficient to find a coherent collection $\{\coi(W_x, \iota_x, U_x)\}_{x \in X}$ of coi triples from $\Red(\{G_n\}_{n \in \omega})$ to $\Red(\{K_n\}_{n \in \omega})$ such that $\beth(\Fine(\{W_x\}_{x \in X})) = \mathcal{A}(\{G_n\}_{n \in \omega})$ and $\beth(\Fine(\{U_x\}_{x \in X})) = \mathcal{A}(\{K_n\}_{n \in \omega})$.  We do this by induction, using a back-and-forth argument which ensures that the isomorphism is complete, and not only defined on a proper subgroup of the domain or of the codomain.
\end{section}

\begin{section}{Some straightforward coi extensions}\label{uncomplicated}

Our task is now to give lemmas which will allow us to create the sufficiently large coherent coi collection in order to prove the main theorem.

\begin{lemma}\label{findsomerepresentative}  Suppose that $\{G_n\}_{n \in \omega}$ and $\{K_n\}_{n \in \omega}$ are sequences of nontrivial groups and that $\{\coi(W_x, \iota_x, U_x)\}_{x\in X}$ is a coherent collection of coi triples from $\Red(\{G_n\}_{n \in \omega})$ to $\Red(\{K_n\}_{n \in \omega})$.

\begin{enumerate}
\item  If $W \in \Fine(\{W_x\}_{x\in X})$ then there exists a $U \in \Red(\{K_n\}_{n \in \omega})$ and coi $\iota$ from $W$ to $U$ such that $\{\coi(W_x, \iota_x, U_x)\}_{x\in X} \cup \{\coi(W, \iota, U)\}$ is coherent.  Moreover if $W$ is nonempty the domain and range of $\iota$ can be made to be nonempty.

\item If $U \in \Fine(\{U_x\}_{x \in X})$ then there exists a $W \in \Red(\{G_n\}_{n \in \omega})$ and coi $\iota'$ from $W$ to $U$ such that $\{\coi(W_x, \iota_x, U_x)\}_{x\in X} \cup \{\coi(W, \iota', U)\}$ is coherent, with $\iota'$ having nonempty domain and range if $U$ is nonempty.

\end{enumerate}

\end{lemma}

\begin{proof}  We prove claim (1), and claim (2) follows in precisely the same way as the first claim, given the symmetric nature of the definition of coi triples.

If $W$ is empty then we let $U$ and $\iota$ be empty, the check that the larger coi collection is coherent is easy.  Else we write 

\begin{center}

$W \equiv W_0 \cdots W_k$

\end{center}

\noindent so that for each $0 \leq r \leq k$ we have $W_r$ a subword of an element in $\{W_x\}_{x\in X} \cup \{W_x^{-1}\}_{x \in X}$, or $|\overline{W_r}| = 1$,  as in Lemma \ref{nicefine}.  Let $J = \{r_0, \ldots, r_s\}$ be the set of those elements in $0 \leq r \leq k$ such that $W_{r_j}$ is infinite, with $r_0 < r_1 < \cdots < r_s$.   Select $x_j \in X$ and $\delta_j \in \{-1, 1\}$ and intervals $I_j \subseteq \overline{W_{x_j}}$ such that $W_{r_j} \equiv (W_{x_j} \upharpoonright I_j)^{\delta_j}$.  For each $r_j \in J$ we let $U_j' \equiv (U_{x_j} \upharpoonright \alpha(I_j, \iota_{x_j}))^{\delta_j}$.  By Lemma \ref{coilemma2} we know that for each $0 \leq j \leq s$ the word $U_j'$ is infinite.

If $J = \emptyset$ then the word $W$ is finite, and we select $h \in K_0\setminus \{1\}$, let $U \equiv h$, select $\lambda \in \overline{W}$ and let $\iota: \{\lambda\} \rightarrow \overline{U}$ be the unique bijection.  Again, the check that the extended coi collection is coherent is straightforward.

Now we may assume that $J \neq \emptyset$.  For each $0 \leq j < s$, if $\max\overline{U_j'}$ and $\min\overline{U_{j+1}'}$ both exist and $d(U_j'(\max\overline{U_j'})) = d(U_{j+1}(\min\overline{U_{j+1}'}))$ then we select $$h_j \in K_{d(U_j'\max\overline{U_j'}) + 1} \setminus \{1\}$$ and let $U_j \equiv U_j'h_j$, otherwise we let $U_j \equiv U_j'$.  By Lemma \ref{Eda} it is clear that the word $U_0$ is reduced, that $U_0U_1$ is reduced, that $U_0U_1U_2$ is reduced, etc., and that $U_0\cdots U_s$ is reduced.  We let $U \equiv U_0 \cdots U_s$.

We define a coi $\iota$ from $W$ to $U$.  For each $0 \leq j \leq s$ we let $f_j: \overline{W_{r_j}} \rightarrow I_j^{\delta_j}$ be the order isomorphism witnessing $W_{r_j} \equiv (W_{x_j} \upharpoonright I_j)^{\delta_j}$ and $k_j: \overline{U_j'} \rightarrow \alpha(I_j, \iota_{x_j})^{\delta_j}$ be the order isomorphism witnessing $U_j' \equiv (U_{x_j} \upharpoonright \alpha(I_j, \iota_{x_j}))^{\delta_j}$.  Let the domain of $\iota$ be given by $\dom(\iota) = \bigcup_{j=0}^s f_j^{-1}(\dom(\iota_{x_j}) \cap I_j)$, the range of $\iota$ be given by $\bigcup_{j=0}^s k_j^{-1} \circ \iota_{x_j}(\dom(\iota_{x_j}) \cap I_j)$, and $\iota$ be the unique function which extends $ k_j^{-1} \circ \iota_{x_j} \circ f_j$ for each $0 \leq j \leq s$.  By Lemma \ref{basiccloseproperties} (iii) we have that the domain (respectively range) of $\iota$ is close in $\overline{W}$ (resp. $\overline{U}$).

We check that $\{\coi(W_x, \iota_x, U_x)\}_{x\in X} \cup \{\coi(W, \iota, U)\}$ is coherent.  Suppose that $y\in X$ and intervals $I \subseteq \overline{W}$ and $I' \subseteq \overline{W_y}$ and $\delta \in \{-1, 1\}$ are such that $W \upharpoonright I  \equiv (W_y \upharpoonright I')^{\delta}$.  Let $L \subseteq J$ denote the set of those $r_j$ such that $\overline{W_{r_j}} \cap I \neq \emptyset$.  For each $r_j \in L \cap J$ we have $W\upharpoonright (\overline{W_{r_j}} \cap I) \equiv (W_{x_j}\upharpoonright \Lambda_j^*)^{\delta_j}$ for the obvious choice of interval $\Lambda_j^* \subseteq I_j \subseteq \overline{W_{x_j}}$.  Thus $(W_{x_j}\upharpoonright \Lambda_j^*)^{\delta \cdot \delta_j} \equiv W_y \upharpoonright I_j'$ for the obvious choice of interval $I_j' \subseteq I'$.  By the coherence of $\{\coi(W_x, \iota_x, U_x)\}_{x\in X}$ we therefore have

$$
\begin{array}{ll}
[[U \upharpoonright \alpha(I, \iota)]] & = \prod_{r_j\in L} [[U \upharpoonright \alpha(\overline{W_{r_j}} \cap I, \iota)]]\\
& = \prod_{r_j \in L} [[U_{x_j} \upharpoonright \alpha(\Lambda_j^*, \iota_{x_j})]]^{\delta_j}\\
& = \prod_{r_j \in L^{\epsilon}} [[U_{y}\upharpoonright \alpha(I_j', \iota_y)]]^{\delta}\\
& = [[(U_y \upharpoonright \alpha(I', \iota_y))^{\delta}]]
\end{array}
$$

\noindent where the first and last equalities hold by Lemma \ref{coilemma2}.

If we select intervals $I, I' \subseteq \overline{W}$ and $\delta \in \{-1, 1\}$ such that $W \upharpoonright I \equiv (W \upharpoonright I')^{\delta}$ then a similar strategy of finitely decomposing $I$ and $I'$ is employed to show $[[U\upharpoonright \alpha(I, \iota)]] = [[(U\upharpoonright \alpha(I', \iota))^{\delta}]]$.

The check that if $U \upharpoonright Q \equiv (U_z \upharpoonright Q')^{\epsilon}$, where $z\in X$, then the appropriate elements of $\mathcal{A}(\{G_n\}_{n \in \omega})$ are equal is similar to that above, with slight modifications (although $U \notin \Fine(\{U_x\}_{x\in X})$ is possible, the word $U$ is a finite concatenation of words in $\Fine(\{U_x\}_{x \in X})$ and words of length $1$).  Similarly modifications apply for intervals $Q, Q' \subseteq \overline{U}$, and the proof is complete.
\end{proof}

\begin{lemma}\label{makeitsmaller}  Suppose that $\{G_n\}_{n \in \omega}$ and $\{K_n\}_{n \in \omega}$ are sequences of nontrivial groups and that $\{\coi(W_x, \iota_x, U_x)\}_{x\in X}$ be a coherent collection of coi triples from $\Red(\{G_n\}_{n \in \omega})$ to $\Red(\{K_n\}_{n \in \omega})$.

\begin{enumerate}

\item  If $y \in X$ and $N\in \omega$ then there exists $U \in \Red(\{K_n\}_{n \in \omega})$ with $d(U) > N$ and coi $\iota$ from $W_y$ to $U$ such that $\{\coi(W_x, \iota_x, U_x)\}_{x\in X} \cup \{\coi(W_y, \iota, U)\}$ is coherent.  Also, the domain and codomain of $\iota$ may be chosen to be nonempty provided $\iota_y$  satisfies this property.

\item If $z \in X$ and $M \in \omega$ then there exists $W \in \Red(\{G_n\}_{n \in \omega})$ and coi $\iota''$ from $W$ to $U_z$ such that $d(W) > M$ and $\{\coi(W_x, \iota_x, U_x)\}_{x\in X} \cup \{\coi(W_z, \iota'', U)\}$ is coherent, with the domain of $\iota''$ being nonempty provided $\iota_z$ is.  

\end{enumerate}
\end{lemma}

\begin{proof}  Once again we will prove (1) and the proof of (2) follows the symmetric approach.  Assume the hypotheses.  Select $h \in K_{N+1} \setminus \{1\}$.  If $U_y$ is finite then so is $W_x$ by Lemma \ref{coilemma2}.  In this case we let $U \equiv h$, and if $W_y$ is nonempty we select $\lambda \in \overline{W_y}$ and let $\iota: \{\lambda\} \rightarrow \overline{U}$ be the unique function.  That $\{\coi(W_x, \iota_x, U_x)\}_{x\in X} \cup \{\coi(W_y, \iota, U)\}$ is coherent is straightforward.  Now we suppose that $U_y$ (and also $W_y$) is infinite.  If $i \in \overline{U_y}$ with $d(U_y(i)) \leq N + 2$ then there exists a maximal interval $i \in  I \subseteq \overline{U_y}$ such that $d(U_y(i')) \leq N+2$ for all $i' \in I$, and as $U_y$ is a word we know that $I$ is finite.  If $i \in \overline{U_y}$ with $d(U_y(i)) > N + 2$ then there exists a maximal interval $i \in I \subseteq \overline{U_y}$ such that $d(U_y(i')) > N + 2$ for all $i' \in I$ (one can argue this by Zorn's lemma).  Thus we may write $U_y$ as a finite concatenation

\begin{center}

$U_y \equiv U_0' \cdots U_k'$

\end{center}

\noindent where for each $0 \leq r \leq k$ we have either all letters in $U_r'$ having $d(\cdot) \leq N + 2$ or all letters having $d(\cdot) > N + 2$, and moreover if all letters of $U_r'$ have $d(\cdot) \leq N + 2$ then all letters in $U_{r + 1}'$ have $d(\cdot) > N + 2$ and vice versa.  Now we let

\[
U_r = \left\{
\begin{array}{ll}
h
                                            & \text{if } d(U_r') \leq N+2, \\
U_r'                                        & \text{if }d(U_r') > N + 2.
\end{array}
\right.
\]

\noindent and $U \equiv U_0U_1\cdots U_k$.  By Lemma \ref{Eda} it is clear that $U$ is reduced, for if $U_1, U_3, \ldots, U_{2p+1}$ are all $\equiv h$ then we have $U_0$ reduced, $U_0U_1 \equiv U_0h$ is reduced, $U_0U_1U_2 \equiv U_0hU_2$ is reduced, etc., and in case $U_0, U_2, \ldots$ are $\equiv h$ then the reasoning is similar.  

Let $J$ be the set of those $0 \leq r \leq k$ such that $U_r \not\equiv h$, so that $J$ consists either of all even or all odd natural numbers which are at least $0$ and at most $k$.  For each $r \in J$ we have $U_r$ as a subword of the original word $U_y$, so $\overline{U_r}$ is an interval in $\overline{U_y}$.  Let $\mathbb{J} \subseteq J$ be those $0 \leq r \leq k$ such that $U_r$ is infinite.  The word $U$ is obtained from $U_y$ by deleting finitely many finite words and inserting finitely many finite words, and as $U_y$ was infinite we know $\mathbb{J} \neq \emptyset$.  As $U_r$ is an infinite word for each $r \in \mathbb{J}$, by Lemma \ref{basiccloseproperties} (i) we know $\Close(\im(\iota_y) \cap \overline{U_r}, \overline{U_r})$ for such an $r$, and by Lemma \ref{basiccloseproperties} (iii) we have $\Close(\bigcup_{r \in J} (\im(\iota_y) \cap \overline{U_r}), \overline{U})$.  Now we define a coi $\iota$ from $W_y$ to $U$ by letting $\im(\iota) = \bigcup_{r \in J} \im(\iota_y) \cap \overline{U_r}$, $\dom(\iota) = \iota_y^{-1}(\bigcup_{r \in J} \im(\iota_y) \cap \overline{U_r})$ and $\iota(i) = \iota_y(i)$ for all $i \in \dom(\iota)$.  We have that $\dom(\iota_y) \setminus \dom(\iota)$ is finite by Lemma \ref{coilemma2}, so by Lemma \ref{basiccloseproperties} (i) we have $\Close(\dom(\iota), \overline{W_y})$.  Also, it is clear that $\iota$ is a bijection.

So, $\iota$ is a coi from $W_y$ to $U$, and $d(U) > N$.  If $I$ is an interval in $\overline{W_y}$, then one can obtain $U_y \upharpoonright \alpha(I, \iota_y)$ from $U \upharpoonright \alpha(I, \iota)$ by deleting finitely many finite subwords and inserting finitely many finite subwords, so $[[U_y \upharpoonright \alpha(I, \iota_y)]] = [[U \upharpoonright \alpha(I, \iota)]]$.  Thus, if we let $z \in X$ and intervals $I \subseteq \overline{W_y}$ and $I' \subseteq \overline{W_z}$ and $\delta \in \{-1, 1\}$ are such that $W_y \upharpoonright I \equiv (W_z \upharpoonright I')^{\delta}$, we have 

\begin{center}
$[[U \upharpoonright \alpha(I, \iota)]] = [[U_y \upharpoonright \alpha(I, \iota_y)]]$

$= [[(U_z \upharpoonright \alpha(I', \iota_z))^{\delta}]]$

\end{center}

\noindent where the second equality comes from the fact that the original coi collection $\{\coi(W_x, \iota_x, U_x)\}_{x\in X}$ is coherent.  The comparable proof shows that if $I, I'$ are intervals in $\overline{W_y}$ and $\delta \in \{-1, 1\}$ with $W_y \upharpoonright I \equiv (W_y \upharpoonright I')^{\delta}$ then we have

\begin{center}
$[[U \upharpoonright \alpha(I, \iota)]] = [[(U \upharpoonright \alpha(I', \iota))^{\delta}]]$.
\end{center}

Now we suppose that $z\in X$ and $I$ is an interval in $\overline{U}$, $I'$ is an interval in $\overline{U_z}$, and $\epsilon \in \{-1, 1\}$ are such that $U \upharpoonright I \equiv (U_z \upharpoonright I')^{\epsilon}$.  Let $L = \{0 \leq r \leq k: \overline{U_r} \cap I \neq \emptyset\}$ and notice that

$$
\begin{array}{ll}
[[W_y \upharpoonright \alpha(I, \iota^{-1})]] & = \prod_{r \in L} [[W_y \upharpoonright \alpha(\overline{U_r} \cap I, \iota^{-1})]]\\
& = \prod_{r \in L \cap J} [[W_y \upharpoonright \alpha(\overline{U_r} \cap I, \iota^{-1})]]\\
& = \prod_{r \in L \cap J} [[W_y \upharpoonright \alpha(\overline{U_r} \cap I, \iota_y^{-1})]]\\
& = \prod_{r \in (L \cap J)^{\epsilon}} [[(W_z \upharpoonright \alpha(K_r,\iota_z^{-1}))^{\epsilon}]]\\
& = [[(W_z \upharpoonright \alpha(I', \iota_z^{-1}))^{\epsilon}]]

\end{array}
$$

\noindent where $K_r$ is the subinterval in $I' \subseteq \overline{W_z}$ obtained from $\overline{U_r} \cap I$ via the order isomorphism witnessing $U \upharpoonright I \equiv (U_z \upharpoonright I')^{\epsilon}$.  The first equality is clear, the second follows from Lemma \ref{coilemma2}, and the third is because $\iota_y^{-1}$ and $\iota^{-1}$ coincide on $\overline{U_r} \cap I$ for each $r \in L \cap J$, the fourth is from the fact that $\{\coi(W_x, \iota_x, U_x)\}_{x\in X}$ is coherent, the fifth is by Lemma \ref{coilemma2}.  The similar claim when $I, I'$ are intervals in $\overline{U}$ and $\epsilon \in \{-1, 1\}$ are such that $U\upharpoonright I \equiv (U \upharpoonright I')^{\epsilon}$ follows via similar reasoning.  Thus $\{\coi(W_x, \iota_x, U_x)\}_{x\in X} \cup \{\coi(W_y, \iota, U)\}$ is coherent, and also in this case $\dom(\iota)$ is nonempty (as $\overline{W_y}$ is infinite).
\end{proof}

The following lemmas will be useful later.

\begin{lemma}\label{newword1}  Suppose that $\{H_n\}_{n \in \omega}$ is a sequence of nontrivial groups.  Let $\Lambda$ be a totally ordered set, $f: \Lambda \rightarrow \omega$ a function, and $V \in \W(\{H_n\}_{n \in \omega})$.  There are only finitely many order embeddings $p: \Lambda \rightarrow \overline{V}$ such that $p(\Lambda)$ is an interval and for all $\lambda \in \Lambda$ the equality $f(\lambda) = d(V(p(\lambda)))$ holds.
\end{lemma}

\begin{proof}  If $\Lambda$ is empty then the claim is obvious.  Thus we may assume $\Lambda \neq \emptyset$.  If there exists such an order embedding $p$, then we know that for each $n \in \omega$ the preimage $f^{-1}(\{n\})$ in $\Lambda$ is finite.  Thus suppose that $\Lambda \neq \emptyset$ and that such an order embedding exists.  Select $\lambda' \in \Lambda$.  We claim that if $p: \Lambda \rightarrow \overline{V}$ is an order embedding with $p(\Lambda)$ an interval in $\overline{V}$ and $f(\lambda) = d(V(p(\lambda)))$ for all $\lambda \in \Lambda$, then $p$ is totally determined by $p(\lambda')$.  To see this, we let $\lambda \in \Lambda$ be given.  Letting $n = d(V(p(\lambda))) = f(\lambda)$ we have that $f^{-1}(\{n\})$ is finite.  If $\lambda' \leq \lambda$ in $\Lambda$ then we let $\lambda_0 < \lambda_1 < \cdots < \lambda_k$ be the set of all those elements in $f^{-1}(\{n\})$ which are $\geq \lambda'$.  As $p(\Lambda)$ is an interval in $\overline{V}$, we know that $p(\lambda_0) \in \overline{V}$ is the minimal element greater than or equal to $p(\lambda')$ with $d(p(\lambda_0)) = n$, and $p(\lambda_1) \in \overline{V}$ is the minimal element above $p(\lambda_0)$ with $d(p(\lambda_1)) = n$, etc.  Thus if $\lambda = \lambda_i$, then we see that $p(\lambda)$ is uniquely determined by $p(\lambda')$, and the argument in case $\lambda < \lambda'$ is entirely analogous.

Now, since $\lambda'$ must be assigned to one of the finitely many elements $i \in \overline{V}$ such that $f(\lambda') = d(V(i))$, there are only finitely many order embeddings $p$ of $\Lambda$, as an interval, into $\overline{V}$ with $f(\lambda) = d(V(p(\lambda)))$ for all $\lambda \in \Lambda$.

\end{proof}

\begin{lemma}\label{newword2}  Let $\{H_n\}_{n \in \omega}$ be a collection of nontrivial groups.  Suppose that $\{V_x\}_{x \in X} \subseteq \Red(\{H_n\}_{n \in \omega})$, with $|X| < 2^{\aleph_0}$, $\Lambda$ is a totally ordered set and $f_0: \Lambda \rightarrow \omega$ is a function.  Let $f_1: \omega \rightarrow \Lambda$ be an injective function, $f_2: f_1(\omega) \rightarrow \{-1, 1\}$, and $f_3: f_1(\omega) \rightarrow \bigcup_{n \in \omega} H_n$ be such that $f_3(\lambda)$ is an element of infinite order in $H_{f_0(\lambda)}$.  Then there exists a function $q: f_1(\omega) \rightarrow \omega \setminus \{0\}$ such that there exists no $V \in \Sub(\{V_x\}_{x \in X} \cup \{V_x^{-1}\}_{x\in X})$ and order isomorphism $\iota: \Lambda \rightarrow \overline{V}$ such that

\begin{enumerate}[(a)]

\item $d \circ V \circ \iota = f_0$; and

\item $V(\iota(\lambda)) = f_3(\lambda)^{f_2(\lambda)q(\lambda)}$ for all $\lambda \in f_1(\omega)$.
\end{enumerate}
\end{lemma}

\begin{proof}

We let $\{\iota_y\}_{y \in Y}$ be the collection of all order embeddings with domain $\Lambda$ and codomain an element $\overline{V_y} \in \{\overline{V_x}\}_{x \in X} \cup \{\overline{V_x}^{-1}\}_{x \in X}$ such that $f_0 = d \circ V_y \circ \iota_y$ and also $\iota_y(\Lambda)$ is an interval in $\overline{V_y}$.  We also assume that the indexing $Y$ does not include any duplicates, so $y_0 \neq y_1$ implies that $\iota_{y_0} \neq \iota_{y_1}$.  By Lemma \ref{newword1} we know that for a fixed $x\in X$ there can be only finitely many $\iota_y$ with codomain $\overline{V_x}$ or with codomain $\overline{V_x^{-1}}$.  Thus in particular we know that $|Y| < 2^{\aleph_0}$.  As the set of functions $q: f_1(\omega) \rightarrow \omega \setminus \{0\}$ is of cardinality $2^{\aleph_0}$, we select a $q$ such that $(\forall y \in Y)(\exists \lambda_y \in f_1(\omega))$  $V_y(\iota_y(\lambda_y)) \neq f_3(\lambda_y)^{f_2(\lambda_y)q(\lambda_y)}$.  It is clear that this $q$ satisfies the conclusion.
\end{proof}

\begin{lemma}\label{notdoneyet}  Let $\{H_n\}_{n \in \omega}$ be a collection of groups such that each $H_n$ has an element of infinite order.  If $\{V_x\}_{x \in X} \subseteq \Red(\{H_n\}_{n \in \omega})$ is a collection with $|X| < 2^{\aleph_0}$ then there exists a word $V' \in \Red(\{H_n\}_{n \in \omega}) \setminus \Fine(\{V_x\}_{x \in X})$.
\end{lemma}

\begin{proof}  Assume the hypotheses.  For each $n \in \omega$ select $h_n \in H_n$ of infinite order.  Write $\omega = \bigsqcup_{m \in \omega} Z_m$ where each $Z_m$ is infinite and define $z: \omega \rightarrow \omega$ by $n \in Z_{z(n)}$.  Let $f_{0, m}: \omega \cap [\min Z_m, \infty) \rightarrow \omega$ be the identity map.  For each $m \in \omega$ we let $f_{1, m}: \omega \rightarrow Z_m$ be a bijection.  For each $m \in \omega$ we let $f_{2, m}$ be the constant function with domain $Z_m = f_{1, m}(\omega)$ and output $1$.  For each $m \in \omega$ we let $f_{3, m}$ be the function with domain $Z_m = f_{1, m}(\omega)$ given by $f_{3, m}(n) = h_n$.

By Lemma \ref{newword2} (letting $\Lambda = \omega \cap  [\min Z_m, \infty)$ for each $m$) we select for each $m \in \omega$ a function $q_m: f_{1, m}(\omega) \rightarrow \omega \setminus \{0\}$ such that there exists no $V \in \Sub(\{V_x\}_{x \in X} \cup \{V_x^{-1}\}_{x\in X})$ and order isomorphism $\iota: \omega \cap [\min Z_m, \infty) \rightarrow \overline{V}$ such that

\begin{enumerate}[(a)]

\item $d \circ V \circ \iota = f_{0, m}$ on $\omega \cap [\min Z_m, \infty)$; and

\item $V(\iota(n)) = f_{3, m}(n)^{f_{2, m}(n)q_m(n)}$ for all $n \in f_{1, m}(\omega)$.
\end{enumerate}

\noindent Now we let $$V' \equiv h_0^{q_{z(0)}(0)}h_1^{q_{z(1)}(1)} \cdots .$$  It is easy to see that $V'$ is reduced (each finite initial word $h_0^{q_{z(0)}(0)}h_1^{q_{z(1)}(1)} \cdots h_k^{q_{z(k)}(k)}$ is reduced, and if $V'$ were not reduced there would exist a nonempty reduction scheme on $V'$ by Proposition \ref{reductionscheme}, but then a finite initial word would have a nonempty reduction scheme, contradiction).

Were it the case that $V' \in \Fine(\{V_x\}_{x \in X})$, we know by Lemma \ref{nicefine} that we can decompose $V'$ as a finite concatenation

\begin{center}
$V' \equiv V_0'V_1' \cdots V_k'$
\end{center}

\noindent and for each $0 \leq j \leq k$ at least one of the following holds:

\begin{itemize}
\item $V_j' \in \Sub(\{V_x\}_{x \in X} \cup \{V_x^{-1}\}_{x \in X})$ is nonempty;

\item $|\overline{V_j'}| = 1$ with $\Let(V_j') = \{h\}$ such that $g \in H_{n_j}\setminus \{1\}$ and $h$ is a product of elements in $H_{n_j} \cap \Let(\{V_x\}_{x \in X} \cup \{(V_x')^{-1}\}_{x \in X})$.

\end{itemize}

\noindent As $\overline{V'} = \omega$ we know $\overline{V_k'}$ is a nonempty terminal interval in $\omega$ (and therefore of order type $\omega$), and $V_0'\cdots V_{k-1}'$ is a finite word.  Therefore $V_k' \in \Sub(\{V_x\}_{x \in X} \cup \{V_x^{-1}\}_{x \in X})$, say $V_k' \in \Sub(V_{\underline{x}}^{\epsilon})$, with $\epsilon \in \{-1,1\}$.  As $\overline{V_k'}$ is a nonempty terminal interval in $\omega$ we can select some $m \in \omega$ so that $[\min Z_m, \infty)\cap \omega \subseteq \overline{V_k'}$.  As $V_k' \in \Sub(V_{\underline{x}}^{\epsilon})$, we know that $V = V_k' \upharpoonright ([\min Z_m, \infty) \cap \omega) \in \Sub(V_{\underline{x}}^{\epsilon})$, which contradicts the conditions for our selection of $g_m$.

\end{proof}

\begin{proposition}\label{omegatypeconcat}  Suppose that $\{G_n\}_{n \in \omega}$ and $\{K_n\}_{n \in \omega}$ are sequences of groups, each group having an element of infinite order.  Suppose that $\{\coi(W_x, \iota_x, U_x)\}_{x \in X}$ is a coherent collection of coi from $\Red(\{G_n\}_{n \in \omega})$ to $\Red(\{K_n\}_{n \in \omega})$ with $|X|<2^{\aleph_0}$.

\begin{enumerate}

\item If $W \equiv \prod_{m \in \omega} W_m$ with $W_m \not\equiv E$ and $W_m \in\Fine(\{W_x\}_{x \in X})$ for each $m \in \omega$, and $W \notin \Fine(\{W_x\}_{x \in X})$ then there exists $U \in \Red(\{K_n\}_{n \in \omega})$ and coi $\iota$ from $W$ to $U$ with $\{\coi(W_x, \iota_x, U_x)\}_{x \in X} \cup \{\coi(W, \iota, U)\}$ coherent.

\item If $U \equiv \prod_{n \in \omega} U_n$ with $U_m \not\equiv E$ and $U_m \in \Fine(\{U_x\}_{x \in X})$ for each $m \in \omega$, and $U \notin \Fine(\{U_x\}_{x \in X})$ then there exists $W \in \Red(\{W_x\}_{x \in X})$ and coi $\iota$ from $W$ to $U$ with $\{\coi(W_x, \iota_x, U_x)\}_{x \in X} \cup \{\coi(W, \iota, U)\}$ coherent.

\end{enumerate}

\end{proposition}

\begin{proof}  As usual, we will prove only (1) since (2) is analogous.  Select, by Lemmas \ref{findsomerepresentative} and \ref{makeitsmaller} a nonempty word $U_0' \in \Red(\{K_n\}_{n \in \omega})$ with $d(U_0') > 1$ and a coi $\iota_0$ from $W_0$ to $U_0'$ such that $\{\coi(W_x, \iota_x, U_x)\}_{x \in X} \cup \{\coi(W_0, \iota_0, U_0')\}$ is coherent and $\iota_0$ has nonempty domain.  Suppose that we have already selected nonempty words $U_0', \ldots, U_k' \in \Red(\{K_n\}_{n \in \omega})$ with $d(U_j) > j + 1$ and $\iota_0, \ldots, \iota_k$ with $\iota_j$ a coi from $W_j$ to $U_j'$ with $\iota_j$ having nonempty domain, and $\{\coi(W_x, \iota_x, U_x)\}_{x \in X} \cup \{\coi(W_0, \iota_0, U_0'), \ldots, \coi(W_k, \iota_k, U_k')\}$ coherent.  Then by Lemmas \ref{findsomerepresentative} and \ref{makeitsmaller} select a nonempty word $U_{k+1}' \in \Red(\{K_n\}_{n \in \omega})$ with $d(U_{k+1}') >  k + 1$ and a coi $\iota_{k+1}$ from $W_{k+1}$ to $U_{k+1}'$ with $\iota_{k+1}$ having nonempty domain and $\{\coi(W_x, \iota_x, U_x)\}_{x \in X} \cup \{\coi(W_0, \iota_0, U_0'), \ldots, \coi(W_{k+1}, \iota_{k+1}, U_{k+1}')\}$ coherent.  By Lemma \ref{ascendingchaincoi} we have that $\{\coi(W_x, \iota_x, U_x)\}_{x \in X} \cup \{\coi(W_n, \iota_n, U_n')\}_{n \in \omega}$ is coherent.  For each $n \in \omega$ select an element $k_n \in H_n$ of infinite order.

As in Lemma \ref{notdoneyet}, we write $\omega = \bigsqcup_{m \in \omega} Z_m$ where each $Z_m$ is infinite and define $z: \omega \rightarrow \omega$ by $n \in Z_{z(n)}$.  Consider a word of form

\begin{center}

$U \equiv U_0'k_0^{q_{z(0)}(0)} U_1'k_1^{q_{z(1)}(1)}U_2k_2^{q_{z(2)}(2)} \cdots$

\end{center}

\noindent where $\{q_m\}_{m\in \omega}$ is a sequence of functions with $q_m: Z_m \rightarrow \omega \setminus\{0\}$.  It is easy to see that $U_0'k_0^{q_{z(0)}(0)}$ is reduced, since $d(U_0') > 0$, and similarly that $U_0'k_0^{q_{z(0)}(0)} U_1'$, $U_0'k_0^{q_{z(0)}(0)} U_1'k_1^{q_{z(1)}(1)}$, etc. are all reduced.  Therefore by applying Proposition \ref{reductionscheme} it is clear that such a $U$ is reduced (if $U$ is not reduced, there is a nonempty reduction scheme $\mathcal{S}$, from which we can find a nonempty reduction scheme $\mathcal{S}' \subseteq \mathcal{S}$ on an initial $U_0'k_0^{q_{z(0)}(0)}\cdots U_m'k_m^{q_{z(m)}(0)})$.  Let $\mathcal{F}$ denote the group of reduced words generated by $\{k_0, k_1, \ldots\}$.  Now by Lemma \ref{newword2} for each $m\in \omega$ we select the function $q_m: Z_m \rightarrow \omega \setminus \{0\}$ so that a word of form $$k_{\min Z_m}^{q_m(\min Z_m)}U_{(\min Z_m) + 1}'k_{(\min Z_m) + 1}^{q_{z((\min Z_m) +1)}((\min Z_m) + 1)}U_{(\min Z_m) + 2}'k_{(\min Z_m) + 2}^{q_{z((\min Z_m) + 2)}((\min Z_m) + 2)}\cdots$$ does not occur in $\Sub(\{U_x\}_{x \in X} \cup \{U_x^{-1}\}_{x \in X} \cup \{U_n'\}_{n \in \omega} \cup \{(U_n')^{-1}\}_{n \in \omega} \cup \mathcal{F})$.

Now we have produced a concrete word $U$ and to simplify notation we set $q_{z(n)}(n) = r_n$, so that $U \equiv U_0'k_0^{r_0}U_1'k_1^{r_1}\cdots$.  By arguing as in Lemma \ref{notdoneyet} we get that no nonempty terminal subword of $U$ is an element in $\Fine(\{U_x\}_{x\in X} \cup \{U_n'\}_{n \in \omega} \cup \mathcal{F})$.  Define $\iota$ to be the union $\iota = \bigcup_{m \in \omega}\iota_m$, so $\dom(\iota) = \bigcup_{m \in \omega}\dom(\iota_m)$ and $\im(\iota) = \bigcup_{m \in \omega}\im(\iota_m)$.  It is easy to see that $\iota$ is an order isomorphism.  We get by Lemma \ref{basiccloseproperties} (iii) that $\Close(\dom(\iota), \overline{W})$ and $\Close(\im(\iota), \overline{U})$, and so $\iota$ is a coi from $W$ to $U$.  We will prove that the collection $\{\coi(W_x, \iota_x, U_x)\}_{x \in X} \cup \{\coi(W_n, \iota_n, U_n')\}_{n \in \omega} \cup \{\coi(W, \iota, U)\}$ is coherent, from which the coherence of $\{\coi(W_x, \iota_x, U_x)\}_{x \in X} \cup \{\coi(W, \iota, U)\}$ is immediate.

We point out that every proper initial subword of $W$ is an element in the subgroup $\Fine(\{W_x\}_{x \in X})$, for if $W \equiv V_0V_1$ with $V_1$ nonempty, we can select $\lambda \in \overline{V_1}$, and $\lambda \in \overline{W_N}$ for some $N \in \omega$, so $\overline{V_0} \subseteq \bigcup_{n = 0}^{N-1}\overline{W_n}$, so $V_0$ is a subword of $\prod_{n = 0}^{N-1}W_n$, so $V_0$ is in $\Fine(\{W_x\}_{x \in X})$, since $W_0, \ldots, W_{N-1} \in \Fine(\{W_x\}_{x \in X})$.  Next, every nonempty terminal subword of $W$ is not in $\Fine(\{W_x\}_{x \in X})$, for if $W \equiv V_0V_1$ with $V_1 \not\equiv E$, we have $V_0 \in \Fine(\{W_x\}_{x \in X})$, so that if $V_1 \in \Fine(\{W_x\}_{x \in X})$ we get $W \in \Fine(\{W_x\}_{x \in X})$, contradiction.  Thus, any nonempty subword $V$ of $W$ which is an element in $\Fine(\{W_x\}_{x \in X})$ has $\overline{V} \subseteq \bigcup_{n = 0}^{N-1}\overline{W_n}$, for otherwise we get that $V$ is a nonempty terminal subword.

Let $y \in X \cup \omega$ and intervals $I \subseteq \overline{W}$ and $I' \subseteq \overline{W_y}$ and $\delta \in \{-1, 1\}$ such that $W\upharpoonright I \equiv (W_y \upharpoonright I')^{\delta}$.  Then in particular $W \upharpoonright I \in \Fine(\{W_x\}_{x \in X} \cup \{W_n\}_{n \in \omega}) = \Fine(\{W_x\}_{x \in X})$ (this latter equality holds because each $W_n \in \Fine(\{W_x\}_{x \in X})$).  Then $I$ is not a terminal interval in $\overline{W}$, and so $I \subseteq \bigcup_{n = 0}^{N-1}\overline{W_n}$.  Then $[[U \upharpoonright \alpha(I, \iota)]] = [[(U_y \upharpoonright \alpha(I', \iota_y))^{\delta}]]$, as by the proof of Lemma \ref{findsomerepresentative} we know that the collection 

\begin{center}
$\{\coi(W_x, \iota_x, U_x)\}_{x \in X} \cup \{\coi(W_n, \iota_n, U_n')\}_{n \in \omega}$

 $\cup \{\coi(\prod_{n = 0}^NW_n, \bigcup_{n}^N \iota_n, U_0'k_0^{r_0}\cdots U_N'k_N^{r_N})\}$
\end{center}

\noindent is coherent, and

$$
\begin{array}{ll}
[[U \upharpoonright \alpha(I, \iota)]] & = [[U_0'k_0^{r_0}\cdots U_N'k_N^{r_N} \upharpoonright \alpha(I, \bigcup_{n}^N \iota_n)]]\\
& = [[(U_y \upharpoonright \alpha(I', \iota_y))^{\delta}]]

\end{array}
$$

\noindent where the first equality holds because $U \upharpoonright \alpha(I, \iota) \equiv U_0'k_0^{r_0}\cdots U_N'k_N^{r_N} \upharpoonright \alpha(I, \bigcup_{n}^N \iota_n)$.

Suppose now that $I$ and $I'$ are intervals in $\overline{W}$ and $\delta \in \{-1, 1\}$ such that $W \upharpoonright I \equiv (W \upharpoonright I')^{\delta}$.  Then $W \upharpoonright I \in \Fine(\{W_x\}_{x \in X})$ if and only if $I$ is not a nonempty terminal interval.  If $W \upharpoonright I \in \Fine(\{W_x\}_{x \in X})$ then of course $(W \upharpoonright I)^{\delta} \equiv W \upharpoonright I' \in \Fine(\{W_x\}_{x \in X})$ as well, so that $I'$ is also not a nonempty terminal interval.  Then we may select $N \in \omega$ which is large enough that $I, I' \subseteq \bigcup_{n = 0}^N\overline{W_n}$ and argue as in Lemma \ref{findsomerepresentative} that $[[U \upharpoonright \alpha(I, \iota)]] = [[(U \upharpoonright \alpha(I', \iota))^{\delta}]]$ since the collection

\begin{center}
$\{\coi(W_x, \iota_x, U_x)\}_{x \in X} \cup \{\coi(W_n, \iota_n, U_n')\}_{n \in \omega}$

 $\cup \{\coi(\prod_{n = 0}^NW_n, \bigcup_{n = 0}^N \iota_n, U_0'k_0^{r_0}\cdots U_N'k_N^{r_N})\}$
\end{center}

\noindent is coherent.  Thus we suppose that $I$ and $I'$ are nonempty terminal intervals.  If it were the case that $\delta = -1$, then select $\lambda_0 \in I$ and we have on one hand that $W \upharpoonright \{\lambda \in I: \lambda \leq \lambda_0\} \in \Fine(\{W_x\}_{x \in X})$ since $\{\lambda \in I: \lambda \leq \lambda_0\}$ is not a terminal interval in $\overline{W}$, but on the other hand $W \upharpoonright \{\lambda \in I: \lambda \leq \lambda_0\} \notin \Fine(\{W_x\}_{x \in X})$ since $(W \upharpoonright \{\lambda \in I: \lambda \leq \lambda_0\})^{-1}$ is a nonempty terminal subword of the terminal subword $W \upharpoonright I'$ of $\overline{W}$.  That is a contradiction.  Thus we know $\delta = 1$.  So $I$ and $I'$ are each nonempty terminal intervals in $\overline{W}$, hence either $I \subseteq I'$ or $I'\subseteq I$.  If without loss of generality we have $I' \subsetneq I$ we select $\lambda \in I \setminus I'$.  Let $\underline{\iota}: I \rightarrow I'$ be an order isomorphism witnessing $W\upharpoonright I \equiv W \upharpoonright I'$.  Then we have $\lambda < \underline{\iota}(\lambda) < \underline{\iota}(\underline{\iota}(\lambda))<\cdots$ and $W(\lambda) = W(\underline{\iota}(\lambda)) = W(\underline{\iota}(\underline{\iota}(\lambda))) = \cdots$.  Either $W(\lambda) = 1$, which is impossible since $W$ is reduced, or $W(\lambda) \neq 1$, say $W(\lambda) = g \in G_p \setminus \{1\}$, which is impossible since $W$ is a word and cannot attain infinitely many outputs in $G_p$.  Thus we know $I = I'$, from which we have $U\upharpoonright \alpha(I, \iota) \equiv U\upharpoonright \alpha(I', \iota)$ and $[[U\upharpoonright \alpha(I, \iota)]] = [[U\upharpoonright \alpha(I', \iota)]]$.

Next, we know that no nomepty terminal subword of $U$ is an element of $$\Fine(\{U_x\}_{x \in X} \cup \{U_n'\}_{n \in \omega} \cup \mathcal{F})$$ and therefore if interval $I \subseteq \overline{U}$ is such that $U \upharpoonright I \in \Fine(\{U_x\}_{x \in X} \cup \{U_n'\}_{n \in \omega} \cup \mathcal{F})$ there exists some $N \in \omega$ such that $I \subseteq \overline{\prod_{n = 0}^N U_n'k_n^{r_n}}$.  Thus if $y \in X \cup \omega$, $I$ and $I'$ are intervals in $\overline{U}$ and $\overline{U_y}$ respectively, and $\epsilon \in \{-1, 1\}$ with $U \upharpoonright I \equiv (U_y \upharpoonright I')^{\epsilon}$, we have $U \upharpoonright I \in \Fine(\{U_x\}_{x \in X})$, so $I \subseteq \overline{\prod_{n = 0}^N U_n'k_n^{r_n}}$ for some $N$.  Then $[[W \upharpoonright \alpha(I, \iota^{-1})]] = [[(W_y\upharpoonright \alpha(I', \iota_y^{-1}))^{\epsilon}]]$ follows as in the proof of Lemma \ref{findsomerepresentative} as the collection 

\begin{center}
$\{\coi(W_x, \iota_x, U_x)\}_{x \in X} \cup \{\coi(W_n, \iota_n, U_n')\}_{n \in \omega}$

 $\cup \{\coi(\prod_{n = 0}^NW_n, \bigcup_{n = 0}^N \iota_n , U_0'k_0^{r_0}\cdots U_N'k_N^{r_N})\}$
\end{center}

\noindent is coherent.

Finally, we know that all nonempty terminal subwords of $U$ are not elements of $\Fine(\{U_x\}_{x \in X} \cup \{U_n'\}_{n \in \omega} \cup \mathcal{F})$, but it is also clear that any proper initial subword of $U$ is an element of $\Fine(\{U_x\}_{x \in X} \cup \{U_n'\}_{n \in \omega} \cup \mathcal{F})$, so for any nonempty interval $I \subseteq \overline{U}$ we have $U \upharpoonright I \in \Fine(\{U_x\}_{x \in X} \cup \{U_n'\}_{n \in \omega} \cup \mathcal{F})$ if and only if $I$ is not terminal.

Suppose that $I, I'$ are intervals in $\overline{U}$ and $\epsilon \in \{-1, 1\}$ is such that $U \upharpoonright I \equiv (U \upharpoonright I')^{\epsilon}$.  If $U \upharpoonright I \in \Fine(\{U_x\}_{x \in X} \cup \{U_n'\}_{n \in \omega} \cup \mathcal{F})$ then both $I$ and $I$ are empty or nonterminal, and we argue as before that $[[W \upharpoonright \alpha(I, \iota^{-1})]] = [[(W \upharpoonright \alpha(I', \iota^{-1}))^{\epsilon}]]$ using the fact that 

\begin{center}
$\{\coi(W_x, \iota_x, U_x)\}_{x \in X} \cup \{\coi(W_n, \iota_n, U_n')\}_{n \in \omega}$

 $\cup \{\coi(\prod_{n = 0}^NW_n, \bigcup_{n = 0}^N \iota_n , U_0'k_0^{r_0}\cdots U_N'k_N^{r_N})\}$
\end{center}

\noindent is coherent for an appropriate value of $N$.  If $U \upharpoonright I \notin \Fine(\{U_x\}_{x \in X} \cup \{U_n'\}_{n \in \omega} \cup \mathcal{F})$ then both $I$ and $I'$ are nonempty terminal and as before we argue that $\epsilon = 1$ and in fact $I = I'$, so that $[[W \upharpoonright \alpha(I, \iota^{-1})]] = [[(W \upharpoonright \alpha(I', \iota^{-1}))^{\epsilon}]]$ is immediate.  The lemma is proved.

\end{proof}

\end{section}

\begin{section}{Extension to a $\mathbb{Q}$-type concatenation}\label{Qconcatenation}

In this section we shall prove only the following.

\begin{proposition}\label{Qtypeconcat}    Suppose that $\{G_n\}_{n \in \omega}$ and $\{K_n\}_{n \in \omega}$ are collections of groups without elements of order $2$ such that each group has an element of infinite order.  Suppose also that $\{\coi(W_x, \iota_x, U_x)\}_{x \in X}$ is a coherent collection of coi triples from $\Red(\{G_n\}_{n \in \omega})$ to $\Red(\{K_n\}_{n \in \omega})$ with $|X|<2^{\aleph_0}$.

\begin{enumerate}

\item Let $W \in \Red(\{G_n\}_{n \in \omega})$ be such that $\overline{W} \equiv \prod_{s \in \mathbb{Q}} I_s$ with each $I_s \neq \emptyset$, $W \upharpoonright I_s \in \Fine(\{W_x\}_{x\in X})$ for each $s \in \mathbb{Q}$, and $W \upharpoonright \bigcup_{s \in \Lambda} I_s \notin \Fine(\{W_x\}_{x\in X})$ for each interval $\Lambda \subseteq \mathbb{Q}$ with more than one point.  Then there exists $U \in \Red(\{K_n\}_{n \in \omega})$ and coi $\iota$ from $W$ to $U$ such that $\{\coi(W_x, \iota_x, U_x)\}_{x\in X} \cup \{\coi(W, \iota, U)\}$ is coherent.

\item Let $U \in \Red(\{K_n\}_{n \in \omega})$ be such that $\overline{U} \equiv \prod_{s \in \mathbb{Q}} I_s$ with each $I_s \neq \emptyset$, $U \upharpoonright I_s \in \Fine(\{U_x\}_{x\in X})$ for each $s \in \mathbb{Q}$, and $U \upharpoonright \bigcup_{s \in \Lambda} I_s \notin \Fine(\{U_x\}_{x\in X})$ for each interval $\Lambda \subseteq \mathbb{Q}$ with more than one point.  Then there exists $W \in \Red(\{G_n\}_{n \in \omega})$ and coi $\iota$ from $W$ to $U$ such that $\{\coi(W_x, \iota_x, U_x)\}_{x\in X} \cup \{\coi(W, \iota, U)\}$ is coherent.

\end{enumerate}

\end{proposition}

We let $\{W_m\}_{m \in \omega}$ be an enumeration such that for each $s \in \mathbb{Q}$ we have some $m\in \omega$ for which $W \upharpoonright I_s \equiv W_m$ or $W \upharpoonright I_s \equiv W_m^{-1}$  (and notice that both cannot hold as there are no elements of order $2$ in the groups, by \cite[Corollary 1.6]{E}), and for distinct $m_0 \neq m_1$ in $\omega$ we have $W_{m_0} \not\equiv W_{m_1} \not\equiv W_{m_0}^{-1}$.  Such a list $\{W_m\}_{m \in \omega}$ must be infinite, otherwise by the pidgeonhole principle there is some $q' \in \mathbb{Q}$ such that $\{q \in \mathbb{Q} \mid W \upharpoonright I_q \equiv W \upharpoonright I_{q'}\}$ is infinite, which means $W$ is not a reduced word.  Now we define a function $P: \mathbb{Q} \rightarrow \omega$ be given by $P(s) = m$ where $W \upharpoonright I_s \in \{W_m, W_m^{-1}\}$ and function $A: \mathbb{Q} \rightarrow \{-1, 1\}$ by

\[
A(s) = \left\{
\begin{array}{ll}
1
                                            & \text{if }  W \upharpoonright I_s \equiv W_{P(s)}, \\
-1                                        & \text{if }W \upharpoonright I_s \equiv W_{P(s)}^{-1}.
\end{array}
\right.
\]

\noindent Thus we may write $W \equiv \prod_{s \in \mathbb{Q}} (W_{P(s)})^{A(s)}$.

Select nonempty $U_0' \in \Red(\{K_n\}_{n \in \omega})$ with $d(U_0') > 0$ and coi $\iota_0$ from $W_0$ to $U_0'$ with nonempty domain such that $\{\coi(W_x, \iota_x, U_x)\}_{x\in X} \cup \{\coi(W_0, \iota_0, U_0')\}$ is coherent by Lemmas \ref{findsomerepresentative} and \ref{makeitsmaller}.  Assuming we have chosen $U_m'$ and $\iota_m$ we select nonempty $U_{m + 1}' \in \Red(\{K_n\}_{n \in \omega})$ with $d(U_{m+1}') > m+1$ and coi $\iota_{m + 1}$ from $W_{m+1}$ to $U_{m+1}'$, the domain of $\iota_{m + 1}$ nonempty, and $\{\coi(W_x, \iota_x, U_x)\}_{x\in X} \cup \{\coi(W_j, \iota_j, U_j')\}_{j = 0}^{m+1}$ coherent by Lemmas \ref{findsomerepresentative} and \ref{makeitsmaller}.  By Lemma \ref{ascendingchaincoi} the collection $$\{\coi(W_x, \iota_x, U_x)\}_{x\in X} \cup \{\coi(W_m, \iota_m, U_m')\}_{m \in \omega}$$ is coherent.

For each $m \in \omega$ select $h_m \in K_m$ of infinite order.  For $s\in \mathbb{Q}$ write $U_s \equiv h_{P(s)}^{A(s)r_s}(U_{P(s)}')^{A(s)}h_{P(s)}^{A(s)r_s}$, where the number $r_s \in \omega \setminus \{0\}$ has yet to be determined.  We consider a word of form

\begin{center}

$U \equiv \prod_{s \in \mathbb{Q}} U_s$.

\end{center}

\noindent Thus the totally ordered set $\overline{U}$ is known, the function $d \circ U$ is known, and many values of $U$ are known, and such a $U$ will be totally defined once we have determined the values of the $r_s$.

The set of nonempty open intervals in $\mathbb{Q}$ with rational endpoints is countable, so let $\{J_j\}_{j \in \omega}$ be an enumeration of this set.  Let $L: \omega \rightarrow \omega \times \omega$ be a bijection, with $L(k) = (L_0(k), L_1(k))$, and inductively define $\delta(k) = \min (\{P(s):s \in J_{L_0(k)}\} \setminus \{\delta(0), \ldots, \delta(k-1)\})$.  For each $j \in \omega$ we let $Z_j = \delta(L^{-1}(\{j\} \times \omega))$.  Then for every $m \in Z_j$ there exists $s \in J_j$ with $P(s) = m$.  Also, $\omega = \bigsqcup_{j \in \omega} Z_j$, and $|Z_j| = \aleph_0$ for each $j \in \omega$.  

For a fixed $j \in \omega$ we let $v_{0, j}: \omega \rightarrow Z_j$ be a bijection, and let $v_{1, j}: \omega \rightarrow J_j$ be a function such that $P(v_{1, j}(k)) = v_{0, j}(k)$ (and we note that $v_{1, j}$ is an injection).  Let $f_{1, j}: \omega \rightarrow \bigcup_{s \in J_j} \overline{U_{s}} \subseteq \overline{U}$ be the function given by $f_{1, j}(k) = \min \overline{U_{v_{1, j}(k)}}$, so $f_{1, j}$ is also an injection.  Now by Lemma \ref{newword2} we select a function $q_j: f_{1, j}(\omega) \rightarrow \omega \setminus \{0\}$ so that there is no word $V \in \Sub(\{U_x\}_{x \in X} \cup \{U_x^{-1}\}_{x \in X} \cup \{U_m'\}_{m \in \omega} \cup \{(U_m')^{-1}\}_{m \in \omega})$ with domain $\overline{V}$ which is order isomorphic, via some $\iota$, to $\overline{\prod_{s \in J_j} U_s}$ with $d \circ V \circ \iota = d \circ U\upharpoonright  \overline{\prod_{s \in J_j} U_s}$ and such that $V(\iota(f_{1, j}(k))) = h_{v_{0, j}(k)}^{A(v_{1, j}(k))q_j(k)}$ for all $k \in \omega$.  Note that for arbitrary $s \in \mathbb{Q}$ there is a unique $j \in \omega$ such that $P(s) \in J_j$ and unique $k \in \omega$ such that $v_{0, j}(k) = P(s)$, and we let $r_s = q_j(k)$.  For convenience we let $R: \omega \rightarrow \omega$ be the function such that for each $s \in \mathbb{Q}$ we have $$U_s \equiv h_{P(s)}^{A(s)R(P(s))}(U_{P(s)}')^{A(s)}h_{P(s)}^{A(s)R(P(s))}.$$Now we have fully determined the function $U$. 

Notice that $U$ is indeed a word, for it is a concatenation $U \equiv \prod_{s \in \mathbb{Q}} U_s$ such that for each $N \in \omega$, the set $\{s \in \mathbb{Q}: d(U_s) \leq N\}$ is finite.  We prove that $U$ is reduced.  We give some observations.  We have for each $s \in \mathbb{Q}$ that $U_s \equiv h_{P(s)}^{A(s)R(P(s))}(U_{P(s)}')^{A(s)}h_{P(s)}^{A(s)R(P(s))}$ is reduced, since $d(U_{P(s)}') > d(h_{P(s)}^{A(s)R(P(s))}) = P(s)$, using Lemma \ref{Eda}.  Also, $d(U_s) = P(s)$ for each $s \in \mathbb{Q}$.  We also point out that $U_s \equiv (U_{s'})^{-1}$ if and only if $W_s \equiv (W_{s'})^{-1}$.  Also, $U_s \equiv (U_{s'})^{-1}$ if and only if $P(s) = P(s')$ and $A(s) = -A(s')$.  Furthermore we know that for $i, i' \in \overline{U}$, if $i$ is adjacent to $i'$ then there is some $s \in \mathbb{Q}$ such that $i, i' \in \overline{U_s}$.  This is clear since if $i \in \overline{U_s}$ and $i' \in \overline{U_{s'}}$ and without loss of generality $s < s'$ then the open interval $(s, s') \subseteq \mathbb{Q}$ is infinite and so we select $s'' \in (s, s')$ and note that $U_{s''}$ is not an empty word, so letting $i'' \in \overline{U_{s''}}$ we have $i < i'' < i'$.

Now we must check that our word $U$ is indeed reduced.  In the appendix we prove a more general fact (Theorem \ref{intheappendix}) which allows involutions in the groups.  The proof is kinder because the statements are more general and there are fewer exponents floating around.

\begin{lemma}\label{Qwordreduced}  The word $U$ is reduced.

\end{lemma}

\begin{proof}  Suppose for contradiction that $U$ is not reduced.  We point out that there cannot be adjacent elements $i_0, i_1 \in \overline{U}$ such that $U(i_0)$ and $U(i_1)$ are in the same group $G_N$, since each of the words $U_s$ is nonempty reduced and $\mathbb{Q}$ is order dense.  Then as $U$ is not reduced, there exists a nonempty interval $I \subseteq \overline{U}$ such that $U \upharpoonright I \sim E$.  Then by Proposition \ref{reductionscheme} (1) we have a reduction scheme $\mathcal{S}$ on $U \upharpoonright I$ such that $\bigcup_{C \in \mathcal{S}} \set(C) = I$ and $\pi(U \upharpoonright I, C) \equiv E$ for all $C \in \mathcal{S}$.

Let $N_0' = \min\{P(s):(\exists s \in \mathbb{Q}) \overline{U_s} \cap I \neq \emptyset\}$.  We claim that there exists a component $C = (i_0; \ldots; i_k) \in \mathcal{S}$ such that $d(C) = N_0'$ and for each $0 \leq j \leq k$ we have $U(i_j) \in \{h_{N_0'}^{R(N_0')}, h_{N_0'}^{-R(N_0')}\}$.  To see this, take $s' \in \mathbb{Q}$ such that $\overline{U_{s'}} \cap I \neq \emptyset$ and $P(s') = N_0'$.  Take $C' = (i_0'; \ldots; i_m') \in \mathcal{S}$ such that $\set(C') \cap \overline{U_{s'}} \neq \emptyset$, say $0 \leq j' \leq m$ has $i_{j'} \in \set(C') \cap \overline{U_{s'}}$.  If $d(C') = N_0'$ we will take $C = C'$, so suppose that this is not the case.  As $\pi(U \upharpoonright I, C) = E$ and the word $U$ is simple, we know $m > 0$.  Let without loss of generality $j' + 1 \leq m$ (otherwise $0 \leq j' - 1$ and the proof will be similar).  As $U_{s'}$ is reduced, we know $i_{j' + 1} \notin \overline{U_{s'}}$, say $i_{j' + 1} \in \overline{U_{s''}}$ and $s' < s''$ in $\mathbb{Q}$.  As $\mathcal{S}$ is a reduction scheme and $i_{j'} < \max(\overline{U_{s'}}) < i_{j' + 1}$ there exists $C = (i_0; \ldots; i_k) \in \mathcal{S}$ such that $\max(\overline{U_{s'}}) = i_j \in \set(C)$.  As $P(s') = N_0'$ we know $U(i_j) \in \{h_{N_0'}^{R(N_0')}, h_{N_0'}^{-R(N_0')}\}$ so in particular $d(C) = N_0'$.  Thus in either case we have found a component $C \in \mathcal{S}$ with $d(C) = N_0'$ and $\set(C) \cap \overline{U_{s'}} \neq \emptyset$.  Letting $J = \{s \in \mathbb{Q} \mid \set(C) \cap \overline{U_s} \neq \emptyset\}$, by minimality of $N_0'$ we know that $P(s) = N_0'$ for each $s \in J$.  Then as $d(C) = N_0'$, each element of $\set(C)$ is either a $\max(\overline{U_s})$ or a $\min(\overline{U_s})$ for some $s \in J$.  Therefore $U(i_j) \in \{h_{N_0'}^{R(N_0')}, h_{N_0'}^{-R(N_0')}\}$ for each $i_j \in \set(C)$.

As $\pi(U, C) \equiv E$ we have that there exist some $0 \leq j \leq k$ for which $U(i_j) = h_{N_0'}^{R(N_0')}$ and there also exist some $0 \leq j \leq k$ for which $U(i_j) = h_{N_0'}^{-R(N_0')}$.  Then there exists some $0 \leq j' \leq k$ for which $U(i_{j'}) = (U(i_{j'+1}))^{-1}$.  Then the reduction scheme $\{C' \in \mathcal{S}: \set(C') \cap (i_{j'}, i_{j'+1}) \neq \emptyset\}$ witnesses that $U \upharpoonright (i_{j'}, i_{j'+1}) \sim E$.  Letting $i_{j'} \in \overline{U_{s_0'}}$ and $i_{j'+1} \in \overline{U_{s_1'}}$ we have $P(s_0') = P(s_1') = N_0'$ and $i_{j'} \in \{\min\overline{U_{s_0'}}, \max\overline{U_{s_0'}}\}$ and similarly $i_{j'+1} \in \{\min\overline{U_{s_1'}}, \max\overline{U_{s_1'}}\}$.  If $i_{j'} = \min\overline{U_{s_0'}}$ then we claim that $i_{j'+1} = \max\overline{U_{s_1'}}$, for otherwise we have $(i_{j'}, i_{j'+1}) \cap \{i \in \overline{U}: d(U(i)) = N_0'\}$ is odd and so the word $p_{N_0'}(U \upharpoonright (i_{j'}, i_{j'+1}))$ does not represent the trivial element in $G_{N_0'}$, since $h_{N_0'}$ is of infinite order, a contradiction.  By the same token, if $i_{j'} = \max\overline{U_{s_0'}}$ then $i_{j'+1} = \min\overline{U_{s_1'}}$.  In either case, we see that $U \upharpoonright (\max\overline{U_{s_0'}}, \min\overline{U_{s_1'}}) \sim E$, for even if $i_{j'} = \min\overline{U_{s_0'}}$ and $i_{j'+1} = \max\overline{U_{s_1'}}$ we have that

\begin{center}

$E \sim (U_{s_0'}\upharpoonright (\overline{U_{s_0'}} \setminus \{\min\overline{U_{s_0'}}\}))^{-1} E (U_{s_0'}\upharpoonright (\overline{U_{s_0'}} \setminus \{\min\overline{U_{s_0'}}\}))$

$\equiv (U_{s_0'}\upharpoonright \overline{U_{s_0'}} \setminus \{\min\overline{U_{s_0'}}\})^{-1} E (U_{s_1'}\upharpoonright( \overline{U_{s_1'}} \setminus \{\max\overline{U_{s_1'}}\}))^{-1}$

$\sim  (U_{s_0'}\upharpoonright \overline{U_{s_0'}} \setminus \{\min\overline{U_{s_0'}}\})^{-1} (U\upharpoonright (i_{j'}, i_{j'+1})) (U_{s_1'}\upharpoonright( \overline{U_{s_1'}} \setminus \{\max\overline{U_{s_1'}}\}))^{-1}$

$\sim U\upharpoonright (\max\overline{U_{s_0'}}, \min\overline{U_{s_1'}})$.

\end{center}

\noindent  In particular we may replace the interval $I$ with the nonempty interval $$(\max\overline{U_{s_0'}}, \min\overline{U_{s_1'}})$$ and thus get that $I$ is an open interval such that $I = \bigcup_{I \cap \overline{U_s} \neq \emptyset} \overline{U_s}$, and also replace the old reduction scheme $\mathcal{S}$ with $\{C' \in \mathcal{S}: \set(C) \cap (\max\overline{U_{s_0'}}, \min\overline{U_{s_1'}}) \neq \emptyset\}$.  Henceforth in the proof we will therefore assume that $I = \bigcup_{I \cap \overline{U_s} \neq \emptyset} \overline{U_s}$ and that $\mathcal{S}$ is a reduction scheme on $U\upharpoonright I$ such that $\pi(U \upharpoonright I, C) \equiv E$ for all $C \in \mathcal{S}$ and $\bigcup_{C\in \mathcal{S}}\set(C) = I$.  Let $\mathcal{I} \subseteq \mathbb{Q}$ be the open interval $\{s \in \mathbb{Q}: I \cap \overline{U_s} \neq \emptyset\}$.  We let $Q: I \rightarrow \mathcal{I} \subseteq \mathbb{Q}$ be the surjective function defined by $Q(i)= s$ where $i \in \overline{U_s}$.

Let $\{N_0, N_1, \ldots\} = \{P(s):(\exists s \in \mathbb{Q}) \overline{U_s} \cap I \neq \emptyset\}$, with $N_k < N_{k+1}$, so in particular $N_0 = \min\{P(s):(\exists s \in \mathbb{Q}) \overline{U_s} \cap I \neq \emptyset\}$.  Suppose first that $C = (i_0; \ldots; i_k) \in \mathcal{S}$ and $s_0, s_1 \in \mathbb{Q}$ are such that $\overline{U_{s_0}} \cap \set(C) \neq \emptyset$ and $\overline{U_{s_1}} \cap \set(C) \neq \emptyset$ and $P(s_0) = N_0$.  We'll show that $P(s_1) = N_0$.  If this is not the case, there exist $i_l, i_{l+1} \in \set(C)$ with $i_l \in \overline{U_{s_0'}}$ and $i_{l+1} \in \overline{U_{s_1'}}$ such that either $P(s_0') = N_0$ and $P(s_1') > N_0$, or such that $P(s_0') > N_0$ and $P(s_1') = N_0$.  If without loss of generality $P(s_0') = N_0$ and $P(s_1') > N_0$, we know by Definition \ref{redschdef} condition (2) and Proposition \ref{reductionscheme} part (1) that $U \upharpoonright (i_l, i_{l+1}) \sim E$, however $(i_l, i_{l+1}) \cap \{i \in \overline{U}: d(U(i)) = N_0\}$ is odd and so the word $p_{N_0}(U \upharpoonright (i_l, i_{l+1}))$ does not represent the trivial element in $G_{N_0}$, since $h_{N_0}$ is of infinite order, a contradiction.  Now, suppose that it is the case that whenever $k \leq K$ and $C \in \mathcal{S}$ and $s_0, s_1 \in \mathbb{Q}$ are such that $\overline{U_{s_0}} \cap \set(C) \neq \emptyset$ and $\overline{U_{s_1}} \cap\set(C) \neq \emptyset$ and $P(s_0) = N_k$, then $P(s_1) = N_k$.  Let $C = (i_0; \ldots; i_p) \in \mathcal{S}$ and $s_0, s_1 \in \mathbb{Q}$ be such that $\overline{U_{s_0}} \cap \set(C) \neq \emptyset$ and $\overline{U_{s_1}} \cap\set(C) \neq \emptyset$ and $P(s_0) = N_{K+1}$.  We'll show $P(s_1) = N_{K+1}$.  If this is not the case then there exist $i_l, i_{l+1} \in \set(C)$ with $i_l \in \overline{U_{s_0'}}$ and $i_{l+1} \in \overline{U_{s_1'}}$ such that either $P(s_0') = N_{K+1}$ and $P(s_1') > N_{K+1}$, or such that $P(s_0') > N_{K+1}$ and $P(s_1') = N_{K+1}$.  Without loss of generality $P(s_0') = N_{K+1}$ and $P(s_1') > N_{K+1}$.  Letting $Y$ be the finite set $\{s \in \mathbb{Q}: P(s) \leq K \wedge \overline{U_s} \subseteq (i_l, i_{l+1})\}$ we have $U \upharpoonright ((i_0, i_1) \setminus \bigcup_{s \in Y} \overline{U_s}) \sim E$ as witnessed by the reduction scheme $\mathcal{S}' = \{C' \in \mathcal{S}: \set(C') \cap (i_l, i_{l+1}) \neq \emptyset \wedge \set(C) \cap \bigcup_{s \in Y}\overline{U_s} = \emptyset\}$ (here we are using the fact that our induction hypothesis implies that $\bigcup_{s \in Y}\overline{U_s} = \bigcup_{C' \in \mathcal{S}, \set(C') \cap \bigcup_{s \in Y}\overline{U_s} \neq \emptyset} \set(C')$).  However the set $M = (i_0, i_1) \cap \{i \in \overline{U}: d((U(i))) = N_{K+1}\} \setminus \bigcup_{s \in Y}\overline{U_s}$ is of odd cardinality, and  we have $U(i) \in \{h_{N_{K+1}}^{R(N_{K+1})}, h_{N_{K+1}}^{-R(N_{K+1})}\}$ for each $i \in M$, and since $h_{N_{K+1}}$ is of infinite order we get in particular that $p_{N_{K+1}}(U \upharpoonright (i_0, i_1) \setminus \bigcup_{s \in Y} \overline{U_s}) = \sum_{i \in M} U(i)$ is not trivial, contradiction.  What we have just shown is that for each $C \in \mathcal{S}$ the function $P\circ Q \upharpoonright \set(C)$ is constant.  We also know that for each $C \in \mathcal{S}$ the function $Q \upharpoonright \set(C)$ is injective, since each $U_s$ is reduced.

Now we make slight modifications to the scheme $\mathcal{S}$.  For $C = (i_0; \ldots; i_k) \in \mathcal{S}$ such that there exists $i \in \set(C)$ with $i \in \overline{U_s}$ and $d(C) = P(s)$, we have $i \in \{\min\overline{U_s}, \max\overline{U_s}\}$ and by the preceeding paragraph we have that each $i_j \in \set(C)$ has some $s_j \in \mathcal{I}$ with $i_j \in \{\min\overline{U_{s_j}}, \max\overline{U_{s_j}}\}$ and $d(C) = P(s_j)$.  More particularly we have that $U(i_j) \in \{h_{d(C)}^{R(d(C))}, h_{d(C)}^{-R(d(C))}\}$ for all $0 \leq j \leq k$.  Since $\pi(U, C) = 0$ we know $U(i_j) = h_{d(C)}^{R(d(C))}$ for some $j$ and $U(i_j) = h_{d(C)}^{-R(d(C))}$ for some other values of $j$.  Then for some $0\leq j' \leq k$ we have $U(i_{j'}) = (U(i_{j' +1}))^{-1}$, and we can replace $C$ in $\mathcal{S}$ with two components $(i_{j'}; i_{j' +1})$, $C' = (i_{0}; \ldots; i_{j'-1}; i_{j' + 2}; \ldots; i_{k})$.  If $|\set(i_{0}, \ldots, i_{j'-1}, i_{j' + 2}, \ldots, i_{k})| > 2$ then performing the same analysis on the finite sequence $(i_{0}; \ldots; i_{j'-1}; i_{j' + 2}; \ldots; i_{k})$ we produce two components $C', C''$ with $|\set(C')| = 2$ and $|\set(C'')|$ being of positive even cardinality.  By performing finitely many steps we determine that we can replace $C$ with $|\set(C)|/2$ components.  Thus we can assume that for each $s \in \mathcal{I}$ we have that $\min \overline{U_s} \in \set(C)$ and $C \in \mathcal{S}$ implies that $|\set(C)| = 2$ and similarly for $\max \overline{U_s}$.

Now, if $s \in \mathcal{I}$, $i = \min\overline{U_s}$, $i \in \set(C)$, $\{i'\} = \set(C) \setminus \{i\}$, with $i' \in \overline{U_{s'}}$ then $i' = \max\overline{U_{s'}}$.  To see this, we know of course that $i' \in \{\min\overline{U_{s'}}, \max\overline{U_{s'}}\}$, and for contradiction if $i' = \min\overline{U_{s'}}$ and say $i < i'$ then the word $V \equiv U \upharpoonright ((i, i') \setminus\bigcup_{s \in \mathcal{I}, P(s) < d(C)} \overline{U_s})$ is not $\sim E$ since $p_{d(C)}(V)$ is $h_{d(C)}$ raised to an odd power, on the other hand $p_{d(C)}(V) \sim E$ by Proposition \ref{reductionscheme} part (1) (using the reduction scheme $\{C' \in \mathcal{S}: \set(C') \cap (i, i') \neq \emptyset \wedge d(C') < d(C)\}$), contradiction.  A similar proof works when $i' < i$.  By similar reasoning if $s \in \mathcal{I}$, $i = \max\overline{U_s}$, $i \in \set(C)$, $\{i'\} = \set(C) \setminus \{i\}$, with $i' \in \overline{U_{s'}}$ then $i' = \min\overline{U_{s'}}$.

We define a collection $\mathcal{P}$ of ordered pairs of elements of $\mathcal{I}$.  Consider a finite sequence $D = (s_0; \ldots; s_k)$ such that 

\begin{center}

$(\max\overline{U_{s_0}}, \min\overline{U_{s_1}}), (\max\overline{U_{s_1}}, \min\overline{U_{s_2}}), \ldots$

$(\max\overline{U_{s_{k-1}}}, \min\overline{U_{s_k}}), (\min\overline{U_{s_0}}, \max\overline{U_{s_k}}) \in \mathcal{S}$

\end{center}

\noindent Since $\bigcup_{s \in \mathcal{I}}\overline{U_s} = I$ and by the arguments above, we know that each element of $\mathcal{I}$ occurs in a unique such finite sequence.  Also for such a $D$ we have $P(s_0) = \cdots = P(s_k)$, and for each $0 \leq j < k$ we have $U_{s_j} \equiv U_{s_{j+1}}^{-1}$ and $A(s_j) = -A(s_{j+1})$, and also $U_{s_0} \equiv U_{s_k}^{-1}$.  We take ordered pairs $(s_0; s_1), \ldots, (s_{k-1}; s_k)$ and let $\mathcal{P}$ be the set of all such ordered pairs for all such sequences $D$.  We observe that

\begin{itemize}

\item $(s; s') \in \mathcal{P}$ implies $s < s'$ in $\mathbb{Q}$;

\item $(s; s') \in \mathcal{P}$ implies $U_{s} \equiv U_{s'}^{-1}$;

\item $(s; s') \in \mathcal{P}$ implies $W_s \equiv W_{s'}^{-1}$;

\item $(\forall s'' \in \mathcal{I})(\exists ! (s; s') \in \mathcal{P}) s'' = s \vee s'' = s'$;

\item for $(s; s'), (s''; s''') \in \mathcal{P}$ such that the intervals $(s, s'), (s'', s''') \subseteq \mathcal{I}$ have nonempty intersection, we have that $(s, s') \subseteq (s'', s''')$ or $(s'', s''') \subseteq (s, s')$.

\end{itemize}

Now we see that the nonempty subword $\prod_{s \in \mathcal{I}} W_s$ of $W$ is $\sim E$, for we can define a reduction scheme $\mathcal{S}''$ on $\prod_{s \in \mathcal{I}} W_s$ by taking for each $(s; s') \in \mathcal{P}$ an order reversing $f_{(s; s')}: \overline{U_s} \rightarrow \overline{U_{s'}}$ such that $U_s(i) = (U_{s'}(f_{(s; s')}(i)))^{-1}$ and letting

\begin{center}

$\mathcal{S}'' = \bigcup_{(s; s') \in \mathcal{P}} \bigcup_{i \in \overline{U_s}} \{(i, f_{(s; s')}(i))\}$.

\end{center}

\noindent Thus $W$ is not reduced, contrary to assumption, a contradiction.

\end{proof}

We next turn our attention to another important fact.

\begin{lemma}\label{notsofast}  If $I \subseteq \overline{U}$ is an interval such that $U \upharpoonright I \in \Fine(\{U_x\}_{x \in X} \cup \{U_m'\}_{m \in \omega})$ then there exists some $s \in \mathbb{Q}$ for which $I \subseteq \overline{U_s}$.
\end{lemma}

\begin{proof}  If the hypothesis holds but the conclusion fails, then since $\mathbb{Q}$ is order dense there exists an interval $I' \subseteq I$ and infinite interval $\mathcal{I}' \subseteq \mathbb{Q}$ such that $I' = \prod_{s \in \mathcal{I}'} \overline{U_s}$.  Since $U \upharpoonright I \in \Fine(\{U_x\}_{x \in X} \cup \{U_m'\}_{m \in \omega})$, we also have $U \upharpoonright I' \in \Fine(\{U_x\}_{x \in X} \cup \{U_m'\}_{m \in \omega})$.  By Lemma \ref{nicefine} we can write $U \upharpoonright I'$ as a finite concatenation

\begin{center}

$U \upharpoonright I' \equiv V_0V_1\cdots V_k$

\end{center}

\noindent where for $0 \leq r \leq k$ we have $|\overline{V_r}| =1$ or 

\begin{center}

$V_r \in \Sub(\{U_x\}_{x \in X} \cup \{U_x^{-1}\}_{x \in X} \cup \{U_m'\}_{m \in \omega} \cup \{(U_m')^{-1}\}_{m \in \omega})$.

\end{center}

\noindent Then there exists $0 \leq l \leq k$ and nonempty subinterval $\mathcal{I}'' \subseteq \mathcal{I}'$, having rational endpoints $\mathcal{I}'' = (s, s')$, such that $\overline{V_l} \supseteq I' \supseteq \overline{\prod_{s \in \mathcal{I}''} U_s}$, so in particular $$\prod_{s \in \mathcal{I}''} U_s \in \Sub(\{U_x\}_{x \in X} \cup \{U_x^{-1}\}_{x \in X} \cup \{U_m'\}_{m \in \omega} \cup \{(U_m')^{-1}\}_{m \in \omega}).$$  However letting $(s, s') = J_j$ we have contradicted our selection criterion for the function $q_j$.
\end{proof}

Now we define the coi from $W$ to $U$ and verify the various properties.  For each $s \in \mathbb{Q}$ we let $L_s: \overline{h_{P(s)}^{A(s)R(P(s))}U_{P(s)}^{A(s)}h_{P(s)}^{A(s)R(P(s))}} \rightarrow \overline{U_s}$ witness $$h_{P(s)}^{A(s)R(P(s))}U_{P(s)}^{A(s)}h_{P(s)}^{A(s)R(P(s))} \equiv U_s$$ and let $L_s': \overline{W_{P(s)}^{A(s)}}\rightarrow I_s$ witness $W_{P(s)}^{A(s)} \equiv W \upharpoonright I_s$.  For each $s \in \mathbb{Q}$ define function $\iota_s$ by having $\dom(\iota_s) = L_s'(\dom(\iota_{P(s)})) \subseteq I_s$, $\im(\iota_s) = L_s(\im(\iota_{P(s)}))  \subseteq \overline{U_s} \setminus \{\min(\overline{U_s}), \max(\overline{U_s})\}$ and $\iota_s(i) = L_s \circ \iota_{P(s)}\circ (L_s')^{-1}(i)$.  Thus $\iota_s$ is an order-preserving bijection (if $A(s) = -1$ then the definition of $\iota_s$ is a composition of three functions, the first and the last are order-reversing and the middle is order-preserving).  It is also clear that $\Close(\dom(\iota_s),  I_s)$ and $\Close(\im(\iota_s), \overline{U_s})$.  Define a function $\iota$ by letting $\iota = \bigcup_{s \in \mathbb{Q}} \iota_s$.  We have that $\dom(\iota) = \bigcup_{s \in \mathbb{Q}} \dom(\iota_s)$ and $\im(\iota) = \bigcup_{s \in \mathbb{Q}}\im(\iota_s)$, and $\iota$ is an order isomorphism between its domain and range.  Moreover we have $\Close(\dom(\iota), \overline{W})$ and $\Close(\im(\iota), \overline{U})$ by Lemma \ref{basiccloseproperties} (iii).  Therefore $\iota$ is a coi from $W$ to $U$.

\begin{lemma}\label{itworksforQ}  The collection $$\{\coi(W_x, \iota_x, U_x)\}_{x \in X} \cup \{\coi(W_m, \iota_m, U_m')\}_{m \in \omega} \cup \{\coi(W, \iota, U)\}$$ is coherent.  Thus, more particularly $\{\coi(W_x, \iota_x, U_x)\}_{x \in X} \cup \{\coi(W, \iota, U)\}$ is coherent.
\end{lemma}

\begin{proof}  Suppose that we have an $x \in X \cup \omega$, intervals $I \subseteq \overline{W}$ and $I' \subseteq \overline{W_x}$ and $\delta \in \{-1, 1\}$ such that $W \upharpoonright I \equiv (W_x \upharpoonright I')^{\delta}$.  We know that $W_x \in \Fine(\{W_x\}_{x \in X} \cup \{W_n\}_{n \in \omega}) = \Fine(\{W_x\}_{x \in X})$, and therefore also we know that $(W_x \upharpoonright I')^{\delta} \in \Fine(\{W_x\}_{x \in X})$.  By our assumptions on the word $W$ we therefore have some $s \in \mathbb{Q}$ such that $I \subseteq I_s$.  Let $f: I_s \rightarrow \overline{(W_{P(s)})^{A(s)}}$ witness that $W \upharpoonright I_s \equiv W_{P(s)}^{A(s)}$.  By the coherence of $$\{\coi(W_x, \iota_x, U_x)\}_{x \in X} \cup \{\coi(W_m, \iota_m, U_m')\}_{m \in \omega}$$ we have that

$$
\begin{array}{ll}
[[U \upharpoonright \alpha(I, \iota)]] & = [[U_s \upharpoonright \alpha(I, \iota_s)]]\\
& = [[(U_{P(s)} \upharpoonright \alpha(f(I), \iota_{P(s)}))^{A(s)}]]\\
& = [[(U_x \upharpoonright \alpha(I', \iota_x))^{\delta}]]

\end{array}
$$

\noindent where the first equality holds because $U \upharpoonright \alpha(I, \iota) \equiv U_s \upharpoonright \alpha(I, \iota_s)$, the second holds since $U_s \upharpoonright \alpha(I, \iota_s) \equiv (U_{P(s)} \upharpoonright \alpha(f(I), \iota_{P(s)}))^{A(s)}$ (because of how the function $\iota_s$ is defined), and the third equality holds because the collection $$\{\coi(W_x, \iota_x, U_x)\}_{x \in X} \cup \{\coi(W_m, \iota_m, U_m')\}_{m \in \omega}$$ is coherent.

Next, we suppose that $I, I' \subseteq \overline{W}$ are intervals and $\delta \in \{-1, 1\}$ is such that $W \upharpoonright I \equiv (W \upharpoonright I')^{\delta}$.  We'll only consider the case where $\delta = -1$, as the other case is much more straightforward.  Let $f: I \rightarrow I'$ be an order-reversing bijection witnessing that $W \upharpoonright I \equiv (W \upharpoonright I')^{-1}$, so that $W(i) = (W(f(i)))^{-1}$.  Notice that $f$ takes subintervals of $I$ to subintervals of $I'$, while reversing the order of the points.  However, for a subinterval $I'' \subseteq I$ we consider $f(I'') \subseteq I'$ to have the order inherited from $I'$ and from $\overline{W}$, so we may make sense of the expression $W \upharpoonright f(I'')$.  Thus for each subinterval $I'' \subseteq I$ we have $W \upharpoonright I'' \equiv (W \upharpoonright f(I''))^{-1}$.  We notice that for each subinterval $I'' \subseteq I$ we have $W \upharpoonright I'' \in \Fine(\{W_x\}_{x \in X})$ if and only if $W \upharpoonright f(I'')\in \Fine(\{W_x\}_{x \in X})$ if and only if $(W \upharpoonright f(I''))^{-1} \in \Fine(\{W_x\}_{x \in X})$.

Let $\mathcal{I}$ be the interval in $\mathbb{Q}$ defined by $\mathcal{I} = \{s \in \mathbb{Q}: I_s \cap I \neq \emptyset\}$ and similarly define another interval $\mathcal{I}' = \{s \in \mathbb{Q}: I_s \cap I' \neq \emptyset\}$.  By our assumptions on the intervals $I_s$, we know that a subinterval $I'' \subseteq I$ is a subinterval of one of the $I_s$ if and only if $W \upharpoonright I'' \in \Fine(\{W_x\}_{x \in X})$, and so $f(I'')$ is a subinterval of one of the $I_s$ if and only if $W \upharpoonright I'' \in \Fine(\{W_x\}_{x \in X})$ if and only if $W \upharpoonright f(I'') \in \Fine(\{W_x\}_{x \in X})$.  Thus we have an order-reversing bijection $F: \mathcal{I} \rightarrow \mathcal{I}'$ given by $F(s) = s'$ where $f(I \cap I_s) \cap I_{s'} \neq \emptyset$.  Therefore $\mathcal{I}$ has a maximum if and only if $\mathcal{I}'$ has a minimum, and $\mathcal{I}$ has a minimum if and only if $\mathcal{I}'$ has a maximum.

Now we consider various cases.  If $\mathcal{I} = \emptyset = \mathcal{I}'$ then both $I$ and $I'$ are empty, and we have $[[U \upharpoonright\alpha(I, \iota)]] = [[E]] = [[(U \upharpoonright\alpha(I', \iota))^{-1}]]$.  If both $\mathcal{I}$ and $\mathcal{I}'$ are of cardinality $1$ then we let $\{s\} = \mathcal{I}$ and $\{s'\} = \mathcal{I}'$ and $L: I_s \rightarrow \overline{W_{(P(s))}^{A(s)}}$ and $L': I_{s'} \rightarrow \overline{W_{(P(s'))}^{A(s')}}$ witness that $W \upharpoonright I_s \equiv W_{(P(s))}^{A(s)}$ and $W\upharpoonright I_{s'} \equiv W_{(P(s'))}^{A(s')}$.  We have

$$
\begin{array}{ll}
[[U \upharpoonright \alpha(I, \iota)]] & = [[U_s \upharpoonright \alpha(I, \iota_s)]]\\
& = [[(U_{P(s)}' \upharpoonright \alpha(L(I), \iota_{P(s)}))^{A(s)}]]\\
& = [[(U_{P(s')}' \upharpoonright \alpha(L'(I'), \iota_{P(s')}))^{-A(s')}]]\\
& = [[(U \upharpoonright \alpha(I', \iota_{s'}))^{-1}]]\\
& = [[(U \upharpoonright \alpha(I', \iota))^{-1}]]

\end{array}
$$

\noindent where the first and last equalities hold by the fact that $\iota \upharpoonright (\dom(\iota) \cap \overline{U_s''}) = \iota_{s''}$ for all $s'' \in \mathbb{Q}$, the second and fourth equalities hold by how the functions $\iota_s$ and $\iota_{s'}$ are defined, and the third equality holds by the fact that the subcollection $\{\coi(W_m, \iota_m, U_m')\}_{m \in \omega}$ is coherent.

If both $\mathcal{I}$ and $\mathcal{I}'$ have at least two points then they are infinite (since $\mathbb{Q}$ is order-dense).  We'll imagine that $\mathcal{I}$ contains a maximum (and so $\mathcal{I}'$ contains a minimum) and that $\mathcal{I}$ contains a minimum (and so $\mathcal{I}'$ contains a maximum), and in case a maximum or a minimum does not exist then the modifications are obvious.  Let $s_0 = \min\mathcal{I}$, $s_1 = \max\mathcal{I}$ and $(\mathcal{I})^* = \mathcal{I} \setminus \{s_0, s_1\}$ and $(\mathcal{I}')^* = \mathcal{I}' \setminus \{F(s_1), F(s_0)\}$.  We have that $W \upharpoonright I_s \equiv (W \upharpoonright I_{F(s)})^{-1}$ for each $s \in (\mathcal{I})^*$, as witnessed by $f$, but the equivalence may fail for $s = s_0$ and $s = s_1$.  Let

\begin{itemize}
\item $I_0 = I \cap I_{s_0}$

\item $I_1 = I \cap I_{s_1}$

\item $I_2 = \alpha(I, \iota) \cap \overline{U_{s_0}}$

\item $I_3 = \alpha(I, \iota) \cap \overline{U_{s_1}}$

\item $I_4 = I' \cap I_{F(s_1)}$

\item $I_5 = I' \cap I_{F(s_0)}$

\item $I_6 = \alpha(I', \iota) \cap \overline{U_{F(s_1)}}$

\item $I_7 = \alpha(I', \iota) \cap \overline{U_{F(s_0)}}$.

\end{itemize}

\noindent We notice that $\alpha(I_0, \iota)$ is obtained from $I_2$ by deleting a finite terminal subinterval, and possibly also a finite initial interval, so in particular $[[U_{s_0} \upharpoonright I_2]] = [[U_{s_0} \upharpoonright \alpha(I_0, \iota)]]$, and by similar reasoning we may write $[[U_{s_1} \upharpoonright I_3]] = [[U_{s_1} \upharpoonright \alpha(I_1, \iota)]]$, $[[U_{F(s_1)} \upharpoonright I_6]] = [[U_{F(s_1)} \upharpoonright \alpha(I_4, \iota)]]$, $[[U_{F(s_0)}\upharpoonright I_7]]=[[U_{F(s_0)}\upharpoonright \alpha(I_5, \iota)]]$.  Moreover it is the case that $$[[U_{s_0} \upharpoonright \alpha(I_0, \iota)]] = [[(U_{F(s_0)} \upharpoonright \alpha(I_5, \iota))^{-1}]]$$ and $$[[U_{s_1}\upharpoonright \alpha(I_1, \iota)]] = [[(U_{F(s_1)}\upharpoonright \alpha(I_4, \iota))^{-1}]]$$ since $\{\coi(W_m, \iota_m, U_m')\}_{m \in \omega}$ is coherent.  Thus we have

$$
\begin{array}{ll}
[[U\upharpoonright \alpha(I, \iota)]] & = [[(U_{s_0} \upharpoonright I_2)(\prod_{s \in (\mathcal{I})^*}U_s)(U_{s_1}\upharpoonright I_3)]]\\
& = [[U_{s_0} \upharpoonright I_2]][[\prod_{s \in (\mathcal{I})^*}U_s]][[U_{s_1}\upharpoonright I_3]]\\
& = [[U_{s_0} \upharpoonright \alpha(I_0, \iota)]][[\prod_{s \in (\mathcal{I})^*}U_s]][[U_{s_1}\upharpoonright \alpha(I_1, \iota)]]\\
& = [[U_{s_0} \upharpoonright \alpha(I_0, \iota)]][[\prod_{s \in (\mathcal{I})^*}U_{F(s)}^{-1}]][[U_{s_1}\upharpoonright \alpha(I_1, \iota)]]\\
& = [[(U_{F(s_0)} \upharpoonright \alpha(I_5, \iota))^{-1}]][[(\prod_{s \in (\mathcal{I}')^*}U_s)^{-1}]][[(U_{F(s_1)}\upharpoonright \alpha(I_4, \iota))^{-1}]]\\
& = [[(U_{F(s_0)} \upharpoonright I_7)^{-1}]][[(\prod_{s \in (\mathcal{I}')^*}U_s)^{-1}]][[(U_{F(s_1)}\upharpoonright I_6)^{-1}]]\\
& = [[(U_{F(s_0)} \upharpoonright I_7)^{-1}(\prod_{s \in (\mathcal{I}')^*}U_s)^{-1}(U_{F(s_1)}\upharpoonright I_6)^{-1}]]\\
& = [[((U_{F(s_1)}\upharpoonright I_6)(\prod_{s \in (\mathcal{I}')^*}U_s)(U_{F(s_0)} \upharpoonright I_7))^{-1}]]\\
& = [[(U \upharpoonright \alpha(I', \iota))^{-1}]]

\end{array}
$$

\noindent The cases where $\min\mathcal{I}$ and/or $\max\mathcal{I}$ are considered using obvious modifications, and that where $\delta = 1$ is even more straightforward.

Next we suppose $z \in X \cup \omega$ and $I \subseteq \overline{U}$, $I' \subseteq \overline{U_z}$ and $\epsilon \in \{-1, 1\}$ are such that $U \upharpoonright I \equiv (U_z \upharpoonright I')^{\epsilon}$.  We must show that $[[W \alpha(I, \iota^{-1})]] = [[(W_z \alpha(I', \iota^{-1}))^{\epsilon}]]$.  We know by Lemma \ref{notsofast} that there exists $s \in \mathbb{Q}$ such that $I \subseteq \overline{U_s}$.  Define $$I'' \subseteq \overline{h_{P(s)}^{A(s)R(P(s))}U_{P(s)}^{A(s)}h_{P(s)}^{A(s)R(P(s))}}$$ by $I'' = L_s^{-1}(I)$  By how $\iota$ is defined, and as $I \subseteq \overline{U_s}$ we have

$$
\begin{array}{ll}
W\upharpoonright \alpha(I, \iota^{-1}) & \equiv W\upharpoonright \alpha(I \cap \overline{U_s}, \iota^{-1})\\
& \equiv W_{P(s)}^{A(s)} \upharpoonright \alpha(I'', \iota_{P(s)}^{-1})

\end{array}
$$

\noindent and as $\{\coi(W_x, \iota_x, U_x)\}_{x \in X} \cup \{\coi(W_n, \iota_m, U_m)\}_{m \in \omega}$ is coherent we have $[[W_{P(s)}^{A(s)} \upharpoonright \alpha(I'', \iota_{P(s)}^{-1})]] = [[(W_z \upharpoonright \alpha(I', \iota_z^{-1}))^{\epsilon}]]$, so that in fact $[[W \upharpoonright \alpha(I, \iota^{-1})]] = [[(W \upharpoonright \alpha(I', \iota_z^{-1}))^{\epsilon}]]$ as required.

Finally suppose that intervals $I, I' \subseteq \overline{U}$ and $\epsilon \in \{-1, 1\}$ are such that $U \upharpoonright I \equiv (U \upharpoonright I')^{\epsilon}$.  We prove the difficult case where $\epsilon = -1$, and the case $\epsilon = 1$ is left to the reader.  Take $f: I \rightarrow I'$ to be an order-reversing bijection witnessing $U \upharpoonright I \equiv (U \upharpoonright I')^{\epsilon}$, so $U(i) = (U(f(i)))^{-1}$.  Such an $f$ takes subintervals of $I$ to subintervals of $I'$.  As before, for a subinterval $I'' \subseteq I$ we have $U \upharpoonright I'' \in \Fine(\{U_x\}_{x \in X} \cup \{U_m\}_{m \in \omega})$ if and only if $U \upharpoonright f(I'') \in \Fine(\{U_x\}_{x \in X} \cup \{U_m\}_{m \in \omega})$.

Let $\mathcal{I}$ be the interval in $\mathbb{Q}$ defined by $\mathcal{I} = \{s \in \mathbb{Q} \mid \overline{U_s} \cap I \neq \emptyset\}$ and similarly $\mathcal{I}' = \{s \in \mathbb{Q} \mid \overline{U_s} \cap I' \neq \emptyset\}$.  Using Lemma \ref{notsofast} one can argue as before that $f$ induces an order-reversing bijection $F: \mathcal{I} \rightarrow \mathcal{I}'$ (where $F(s) = s'$ means that $f(I \cap \overline{U_s}) \cap \overline{U_{s'}} \neq \emptyset$).  If both $\mathcal{I}$ and $\mathcal{I}'$ are of cardinality $1$ then $I \subseteq \overline{U_s}$ and $I' \subseteq \overline{U_{F(s)}}$ for some $s \in \mathbb{Q}$ and $[[W \upharpoonright \alpha(I, \iota^{-1})]] = [[(W \upharpoonright \alpha(I', \iota^{-1}))^{\epsilon}]]$ simply because $\{\coi(W_x, \iota_x, U_x)\}_{x \in X} \cup \{\coi(W_m, \iota_m, U_m)\}_{m \in \omega}$ is coherent.  In case $\mathcal{I}$ and $\mathcal{I}'$ are both empty we have $I = \emptyset = I'$ and $[[W \upharpoonright \alpha(I, \iota^{-1})]] = [[E]] = [[(W \upharpoonright \alpha(I', \iota^{-1}))^{\epsilon}]]$.  If both $\mathcal{I}$ and $\mathcal{I}'$ have at least $2$ points then they are infinite.  Defining $(\mathcal{I})^*$ to be $\mathcal{I}$ minus any minimum or maximum, and define $(\mathcal{I}')^*$ similarly.  Then $U_s \equiv (U_{F(s)})^{-1}$ for $s \in (\mathcal{I})^*$ and $U \upharpoonright (\overline{U_{\min(\mathcal{I})}} \cap I) \equiv (U \upharpoonright (\overline{U_{\max(\mathcal{I}')}} \cap I'))^{\epsilon}$ (provided $\min(\mathcal{I})$ exists) and $U \upharpoonright (\overline{U_{\max(\mathcal{I})}} \cap I) \equiv (U \upharpoonright (\overline{U_{\min(\mathcal{I}')}} \cap I'))^{\epsilon}$ (provided $\max(\mathcal{I})$ exists).

If, for example, both $\max(\mathcal{I})$ and $\min(\mathcal{I})$ exist we get that $[[W \upharpoonright \alpha(\overline{U_{\min(\mathcal{I})}} \cap I, \iota^{-1})]] = [[(W \upharpoonright \alpha(\overline{U_{\max(\mathcal{I}')}} \cap I', \iota^{-1}))^{\epsilon}]]$ and $[[W \upharpoonright \alpha(\overline{U_{\max(\mathcal{I})}} \cap I, \iota^{-1})]] = [[(W \upharpoonright \alpha(\overline{U_{\min(\mathcal{I}')}} \cap I', \iota^{-1}))^{\epsilon}]]$ since $\{\coi(W_x, \iota_x, U_x)\}_{x \in X} \cup \{\coi(W_m, \iota_m, U_m)\}_{m \in \omega}$ is coherent.  Moreover $W \upharpoonright I_s \equiv (W \upharpoonright I_{F(s)})^{-1}$ for each $s \in (\mathcal{I})^*$ by how $U$ was constructed, and so

$$
\begin{array}{ll}
[[W \upharpoonright \alpha(I, \iota^{-1})]] & = [[W \upharpoonright \alpha(\overline{U_{\min(\mathcal{I})}} \cap I, \iota^{-1})]] \\

& \cdot [[W \upharpoonright \alpha(\bigcup_{s \in (\mathcal{I})^*}\overline{U_s}, \iota^{-1})]][[W \upharpoonright \alpha(\overline{U_{\max(\mathcal{I})}} \cap I, \iota^{-1})]]\\
& = [[(W \upharpoonright \alpha(\overline{U_{\max(\mathcal{I}')}} \cap I', \iota^{-1}))^{\epsilon}]][[(W \upharpoonright \alpha(\bigcup_{s \in (\mathcal{I}')^*}\overline{U_s}, \iota^{-1}))^{\epsilon}]]\\
& \cdot[[(W \upharpoonright \alpha(\overline{U_{\min(\mathcal{I}')}} \cap I', \iota^{-1}))^{\epsilon}]]\\
& = [[(W \upharpoonright \alpha(I', \iota^{-1}))^{\epsilon}]]

\end{array}
$$

\noindent where the first and last equalities are by Lemmas \ref{coilemma} and \ref{coilemma2}.  If $\max(\mathcal{I})$ or $\min(\mathcal{I})$ do not exist then the modifications are obvious.  The proof of this lemma is finished.

\end{proof}

Claim (1) is now seen to be true from Lemma \ref{itworksforQ}.  The proof of claim (2) is totally analogous, and so the proof of the proposition is complete.

\end{section}

\begin{section}{Conclusion of the proof}\label{Finally}

We are now armed to give the finishing arguments of the proof.  We begin with the following.

\begin{proposition}\label{generalextension}  Let $\{G_n\}_{n \in \omega}$ and $\{K_n\}_{n \in \omega}$ be sequences of groups, each having an element of infinite order and no elements of order $2$.  Suppose that $\{\coi(W_x, \iota, U_x)\}_{x \in X}$ is a coherent collection of coi from $\Red(\{G_n\}_{n \omega})$ to $\Red(\{K_n\}_{n \omega})$ such that $|X| < 2^{\aleph_0}$.

\begin{enumerate}

\item If $W \in \Red(\{G_n\}_{n \omega})$ then there exists a $U \in \Red(\{K_n\}_{n \omega})$ and coi $\iota$ from $W$ to $U$ such that $\{\coi(W_x, \iota, U_x)\}_{x \in X} \cup \{\coi(W, \iota, U)\}$ is coherent.

\item If $U \in \Red(\{K_n\}_{n \omega})$ then there exists a $W \in \Red(\{G_n\}_{n \omega})$ and coi $\iota$ from $W$ to $U$ such that $\{\coi(W_x, \iota, U_x)\}_{x \in X} \cup \{\coi(W, \iota, U)\}$ is coherent.

\end{enumerate}

\end{proposition}

\begin{proof}  As usual, we'll only prove (1).  If $W \equiv E$ then we let $U \equiv E$ and $\iota$ be the empty function, and it is obvious that $\{\coi(W_x, \iota, U_x)\}_{x \in X} \cup \{\coi(W, \iota, U)\}$ is coherent.  If $W \not\equiv E$ then we begin by extending the original collection by letting $\iota_i$ be the empty function and noticing that $\{\coi(W_x, \iota, U_x)\}_{x \in X} \cup \{\coi(W\upharpoonright \{i\}, \iota_i, E)\}_{i \in \overline{W}}$ is coherent.  Let $\mathcal{T}_0 = \{\coi(W_x, \iota, U_x)\}_{x \in X} \cup \{\coi(W\upharpoonright \{i\}, \iota_i, E)\}_{i \in \overline{W}}$.  Notice that $|\mathcal{T}_0| < 2^{\aleph_0}$, since $|X \cup \overline{W}| < 2^{\aleph_0}$.  For any collection $\mathcal{T} = \{\coi(W_z, \iota_z, U_z)\}_{z \in Z}$ of coi triples we let $h(\mathcal{T}) = \{W_z\}_{z \in Z}$.  Let $\prec$ be a well-order on the countable set $\overline{W}$.  We detail a procedure which will we will carry through until termination.

Suppose that we have defined coherent $\mathcal{T}_{\beta}$ for all $\beta < \gamma < \aleph_1$, where $\gamma > 0$, so that $\mathcal{T}_{\beta_0} \subseteq \mathcal{T}_{\beta_1}$ when $\beta_0 \leq \beta_1 < \gamma$, and $|\mathcal{T}_{\zeta} \setminus \bigcup_{\beta < \zeta}\mathcal{T}_{\beta}| \leq 1$ for all $0 < \zeta < \gamma$ .  We have that $\bigcup_{\beta < \gamma} \mathcal{T}_{\beta}$ is coherent by Lemma \ref{ascendingchaincoi}.  Also, $|\bigcup_{\beta < \gamma} \mathcal{T}_{\beta}| \leq |X|\cdot\aleph_0 + |\gamma| < 2^{\aleph_0}$.  Notice also that for every $i \in \overline{W}$ there exists an interval $i \in I \subseteq W$ such that $W \upharpoonright I \in \Fine(\bigcup_{\beta < \gamma} \mathcal{T}_{\beta})$, as indeed one can simply take $I = \{i\}$, since $\mathcal{T}_0 \subseteq \bigcup_{\beta < \gamma} \mathcal{T}_{\beta}$.

\begin{enumerate}[(a)]

\item If it is the case that for every $i \in \overline{W}$ there exists an interval $i \in I \subseteq \overline{W}$ such that $W \upharpoonright I \in \Fine(h(\bigcup_{\beta < \gamma} \mathcal{T}_{\beta}))$ and for any larger interval $I \subsetneq I' \subseteq \overline{W}$ we get $W \upharpoonright I' \notin \Fine(h(\bigcup_{\beta < \gamma} \mathcal{T}_{\beta}))$, then we terminate the procedure and let $\mathcal{T}_{\gamma} = \bigcup_{\beta < \gamma} \mathcal{T}_{\beta}$.

\item If (a) does not hold, then select $i \in \overline{W}$, minimal under $\prec$, such that such a maximal interval $I$ does not exist.  Suppose further that it is the case that there exist intervals $\{I_m\}_{m \in \omega}$ such that $i = \min I_m$ for all $m \in \omega$, with $W \upharpoonright I_m \in \Fine(h(\bigcup_{\beta < \gamma} \mathcal{T}_{\beta}))$ for all $m \in \omega$, $I_m \subsetneq I_{m+1}$, and with $W \upharpoonright \bigcup_{m \in \omega} I_m \notin \Fine(h(\bigcup_{\beta < \gamma} \mathcal{T}_{\beta}))$.  Then by Proposition \ref{omegatypeconcat} there exists a $U \in \Red(\{K_n\}_{n \in \omega})$ and coi $\iota$ from $W\upharpoonright \bigcup_{m \in \omega} I_m$ to $U$ such that $\mathcal{T}_{\beta} = \bigcup_{\beta < \gamma} \mathcal{T}_{\beta} \cup \{\coi(W \upharpoonright \bigcup_{m \in \omega} I_m, \iota, U)\}$ is coherent.  If such sequence $\{I_m\}_{m \in \omega}$ as above does not exist, then there exist intervals $\{I_m\}_{m \in \omega}$ such that $i = \max I_m$ for all $m \in \omega$, with $W \upharpoonright I_m \in \Fine(h(\bigcup_{\beta < \gamma} \mathcal{T}_{\beta}))$ for all $m \in \omega$, $I_m \subsetneq I_{m+1}$, and with $W \upharpoonright \bigcup_{m \in \omega} I_m \notin \Fine(h(\bigcup_{\beta < \gamma} \mathcal{T}_{\beta}))$.  We apply Proposition \ref{omegatypeconcat} to $(W \upharpoonright \bigcup_{m \in \omega} I_m)^{-1}$ to obtain $\iota$ and $U$ such that $\mathcal{T}_{\gamma} = \bigcup_{\beta < \gamma} \mathcal{T}_{\beta} \cup \{(\coi(W \upharpoonright \bigcup_{m \in \omega} I_m)^{-1}, \iota, U)\}$ is coherent.
\end{enumerate}

We claim that the process terminates after countably many steps.  Supposing that this is not the case, the process is carried out for all $0 < \gamma < \aleph_1$.  Then there is some $i \in \overline{W}$ which is considered in Case (b) uncountably often, and without loss of generality it is uncountably often the case that $i$ is the minimum of all elements of the sequence of intervals $\{I_m\}_{m \in \omega}$ considered.  Let $T \subseteq \aleph_1$ be the set of those $\gamma < \aleph_1$ on which this holds, so for every $\gamma \in T$ we have $\mathcal{T}_{\gamma} \setminus \bigcup_{\beta < \gamma} \mathcal{T}_{\beta} = \{\coi(W \upharpoonright I, \iota, U)\}$ for some $I$ (we'll label $I = I_{\gamma}$ as it is uniquely determined by $\gamma$) and $U$, where $i = \min I_{\gamma}$.  But now $I_{\gamma} \subsetneq I_{\gamma'}$ when $\gamma < \gamma'$ and $\gamma, \gamma' \in T$, and this is impossible for intervals in a countable totally ordered set $\overline{W}$, a contradiction.

This process terminates and produces $\mathcal{T}_{\gamma} \supseteq \mathcal{T}_0$, with $\gamma < \aleph_1$ and $|\mathcal{T}_{\gamma}| < 2^{\aleph_0}$, such that for every $i \in \overline{W}$ there exists an interval $i \in I \subseteq \overline{W}$ such that $W \upharpoonright I \in \Fine(h(\mathcal{T}_{\gamma}))$ and there is not a larger interval $I \subsetneq I' \subseteq \overline{W}$ with $W \upharpoonright I' \in \Fine(h(\mathcal{T}_{\gamma}))$.  Take $\{I_{\lambda}\}_{\lambda \in \Lambda}$ to be the set of all intervals in $\overline{W}$ such that for any $\lambda \in \Lambda$ and we have $W\upharpoonright I_{\lambda} \in \Fine(h(\mathcal{T}_{\gamma}))$ and for any larger interval $I_{\lambda} \subsetneq I' \subseteq \overline{W}$ we have $W \upharpoonright I' \notin \Fine(h(\mathcal{T}_{\gamma}))$.  We make the indexing injective, so that $I_{\lambda} \neq I_{\lambda'}$ whenever $\lambda \neq \lambda'$.  All $I_{\lambda}$ are nonempty since otherwise we could simply select $i \in \overline{W}$ and note that $W\upharpoonright \{i\} \in \Fine(h(\mathcal{T}_0)) \subseteq \Fine(h(\mathcal{T}_{\gamma}))$.  Also we know that the $I_{\lambda}$ are pairwise disjoint, for if $I_{\lambda} \cap I_{\lambda'} \neq \emptyset$ for $\lambda \neq \lambda'$ we have $W \upharpoonright I_{\lambda} \cup I_{\lambda'} \in \Fine(h(\mathcal{T}_{\gamma}))$.  Thus we endow $\Lambda$ with the natural order which places $\lambda < \lambda'$ if all elements of $I_{\lambda}$ are below all elements in $I_{\lambda'}$.  Note that $\Lambda$ cannot have two elements $\lambda < \lambda'$ which are immediately adjacent, for then $W \upharpoonright I_{\lambda} \cup I_{\lambda'} \in \Fine(h(\mathcal{T}_{\gamma}))$.

Certainly $\Lambda$ has at least one element as $W \not\equiv E$.  If $\Lambda$ has exactly one element then by Lemma \ref{findsomerepresentative} we can select $U \in \Red(\{K_n\}_{n \in \omega})$ and coi $\iota$ such that $\mathcal{T}_{\gamma} \cup \{\coi(W, \iota, U)\}$ is coherent, so in particular $\{\coi(W_x, \iota_x, U_x)\}_{x \in X} \cup \{\coi(W, \iota, U)\}$ is coherent and we are done.  On the other hand if $\Lambda$ has at least two elements then $\Lambda$ is countably infinite and dense-in-itself since there are no adjacent elements.  Let $\Lambda^*$ be subset of $\Lambda$ obtained by removing $\max\Lambda$ and $\min\Lambda$, if either or both exist.  Then $\Lambda^*$ is order isomorphic to $\mathbb{Q}$.  We'll assume that $\min\Lambda$ and $\max\Lambda$ each exist, and the modifications to the proof in the other cases are obvious.  It is clear that for any interval $\Lambda' \subseteq \Lambda$ where $\Lambda'$ has at least two points, we have $W \upharpoonright \bigcup_{\lambda \in \Lambda'} I_{\lambda} \notin \Fine(h(\mathcal{T}_{\gamma}))$ (by the conditions defining the $I_{\lambda}$).  Therefore by Proposition \ref{Qtypeconcat} we can select a $V \in \Red(\{K_n\}_{n \in \omega})$ and coi $\iota_*$ such that $\mathcal{T}_{\beta} \cup \{\coi(W\upharpoonright \bigcup_{\lambda \in \Lambda^*} I_{\lambda}, \iota_*, V)\}$ is coherent.  Then

$$W \equiv (W\upharpoonright I_{\min\Lambda})(W\upharpoonright \bigcup_{\lambda \in \Lambda^*} I_{\lambda})(W\upharpoonright I_{\max\Lambda}) \in \Fine(h(\mathcal{T}_{\gamma} \cup \{W\upharpoonright \bigcup_{\lambda \in \Lambda^*} I_{\lambda}\}))$$

\noindent so by Lemma \ref{findsomerepresentative} we can select $U \in \Red(\{K_n\}_{n \in \omega})$ and coi $\iota$ such that $$\mathcal{T}_{\gamma} \cup \{\coi(W\upharpoonright \bigcup_{\lambda \in \Lambda^*} I_{\lambda}, \iota_*, V)\} \cup \{\coi(W, \iota, U)\}$$ is coherent, so in particular $\{\coi(W_x, \iota_x, U_x)\}_{x \in X} \cup \{\coi(W, \iota, U)\}$ is coherent.

\end{proof}

\begin{proof}[Proof of Main Theorem]  We let $\{H_n\}_{n \in \omega}$ be a sequence of groups without elements of order $2$ such that $1 < |H_n| \leq 2^{\aleph_0}$ for each $n \in \omega$ .  We let $G_n = H_{2n} * H_{2n+1}$ for each $n \in \omega$.  Now each $G_n$ is a group without elements of order $2$ and has an element of infinite order (take $h \in H_{2n}\setminus \{1\}$ and $h' \in H_{2n+1} \setminus \{1\}$ and we have $hh'$ of infinite order).  Furthermore, $1 < |G_n| \leq 2^{\aleph_0}$.  By Lemma \ref{nicefactsaboutarch} (3) we know that $\mathcal{A}(\{G_n\}_{n \in \omega}) \simeq \mathcal{A}(\{H_n\}_{n \in \omega})$.  Thus it will be sufficient to prove that $\mathcal{A}(\{G_n\}_{n \in \omega}) \simeq \mathcal{A}$.  By definition we have $\mathcal{A} = \mathcal{A}(\{K_n\}_{n \in \omega})$ where each $K_n$ is infinite cyclic.

We claim that $|\Red(\{G_n\}_{n \in \omega})| = 2^{\aleph_0}$.  To see this, we note that $\Red(\{G_n\}_{n \in \omega}) \simeq \topprod_{n\in \omega} G_n$ has as quotient the group $\mathcal{A}(\{G_n\}_{n \in \omega})$, and as this latter group is of cardinality $2^{\aleph_0}$ (see \cite[Theorem 9]{CHM}), so $|\Red(\{G_n\}_{n \in \omega})| \geq 2^{\aleph_0}$.  On the other hand, let $\Omega$ be a symbol such that $\Omega \notin \bigcup_{n \in \omega} G_n$ and notice that the set $\mathcal{J}$ of functions from $\mathbb{Q}$ to $(\bigcup_{n \in \omega} G_n) \cup\{\Omega\}$ has $|\mathcal{J}| = 2^{\aleph_0}$, as $\aleph_0 \leq |\bigcup_{n \in \omega} G_n| \leq 2^{\aleph_0}$.  For a word $\overline{W} \in \W(\{G_n\}_{n \in \omega})$ we pick an order embedding $f_W: \overline{W} \rightarrow \mathbb{Q}$, and we let $F(W) \in \mathcal{J}$ be given by

\[
F(W)(s) = \left\{
\begin{array}{ll}
W(i)
                                            & \text{if } s = f_W(i), \\
\Omega                                & \text{if } s\notin f_W(\overline{W}).
\end{array}
\right.
\]

\noindent This function $F$ is easily seen to be an injection, so $$|\Red(\{G_n\}_{n \in \omega})| \leq |\W(\{G_n\}_{n \in \omega})| \leq 2^{\aleph_0}$$ and $|\Red(\{G_n\}_{n \in \omega})| = 2^{\aleph_0}$, and by the same argument $|\Red(\{K_n\}_{n \in \omega})| = 2^{\aleph_0}$.

We let $\prec_G$ be a well order on $\Red(\{G_n\}_{n \in \omega})$ such that every element has fewer than $2^{\aleph_0}$ elements below it.  Similarly let $\prec_K$ be a well order on $\Red(\{K_n\}_{n \in \omega})$ such that every element has fewer than $2^{\aleph_0}$ elements below it.  Each ordinal $\gamma$ can be written as a sum $\gamma = \zeta + n$ where $\zeta$ is $0$ or a limit ordinal and $n \in \omega$, and so we consider an ordinal even or odd according to the parity of $n$.  We inductively define a sequence of length $2^{\aleph_0}$ (considering the cardinal $2^{\aleph_0}$ now as an ordinal) of coi triples.  Let $W_0 \in \Red(\{G_n\}_{n \in \omega})$ be minimal under $\prec_G$ and by Proposition \ref{generalextension} select $U_0 \in \Red(\{K_n\}_{n \in \omega})$ and $\iota_0$ so that $\{\coi(W_0, \iota_0, U_0)\}$ is coherent.  Suppose that we have produced coi triples $\coi(W_{\beta}, \iota_{\beta}, U_{\beta})$ for all $\beta < \gamma < 2^{\aleph_0}$ so that $\{\coi(W_{\beta}, \iota_{\beta}, U_{\beta})\}_{\beta \leq \zeta}$ is coherent for each $\zeta < \gamma$.  We know (by Lemma \ref{ascendingchaincoi} in case $\gamma$ is a limit) that $\{\coi(W_{\beta}, \iota_{\beta}, U_{\beta})\}_{\beta < \gamma}$ is coherent.  If $\gamma$ is even then select by Lemma \ref{notdoneyet} $W_{\gamma} \in \Red(\{G_n\}_{n \in \omega}) \setminus \Fine(\{G_n\}_{n \in \omega})$, with $W_{\gamma}$ minimal such under the well-order $\prec_G$, and by Proposition \ref{generalextension} select $U_{\gamma} \in \Red(\{K_n\}_{n \in \omega})$ and $\iota_{\gamma}$ so that $\{\coi(W_{\beta}, \iota_{\beta}, U_{\beta})\}_{\beta < \gamma} \cup \{\coi(W_{\gamma}, \iota_{\gamma}, U_{\gamma})\}$ is coherent.  If $\gamma$ is odd then select by Lemma \ref{notdoneyet} $U_{\gamma} \in \Red(\{K_n\}_{n \in \omega}) \setminus \Fine(\{K_n\}_{n \in \omega})$, with $U_{\gamma}$ minimal such under the well-order $\prec_K$, and by Lemma \ref{generalextension} select $W_{\gamma} \in \Red(\{K_n\}_{n \in \omega})$ and $\iota_{\gamma}$ so that $\{\coi(W_{\beta}, \iota_{\beta}, U_{\beta})\}_{\beta < \gamma} \cup \{\coi(W_{\gamma}, \iota_{\gamma}, U_{\gamma})\}$ is coherent.  The collection $\{\coi(W_{\gamma}, \iota_{\gamma}, U_{\gamma})\}_{\gamma < 2^{\aleph_0}}$ is coherent by Lemma \ref{ascendingchaincoi}, and it is also clear that $\Fine(\{W_{\gamma}\}_{\gamma < 2^{\aleph_0}}) = \Red(\{G_n\}_{n \in \omega})$ and also $\Fine(\{U_{\gamma}\}_{\gamma < 2^{\aleph_0}}) = \Red(\{K_n\}_{n \in \omega})$.  Therefore by Theorem \ref{coicollectiongivesiso} we obtain an isomorphism $\mathcal{A}(\{G_n\}_{n \in \omega}) \simeq \mathcal{A}(\{K_n\}_{n \in \omega})$ and the proof is complete.
\end{proof}

We finish by illustrating some of the difficulty which arises with elements of order $2$.

\begin{example}\label{nastyword}  Suppose that $\{G_n\}_{n \in \omega}$ and $\{K_n\}_{n \in \omega}$ are collections of groups such that each $G_n$ has an element of order $2$ and none of the $K_n$ has such an element.  Select $g_n \in G_n$ of order $2$ for each $n \in \omega$.  Consider the word $W$ given by the projections

\begin{center}

$p_0(W) \equiv g_0$

$p_1(W) \equiv g_1 g_0 g_1$

$p_2(W) \equiv g_2 g_1 g_2 g_0 g_2 g_1 g_2$

$p_3(W) \equiv g_3 g_2 g_3 g_1 g_3 g_2 g_3 g_0 g_3 g_2 g_3 g_1 g_3 g_2 g_3$

$\vdots$

\end{center}

\noindent It is evident that $W$ is reduced and that $W \equiv W^{-1}$.  Moreover we can write $W$ as $W \equiv W_1g_0W_1$ where $W_1$ is defined by

\begin{center}

$p_0(W_1) \equiv E$

$p_1(W_1) \equiv g_1$

$p_2(W_1) \equiv g_2 g_1 g_2$

$p_3(W_1) \equiv g_3 g_2 g_3 g_1 g_3 g_2 g_3$

$\vdots$

\end{center}

\noindent and similarly $W_1 \equiv W_1^{-1}$, and $W_1 \equiv W_2g_1W_2$ where

\begin{center}

$p_0(W_2) \equiv E$

$p_1(W_2) \equiv E$

$p_2(W_2) \equiv g_2$

$p_3(W_2) \equiv g_3 g_2 g_3$

$\vdots$

\end{center}

\noindent and so forth.  Each of these words $W, W_1, W_2, \ldots$ is of order type $\mathbb{Q}$, and $[[E]] = [[W]] = [[W_1]] = [[W_2]] = \cdots$ since they are each conjugate in $\Red(\{G_n\}_{n \in \omega})$ to a finite word.  Although these words are highly symmetric and trivial in $\mathcal{A}(\{G_n\}_{n \in \omega})$, there are subwords which are not, as for example

$$g_0W_3g_2W_5g_4W_7g_6 \cdots.$$

\noindent  It is not obvious how to select $U \in \Red(\{K_n\}_{n \in \omega})$ and coi $\iota$ so that $\{\coi(W, \iota, U)\}$ is coherent; the curious reader can look at how Proposition \ref{Qtypeconcat} uses the fact that the groups have no order $2$ elements.

\end{example}

\end{section}

\section*{Appendix}

In this appendix we shall state and prove a fact which is much more general than Lemma \ref{Qwordreduced}.  The purpose is twofold.  First, the reader can verify Lemma \ref{Qwordreduced} in a setting where there are fewer distracting symbols.  Second, the more general statement could be used towards an attack of a theorem involving involutions.

We begin with a word $V: \mathbb{Q} \rightarrow \{a_n^{\pm 1}\}_{n \in \omega}$ which is in $\Red(\{\langle a_n \rangle_{\infty}\}_{n \in \omega})$, where $\langle a_n \rangle_{\infty}$ denotes the infinite cyclic group generated by the symbol $a_n$.  Notice such a $V$ is a very special word- none of the outputs of the word $V$ are, for example, $a_2^7$.  The only group elements in the range of $V$ are of form $a_n$ or $a_n^{-1}$.  We take a sequence of groups $\{K_n\}_{n \in \omega}$, each of which has an element of infinite order.  We shall not assume that any of the $K_n$ are free of involutions.  We let sequence $\{U_m\}_{m \in \omega}$ of words in $\Red(\{K_n\}_{n \in \omega})$ have the following characteristics:  both $\min\overline{U_m}$ and $\max\overline{U_m}$ exist, and $U_m(\max\overline{U_m}) = U_m(\min\overline{U_m}) = b_m$, where $b_m$ is of infinite order in $K_m$.  Also, $d(U_m \upharpoonright (\overline{U_m} \setminus \{\min\overline{U_m}, \max\overline{U_m}\})) > m = d(b_m)$.  Thus it could be the case, for example, that $\overline{U_m} \setminus \{\min\overline{U_m}, \max\overline{U_m}\} = \emptyset$.  Define $U \equiv \prod_{s \in \mathbb{Q}} T_s$ where

\[
T_s \equiv \left\{
\begin{array}{ll}
U_m
                                            & \text{if }  V(s) \equiv x_m, \\
U_m^{-1}                                        & \text{if }V(s) \equiv x_m^{-1}.
\end{array}
\right.
\]

\noindent Thus the  $U$ is obtained from the word $V$ by replacing each instance of $x_m^{\pm 1}$ with $U_m^{\pm 1}$.   

\begin{theorem}\label{intheappendix}  The function $U$ defined above is an element in $\Red(\{K_n\}_{n \in \omega})$.
\end{theorem}

\begin{proof}  First, it is clear that $U$ is indeed a word since $d(U_m) = m$ and for each $n \in \omega$ the set $\{s \in \mathbb{Q} \mid d(T_s) \leq n\} = \{s \in \mathbb{Q} \mid d(V(s)) \leq n\}$ is finite (because $V$ is a word).  Define $P:\mathbb{Q} \rightarrow \omega$ by $T_s \equiv U_{P(s)}$ or $T_s \equiv U_{P(s)}^{-1}$.  Since each word $U_m$ begins and ends with the same letter, and the letter is not an involution, the definition of $P$ is unambiguous.

We suppose for contradiction that $U$ is not reduced.  Since each subword $T_s$ is reduced and nonempty and $\mathbb{Q}$ is order dense, by Proposition \ref{reductionscheme} (1) we have a nonempty reduction scheme $\mathcal{S}$ on $U \upharpoonright I$ such that $\bigcup_{C \in \mathcal{S}} \set(C) = I$ and $\pi(U \upharpoonright I, C) \equiv E$ for all $C \in \mathcal{S}$.  Our strategy will be to modify the scheme $\mathcal{S}$ into a new, cleaner reduction scheme in which for each $s \in \mathbb{Q}$ such that elements in $\overline{T_s}$ participate in the new scheme there exists an $s' \in \mathbb{Q}$ such that $T_s \equiv T_{s'}^{-1}$ and the reduction pairs off exactly one element of $\overline{T_s}$ with exactly one element of $\overline{T_{s'}}$.  Thus the scheme will naturally lift to a reduction on the word $V$, which gives a contradiction.

Of course, we will need to rely on our hypotheses in order to obtain such a nice reduction scheme.  Specifically, the fact that each word $U_m$ is reduced and begins and ends with a ``high wall'' makes it so that if $T_s \equiv U_m^{\pm 1}$ is a subword having an element which participates in the reduction, and $m$ is minimal for which such an $s$ exists, then the beginning and/or the ending letter of $T_s$ must also participate in the reduction as well (by the second point in the definition of a reduction scheme) and can only be in a component with first or last letters of a similar such $T_{s'}$ (by the minimality of $m$).  This allows us to march through the reduction scheme and make appropriate adjustments towards our goal.

To begin our attack, we take $$N_0' = \min\{m \in \omega:(\exists s \in \mathbb{Q}) \overline{T_s} \cap I \neq \emptyset \wedge T_s \equiv U_m\text{ or }T_s \equiv U_m^{-1}\}.$$  We claim that there exists a component $C = (i_0; \ldots; i_k) \in \mathcal{S}$ such that $d(C) = N_0'$ and for each $0 \leq j \leq k$ we have $U(i_j) \in \{b_{N_0'}^{\pm 1}\}$.  To see this, take $s' \in \mathbb{Q}$ such that $\overline{T_{s'}} \cap I \neq \emptyset$ and $T_{s'} \equiv U_{N_0'}$ or $\equiv U_{N_0'}^{-1}$.  Take $C' = (i_0'; \ldots; i_m') \in \mathcal{S}$ such that $\set(C') \cap \overline{T_{s'}} \neq \emptyset$, say $0 \leq j' \leq m$ has $i_{j'} \in \set(C') \cap \overline{T_{s'}}$.  If $d(C') = N_0'$ we will take $C = C'$, so suppose that this is not the case.  As $\pi(U \upharpoonright I, C) = E$ and the word $U$ is simple, we know $m > 0$.  Let without loss of generality $j' + 1 \leq m$ (otherwise $0 \leq j' - 1$ and the proof will be similar).  As $T_{s'}$ is reduced, we know $i_{j' + 1} \notin \overline{T_{s'}}$, say $i_{j' + 1} \in \overline{T_{s''}}$ and $s' < s''$ in $\mathbb{Q}$.  As $\mathcal{S}$ is a reduction scheme and $i_{j'} < \max(\overline{T_{s'}}) < i_{j' + 1}$ there exists $C = (i_0; \ldots; i_k) \in \mathcal{S}$ such that $\max(\overline{T_{s'}}) = i_j \in \set(C)$.  As $T_{s'} \equiv U_{N_0'}$ or $\equiv U_{N_0'}^{-1}$ we know $U(i_j) \in \{b_{N_0'}^{\pm 1}\}$ so in particular $d(C) = N_0'$.  In either case we have found a component $C \in \mathcal{S}$ with $d(C) = N_0'$ and $\set(C) \cap \overline{T_{s'}} \neq \emptyset$.  Letting $J = \{s \in \mathbb{Q} \mid \set(C) \cap \overline{T_s} \neq \emptyset\}$, by minimality of $N_0'$ we know that $T_s \equiv U_{N_0'}$ or $\equiv U_{N_0'}^{-1}$ for each $s \in J$.  Then as $d(C) = N_0'$, each element of $\set(C)$ is either a $\max(\overline{T_s})$ or a $\min(\overline{T_s})$ for some $s \in J$.  Therefore $U(i_j) \in \{b_{N_0'}^{\pm 1}\}$ for each $i_j \in \set(C)$.

As $\pi(U, C) \equiv E$ we have that there exist some $0 \leq j \leq k$ for which $U(i_j) = b_{N_0'}$ and there also exist some $0 \leq j \leq k$ for which $U(i_j) = b_{N_0'}^{-1}$ (we are using the fact that $b_{N_0'}$ has infinite order).  Then there exists some $0 \leq j' \leq k$ for which $U(i_{j'}) = (U(i_{j'+1}))^{-1}$.  Then the reduction scheme $\{C' \in \mathcal{S}: \set(C') \cap (i_{j'}, i_{j'+1}) \neq \emptyset\}$ witnesses that $U \upharpoonright (i_{j'}, i_{j'+1}) \sim E$.  Letting $i_{j'} \in \overline{T_{s_0'}}$ and $i_{j'+1} \in \overline{T_{s_1'}}$ we have $i_{j'} \in \{\min\overline{T_{s_0'}}, \max\overline{T_{s_0'}}\}$ and similarly $i_{j'+1} \in \{\min\overline{T_{s_1'}}, \max\overline{T_{s_1'}}\}$.  If $i_{j'} = \min\overline{T_{s_0'}}$ then we claim that $i_{j'+1} = \max\overline{T_{s_1'}}$, for otherwise we have $(i_{j'}, i_{j'+1}) \cap \{i \in \overline{U}: d(U(i)) = N_0'\}$ is odd and so the word $p_{N_0'}(U \upharpoonright (i_{j'}, i_{j'+1}))$ does not represent the trivial element in $G_{N_0'}$, since $b_{N_0'}$ is of infinite order, a contradiction.  By the same token, if $i_{j'} = \max\overline{T_{s_0'}}$ then $i_{j'+1} = \min\overline{T_{s_1'}}$.  In either case, we see that $U \upharpoonright (\max\overline{T_{s_0'}}, \min\overline{T_{s_1'}}) \sim E$, for even if $i_{j'} = \min\overline{T_{s_0'}}$ and $i_{j'+1} = \max\overline{T_{s_1'}}$ we have that

\begin{center}

$E \sim (T_{s_0'}\upharpoonright (\overline{T_{s_0'}} \setminus \{\min\overline{T_{s_0'}}\}))^{-1} E (T_{s_0'}\upharpoonright (\overline{T_{s_0'}} \setminus \{\min\overline{T_{s_0'}}\}))$

$\equiv (T_{s_0'}\upharpoonright \overline{T_{s_0'}} \setminus \{\min\overline{T_{s_0'}}\})^{-1} E (T_{s_1'}\upharpoonright( \overline{T_{s_1'}} \setminus \{\max\overline{T_{s_1'}}\}))^{-1}$

$\sim  (T_{s_0'}\upharpoonright \overline{T_{s_0'}} \setminus \{\min\overline{T_{s_0'}}\})^{-1} (U\upharpoonright (i_{j'}, i_{j'+1})) (T_{s_1'}\upharpoonright( \overline{T_{s_1'}} \setminus \{\max\overline{T_{s_1'}}\}))^{-1}$

$\sim U\upharpoonright (\max\overline{T_{s_0'}}, \min\overline{T_{s_1'}})$.

\end{center}

\noindent  Thus we may replace the interval $I$ with the nonempty interval $$(\max\overline{T_{s_0'}}, \min\overline{T_{s_1'}})$$ and thus get that $I$ is an open interval such that $I = \bigcup_{I \cap \overline{T_s} \neq \emptyset} \overline{T_s}$, and also replace the old reduction scheme $\mathcal{S}$ with $\{C' \in \mathcal{S}: \set(C) \cap (\max\overline{T_{s_0'}}, \min\overline{T_{s_1'}}) \neq \emptyset\}$.  Thus in our proof we will therefore assume that $I = \bigcup_{I \cap \overline{T_s} \neq \emptyset} \overline{T_s}$ and that $\mathcal{S}$ is a reduction scheme on $U\upharpoonright I$ such that $\pi(U \upharpoonright I, C) \equiv E$ for all $C \in \mathcal{S}$ and $\bigcup_{C\in \mathcal{S}}\set(C) = I$.  Let $\mathcal{I} \subseteq \mathbb{Q}$ be the open interval $\{s \in \mathbb{Q}: I \cap \overline{T_s} \neq \emptyset\}$.  We let $Q: I \rightarrow \mathcal{I} \subseteq \mathbb{Q}$ be the surjective function defined by $Q(i)= s$ where $i \in \overline{T_s}$.

Let $\{N_0, N_1, \ldots\} = \{P(s):(\exists s \in \mathbb{Q}) \overline{T_s} \cap I \neq \emptyset\}$, with $N_k < N_{k+1}$, so in particular $N_0 = \min\{P(s):(\exists s \in \mathbb{Q}) \overline{T_s} \cap I \neq \emptyset\}$.  Suppose first that $C = (i_0; \ldots; i_k) \in \mathcal{S}$ and $s_0, s_1 \in \mathbb{Q}$ are such that $\overline{T_{s_0}} \cap \set(C) \neq \emptyset$ and $\overline{T_{s_1}} \cap \set(C) \neq \emptyset$ and $P(s_0) = N_0$.  We'll show that $P(s_1) = N_0$.  If this is not the case, there exist $i_l, i_{l+1} \in \set(C)$ with $i_l \in \overline{T_{s_0'}}$ and $i_{l+1} \in \overline{T_{s_1'}}$ such that either $P(s_0') = N_0$ and $P(s_1') > N_0$, or such that $P(s_0') > N_0$ and $P(s_1') = N_0$.  If without loss of generality $P(s_0') = N_0$ and $P(s_1') > N_0$, we know by Definition \ref{redschdef} condition (2) and Proposition \ref{reductionscheme} part (1) that $U \upharpoonright (i_l, i_{l+1}) \sim E$, however $(i_l, i_{l+1}) \cap \{i \in \overline{U}: d(U(i)) = N_0\}$ is odd and so the word $p_{N_0}(U \upharpoonright (i_l, i_{l+1}))$ does not represent the trivial element in $G_{N_0}$, as $b_{N_0}$ is of infinite order, a contradiction.  Now, suppose that it is the case that whenever $k \leq K$ and $C \in \mathcal{S}$ and $s_0, s_1 \in \mathbb{Q}$ are such that $\overline{T_{s_0}} \cap \set(C) \neq \emptyset$ and $\overline{T_{s_1}} \cap\set(C) \neq \emptyset$ and $P(s_0) = N_k$, then $P(s_1) = N_k$.  Let $C = (i_0; \ldots; i_p) \in \mathcal{S}$ and $s_0, s_1 \in \mathbb{Q}$ be such that $\overline{T_{s_0}} \cap \set(C) \neq \emptyset$ and $\overline{T_{s_1}} \cap\set(C) \neq \emptyset$ and $P(s_0) = N_{K+1}$.  We'll show $P(s_1) = N_{K+1}$.  If this is not the case then there exist $i_l, i_{l+1} \in \set(C)$ with $i_l \in \overline{T_{s_0'}}$ and $i_{l+1} \in \overline{T_{s_1'}}$ such that either $P(s_0') = N_{K+1}$ and $P(s_1') > N_{K+1}$, or such that $P(s_0') > N_{K+1}$ and $P(s_1') = N_{K+1}$.  Without loss of generality $P(s_0') = N_{K+1}$ and $P(s_1') > N_{K+1}$.  Letting $Y$ be the finite set $\{s \in \mathbb{Q}: P(s) \leq K \wedge \overline{T_s} \subseteq (i_l, i_{l+1})\}$ we have $U \upharpoonright ((i_0, i_1) \setminus \bigcup_{s \in Y} \overline{T_s}) \sim E$ as witnessed by the reduction scheme $\mathcal{S}' = \{C' \in \mathcal{S}: \set(C') \cap (i_l, i_{l+1}) \neq \emptyset \wedge \set(C) \cap \bigcup_{s \in Y}\overline{T_s} = \emptyset\}$ (here we are using the fact that our induction hypothesis implies that $\bigcup_{s \in Y}\overline{U_s} = \bigcup_{C' \in \mathcal{S}, \set(C') \cap \bigcup_{s \in Y}\overline{T_s} \neq \emptyset} \set(C')$).  However the set $M = (i_0, i_1) \cap \{i \in \overline{U}: d((U(i))) = N_{K+1}\} \setminus \bigcup_{s \in Y}\overline{T_s}$ is of odd cardinality, and  we have $U(i) \in \{b_{N_{K+1}}, b_{N_{K+1}}^{-1}\}$ for each $i \in M$, and since $b_{N_{K+1}}$ is of infinite order we get in particular that $p_{N_{K+1}}(U \upharpoonright (i_0, i_1) \setminus \bigcup_{s \in Y} \overline{T_s}) = \sum_{i \in M} U(i)$ is not trivial, contradiction.  What we have just shown is that for each $C \in \mathcal{S}$ the function $P\circ Q \upharpoonright \set(C)$ is constant.  We also know that for each $C \in \mathcal{S}$ the function $Q \upharpoonright \set(C)$ is injective, since each $T_s$ is reduced.

Now we proceed with the modifications to the scheme $\mathcal{S}$.  For $C = (i_0; \ldots; i_k) \in \mathcal{S}$ such that there exists $i \in \set(C)$ with $i \in \overline{T_s}$ and $d(C) = P(s)$, we have $i \in \{\min\overline{T_s}, \max\overline{T_s}\}$ and by the preceeding paragraph we have that each $i_j \in \set(C)$ has some $s_j \in \mathcal{I}$ with $i_j \in \{\min\overline{T_{s_j}}, \max\overline{T_{s_j}}\}$ and $d(C) = P(s_j)$.  More particularly we have that $U(i_j) \in \{b_{d(C)}, b_{d(C)}^{-1}\}$ for all $0 \leq j \leq k$.  Since $\pi(U, C) = 0$ we know $U(i_j) = b_{d(C)}$ for some $j$ and $U(i_j) = b_{d(C)}^{-1}$ for some other values of $j$.  Then for some $0\leq j' \leq k$ we have $U(i_{j'}) = (U(i_{j' +1}))^{-1}$, and we can replace $C$ in $\mathcal{S}$ with two components $(i_{j'}; i_{j' +1})$, $C' = (i_{0}; \ldots; i_{j'-1}; i_{j' + 2}; \ldots; i_{k})$.  If $|\set(i_{0}, \ldots, i_{j'-1}, i_{j' + 2}, \ldots, i_{k})| > 2$ then performing the same analysis on the finite sequence $(i_{0}; \ldots; i_{j'-1}; i_{j' + 2}; \ldots; i_{k})$ we produce two components $C', C''$ with $|\set(C')| = 2$ and $|\set(C'')|$ being of positive even cardinality.  By performing finitely many steps we determine that we can replace $C$ with $|\set(C)|/2$ components.  Thus we can assume that for each $s \in \mathcal{I}$ we have that $\min \overline{T_s} \in \set(C)$ and $C \in \mathcal{S}$ implies that $|\set(C)| = 2$ and similarly for $\max \overline{T_s}$.

Now, if $s \in \mathcal{I}$, $i = \min\overline{T_s}$, $i \in \set(C)$, $\{i'\} = \set(C) \setminus \{i\}$, with $i' \in \overline{T_{s'}}$ then $i' = \max\overline{T_{s'}}$.  To see this, we know of course that $i' \in \{\min\overline{T_{s'}}, \max\overline{T_{s'}}\}$, and for contradiction if $i' = \min\overline{T_{s'}}$ and say $i < i'$ then the word $V \equiv U \upharpoonright ((i, i') \setminus\bigcup_{s \in \mathcal{I}, P(s) < d(C)} \overline{T_s})$ is not $\sim E$ since $p_{d(C)}(V)$ is $b_{d(C)}$ raised to an odd power, on the other hand $p_{d(C)}(V) \sim E$ by Proposition \ref{reductionscheme} part (1) (using the reduction scheme $\{C' \in \mathcal{S}: \set(C') \cap (i, i') \neq \emptyset \wedge d(C') < d(C)\}$), contradiction.  A similar proof works when $i' < i$.  By similar reasoning if $s \in \mathcal{I}$, $i = \max\overline{T_s}$, $i \in \set(C)$, $\{i'\} = \set(C) \setminus \{i\}$, with $i' \in \overline{T_{s'}}$ then $i' = \min\overline{T_{s'}}$.

We define a collection $\mathcal{P}$ of ordered pairs of elements of $\mathcal{I}$.  Consider a finite sequence $D = (s_0; \ldots; s_k)$ such that 

\begin{center}

$(\max\overline{T_{s_0}}, \min\overline{T_{s_1}}), (\max\overline{T_{s_1}}, \min\overline{T_{s_2}}), \ldots$

$(\max\overline{T_{s_{k-1}}}, \min\overline{T_{s_k}}), (\min\overline{T_{s_0}}, \max\overline{T_{s_k}}) \in \mathcal{S}$

\end{center}

\noindent Since $\bigcup_{s \in \mathcal{I}}\overline{T_s} = I$ and by the arguments above, we know that each element of $\mathcal{I}$ occurs in a unique such finite sequence.  Also for such a $D$ we have $P(s_0) = \cdots = P(s_k)$, and for each $0 \leq j < k$ we have $T_{s_j} \equiv U_{s_{j+1}}^{-1}$ and $T_{s_j} = T_{s_{j + 1}}^{-1}$, and also $T_{s_0} \equiv T_{s_k}^{-1}$.  We take ordered pairs $(s_0; s_1), \ldots, (s_{k-1}; s_k)$ and let $\mathcal{P}$ be the set of all such ordered pairs for all such sequences $D$.  We observe that

\begin{itemize}

\item $(s; s') \in \mathcal{P}$ implies $s < s'$ in $\mathbb{Q}$;

\item $(s; s') \in \mathcal{P}$ implies $T_{s} \equiv T_{s'}^{-1}$;

\item $(s; s') \in \mathcal{P}$ implies $V(s) \equiv V(s')^{-1}$;

\item $(\forall s'' \in \mathcal{I})(\exists ! (s; s') \in \mathcal{P}) s'' = s \vee s'' = s'$;

\item for $(s; s'), (s''; s''') \in \mathcal{P}$ such that the intervals $(s, s'), (s'', s''') \subseteq \mathcal{I}$ have nonempty intersection, we have that $(s, s') \subseteq (s'', s''')$ or $(s'', s''') \subseteq (s, s')$.

\end{itemize}

Now we see that the nonempty subword $\prod_{s \in \mathcal{I}} V(s)$ of $V$ is $\sim E$ via the reduction scheme $\mathcal{P}$.  Thus $W$ is not reduced, contrary to assumption, a contradiction.

\end{proof}

One can derive Lemma \ref{Qwordreduced} as a corollary of Theorem \ref{intheappendix} by defining the word $V: \mathbb{Q} \rightarrow \{x_n^{\pm 1}\}_{n \in \omega}$ by $V(s) = x_m$ if $W \upharpoonright I_s \equiv W_m$ and $V(s) = x_m^{-1}$ if $W \upharpoonright I_s \equiv W_m^{-1}$.  This definition of $V$ is unambiguous since in the setting of Lemma \ref{Qwordreduced} all groups are without involutions.  The word $V$ is reduced, as a reduction on $V$ can be easily translated into a reduction on the reduced word $W$.  Now, the word $U$ of Lemma \ref{Qwordreduced} is obtained precisely as in the setting of Theorem \ref{intheappendix}, so it is reduced.

\section*{Acknowledgement}

The author thanks an earlier referee for pointing out some clarifications towards improving the paper.

\end{document}